\pdfoutput=1
\documentclass[english]{jnsao}
\usepackage[utf8]{inputenc}
\usepackage{enumerate} 
\usepackage{array}
\usepackage{subcaption}
\usepackage[pdftex]{graphicx}

\usepackage[nameinlink,capitalize]{cleveref}

\newtheorem{Algrthm}[theorem]{Algorithm}

\newenvironment{alg}[1]{\renewcommand{\theenumi}{(S.\arabic{enumi})} 
\begin{Algrthm} #1 \begin{enumerate} \setcounter{enumi}{-1}}{\renewcommand{\theenumi}{\roman{enumi}}\end{enumerate}\end{Algrthm}}

\allowdisplaybreaks

\newcommand{\R}{\mathbb{R}} 
\newcommand{\N}{\mathbb{N}}
\newcommand{\interior}{\operatorname{int}}
\newcommand{\ri}{\operatorname{ri}}

\newcommand{\conv}{\operatorname{conv}}

\let\originalleft\left
\let\originalright\right
\renewcommand{\left}{\mathopen{}\mathclose\bgroup\originalleft}
\renewcommand{\right}{\aftergroup\egroup\originalright}

\DeclareMathOperator*{\argmin}{argmin}
\DeclareMathOperator*{\Limsup}{Lim \, sup}
\DeclareMathOperator*{\dom}{dom}

\renewcommand{\theenumi}{(\roman{enumi})}
\renewcommand{\labelenumi}{\theenumi}

\newcounter{enumv}[enumiv]

\makeatletter

\renewcommand*{\enumerate}{%
    \ifnum \@enumdepth > 5\relax
    \@toodeep
    \else
    \advance\@enumdepth\@ne
    \edef\@enumctr{enum\romannumeral\the\@enumdepth}%
    \expandafter\list\csname label\@enumctr\endcsname{%
        \usecounter\@enumctr
        \def\makelabel##1{\hss\llap{##1}}%
    }%
    \fi
}
\makeatother

\title{An adaptive proximal safeguarded augmented Lagrangian method for nonsmooth DC~problems with convex constraints}
\shorttitle{An ALM for constrained DC~problems}

\author{Christian Kanzow\thanks{Institute of Mathematics, University of Würzburg, Emil-Fischer-Straße 30, 97074 Würzburg, Germany, \email{christian.kanzow@uni-wuerzburg.de}, \orcid{0000-0003-2897-2509}}
    \and Tanja Neder\thanks{Institute of Mathematics, University of Würzburg, Emil-Fischer-Straße 30, 97074 Würzburg, Germany, \email{tanja.neder@uni-wuerzburg.de}, \orcid{0000-0002-9669-1197}}%
}
\shortauthor{Kanzow, Neder}

\manuscriptlicense{CC-BY 4.0}

\manuscriptsubmitted{2025-05-23}
\manuscriptaccepted{2026-03-23}
\manuscriptvolume{6}
\manuscriptnumber{15731}
\manuscriptyear{2026}
\manuscriptdoi{10.46298/jnsao-2026-15731}

\begin{document}

\maketitle

\begin{abstract}
    A proximal safeguarded augmented Lagrangian method for minimizing the difference of convex~(DC) functions over a nonempty, closed and convex set with additional linear equality as well as convex inequality constraints is presented. Thereby, all functions involved may be nonsmooth. Iterates (of the primal variable) are obtained by solving convex optimization problems as the concave part of the objective function gets approximated by an affine linearization. Under the assumption of a modified Slater constraint qualification, both convergence of the primal and dual variables to a generalized Karush-Kuhn-Tucker (KKT) point is proven, at least on a subsequence. Numerical experiments and comparison with existing solution methods are presented using some classes of constrained and nonsmooth DC problems.
\end{abstract}


\section{Introduction}

We consider a class of nonsmooth DC~problems with convex constraints of the following type:
\begin{equation} \label{eq: problemFormulation}
    \min_{x \in \R^n} f(x):=g(x)-h(x) \quad \text{s.t.} \quad Ax=b, \ c(x) \leq 0, \ x \in C \tag{P}
\end{equation}
with $g,h: \R^n \rightarrow \R$ convex, $A \in \R^{p \times n}$ (with $p\leq n$), $b \in \R^p$, $c_i: \R^n \rightarrow \R$ convex for $i=1,\dots, m$ and $\emptyset \neq C \subseteq \R^n$ closed and convex, where for a vector $y \in \R^m$ the relation $y \leq 0$ here as well as in what follows is meant componentwise, that is
\[y \leq 0 \quad :\Longleftrightarrow \quad y_i \leq 0 \ \forall i=1,\dots,m.\] 
As the objective function $f$ has a representation as the difference of two convex functions, it is called a \emph{DC~function}, whereas $g$ and $h$ are referred to as $DC~components$. (These DC~components are, of course, not unique.) 

The increasing interest in DC~optimization is due to the broad field of real world applications (see e.g.~\cite{Tuy2016,LeThi2018} and references therein for exemplary overviews). To put the focus on problems matching the above setting, let us mention questions emerging in logistics, in particular location planning~\cite{Tuy2016,Chen1998} as well as production-transportation problems~\cite{Holmberg1999}, in chance-constrained management problems~\cite{deOliveira2019}, in designing communication systems with physical layer based security~\cite{Pang2017}, in compressed sensing~\cite{Yin2015}, in the packaging industry~\cite{Horst1999} and in machine learning, for example within feature selection in the context of support vector machines~\cite{LeThi2008}.

Many methods for the solution of constrained DC~programs impose additional assumptions on the DC~objective function. In particular, it is frequently assumed that one of the DC~components has to be continuously differentiable (see e.g.~\cite{AragonArtacho2022} or DCA2 in~\cite{LeThi2014}) or even Lipschitz-differentiable (see e.g. LCDC-ALM in \cite{Sun2022}), or that the second DC~component has to be the pointwise maximum of finitely many continuously differentiable functions~\cite{Pang2017,Pang2018,Lu2022}. Extension of these approaches to general nonsmooth DC~objective functions seems to be difficult. 

Hence, we focus on constrained DC~problems where the objective function is given as the difference of two arbitrary nonsmooth convex functions. To the best of our knowledge, there are only a few other approaches dealing with this general setting. 
\cite{deOliveira2019}~presents a proximal bundle method for minimizing a nonsmooth  DC~function over a general closed and convex set. Thereby, the key idea is not only to approximate the second DC~component by a linearization, but also the first one by a cutting-plane model. No penalization of the constraints is carried out. 
An approach for minimizing nonsmooth DC~functions under linear equality constraints can be found in~\cite{Sun2022} by neglecting the smooth part of the first DC~component while considering Composite LCDC-ALM. Then, further restrictions are implicitly embedded as the first DC~component of the objective function is assumed to be proper, lower semicontinuous and convex with compact domain over which it is itself supposed to be Lipschitz-continuous. The crucial point of the method is, to replace once more in each iteration the concave part by a linearization before then considering the Moreau envelope of the resulting augmented Lagrangian function.
In~\mbox{\cite{vanAckooij2019,vanAckooij2019b}} the authors introduce a method for minimizing a nonsmooth DC~function over a closed and convex set which gets further restricted by (without loss of generality one single) DC-type inequality. Adopting the idea of the classical DC~Algorithm (DCA,~\cite{LeThi1996}), one linearizes in each iteration the second DC~component in both, the objective function and the constraint. The resulting subproblems though need not to be solved exactly. Instead, computing a point of the modified feasible set which provides a sufficient decrease in the model function is adequate. However, a feasible point is required for initialization in order to avoid infeasibility of the subproblems which might otherwise occur due to relaxing the DC~type constraint. In case one omits the latter constraint, the method boils down to (a proximal version of) DCA~\cite{LeThi1996} but with the modification of allowing for inexact solutions of the subproblems. 
Also in~\cite{LeThi2014}, DC~programs with DC~constraints that get complemented by a closed and convex set are considered. Referring to DCA1, penalizing the DC~constraints suitably carries over the DC~structure to the resulting penalty problems. This again allows to apply the strategy of the classical DCA, that is to  approximate within each subproblem the concave part of the objective function by an affine majorization. In addition, the penalty parameter gets updated in each iteration. Although the problem setting at the beginning is kept of a general nature, the convergence theory of DCA1 requires, in particular, differentiability of one of the DC~components of the penalized objective function.
Dealing with constrained DC~problems, one can, of course, always apply the classical DCA~\cite{LeThi1996} itself by formally rewriting the problem as an unconstrained one, which, however, is equivalent to carrying over the constraints into the subproblems.

Notwithstanding the method, keeping the constraints in the subproblems might make them hard to solve, depending on the application.
Hence, our aim has been to develop an algorithm for constrained DC~optimization which first, is able to handle a general nonsmooth DC~objective function and second, allows for augmentation of the probably most considered convex constraints, namely linear equality as well as convex inequality restrictions. For the last mentioned ones, also no smoothness assumptions are imposed. In addition, we wanted to provide the possibility of keeping somewhat easy constraints, for example box constraints, in the subproblems. Thus, we add an abstract constraint to our problem formulation. Moreover, we adopt the idea of the classical DCA~\cite{LeThi1996} to approximate the concave part of the objective function in each iteration by an affine majorization. Combining this approach with the basic concept of safeguarded augmented Lagrangian methods, which has some advantages against the classical augmented Lagrangian method (see~\cite{Kanzow2017} for a discussion), we yield convex subproblems, being, in general, simpler to solve than DC~problems. Note that quite many applications (with convex feasible region) involve only linear equality and convex inequality constraints but no further abstract ones. In that case we even may obtain unconstrained convex subproblems. (Though, there is still the possibility of keeping such constraints explicitly whenever they are easy to handle.)

The work is organized as follows. In \cref{sec: preliminaries} we first recall some basic concepts and results from nonsmooth convex analysis which will come into play when examining the properties of the subproblems occurring in our new algorithm before deriving some optimality conditions for the DC~problem under consideration. The new solution method itself gets then introduced and analyzed in \cref{sec: psALMDC}. We conduct in \cref{sec: numerics} some numerical experiments that compare the performance of the new algorithm to established solution methods by means of some applications stemming from location planning as well as compressed sensing. Afterwards, we close with some final remarks in \cref{sec: conclusion}.

\section{Preliminaries} \label{sec: preliminaries}
In this section we first recall some basic definitions and results from nonsmooth convex analysis before turning to the derivation of some optimality conditions for our problem under consideration.

\subsection{Basics from nonsmooth convex analysis}
In the following, we provide essential definitions and results which will be exploited later. For further reference, one can have a look, for example, at~\cite{Dhara2011,Lemarechal2001,Rockafellar1970,Schirotzek2007}.

An extended-valued function $\phi: \R^n \to \overline{\R}$ with $\overline{\R}:=\R \cup \{\infty\}$ is said to be \emph{proper} if its \emph{domain} $\dom(\phi):= \left\{ x \in \R^n \ \middle| \ \phi(x)<\infty \right\}$ is nonempty. It is referred to as \emph{lower semicontinuous (lsc)} if 
\[ \liminf_{x \to \overline{x}} \phi(x) \geq \phi \left(\overline{x}\right)\]
is satisfied for every $\overline{x} \in \R^n$.
Furthermore, a proper function is called \emph{convex} whenever
\[ \phi \left(\lambda x + (1-\lambda)y \right) \leq \lambda \phi(x)+(1-\lambda) \phi(y) \]
holds for all $x,y \in \R^n$ and $\lambda \in (0,1)$, and \emph{strongly convex} if there exists some $\mu > 0$ such that $\phi - \frac{\mu}{2} \Vert \cdot \Vert^2$ is convex. Here and throughout the paper, we denote by $\Vert \cdot \Vert$ the Euclidean norm. Real-valued strongly convex functions attain an unique minimum on a nonempty, closed and convex set. Furthermore, a proper convex function $\phi:\R^n \to \overline{\R}$ is directionally differentiable at any point $x\in \interior \left(\dom(\phi)\right)$ (with $\interior$ denoting the interior of a set) and in each direction $d \in \R^n$ with the corresponding \emph{directional derivative} given by
\[ \phi '(x;d):= \lim_{t \downarrow 0} \frac{\phi(x+td)-\phi(x)}{t}\]
(see e.g. Theorem~2.76 in~\cite{Dhara2011}).
Considering nonsmooth convex functions, we resort to the \emph{convex subdifferential} as a concept of generalized differentiation which, at a point $x \in \dom(\phi)$, is given by the set
\[ \partial \phi(x):=\left\{ s \in \R^n \ \middle| \ \phi(y) \geq \phi(x) +s^T(y-x) \ \forall y \in \R^n \right\}. \]
For every $x \in \interior \left(\dom(\phi)\right)$, the set $\partial \phi (x)$ is nonempty, convex and compact. Moreover, the directional derivative of such $\phi$ relates to the convex subdifferential via
\begin{equation} \label{eq: Subdiff_dirDerivative}
    \phi'(x;d)=\max_{s \in \partial \phi(x)} s^Td
\end{equation}
for each $x \in \interior \left(\dom (\phi)\right)$ and $d \in \R^n$ (see Theorem~23.4 in~\cite{Rockafellar1970}).
Besides, let us mention the following calculus rules for convex subdifferentials which will be used during our analysis.

\begin{theorem}[calculus rules for the convex subdifferential] \label{thm: cvxSubdiff_calculus}
    Let $\phi_1,\dots,\phi_r:\R^n \to \overline{\R}$ be proper and convex functions and $\alpha_1,\dots,\alpha_r>0$ positive constants. 
    \begin{enumerate}
        \item \label{thm: cvxSubdiff_sumRule} Suppose that either $\bigcap_{i=1}^r \ri \left[\dom \left(\phi_i\right)\right] \neq \emptyset$, where $\ri$ denotes the relative interior of a set, or there exists some $\widetilde{x}\in \bigcap_{i=1}^r \dom \left(\phi_i\right)$ such that all $\phi_i$ with at most one exception are continuous at $\widetilde{x}$. Then
            \[ \partial \left(\sum_{i=1}^r \alpha_i \phi_i\right) (x) = \sum_{i=1}^r \alpha_i \partial \phi_i(x) \]
            holds for all $x \in \dom \left(\sum_{i=1}^r \alpha_i \phi_i\right)$. (Moreau-Rockafellar sum rule, see e.g. Theorem~2.91 in~\cite{Dhara2011} and Theorem~4.5.1 in~\cite{Schirotzek2007}, respectively)

        \item Assume all functions $\phi_i, \ i=1,\dots,r,$ to be real-valued and let $ \Phi: \R^n \to \R, \  \Phi(x):= \linebreak \max \left\{ \phi_1(x),\dots, \phi_r(x)\right\}$. Then for every $x \in \R^n$ one has
            \[ \partial \Phi(x) = \conv \bigg\{\bigcup_{i \in I(x)} \partial \phi_i(x) \bigg\},\]
            where $\conv$ denotes the convex hull of a set and $I(x):=\left\{ i \in \{1,\dots, r\} \ \middle| \ \Phi(x)=\phi_i(x)\right\}$ the active index set. (see e.g. Theorem~2.96 in~\cite{Dhara2011})

        \item \label{thm: cvxSubdiffDifferentiable} In case $\phi_1$ is differentiable at some point $x \in \interior \left(\dom \left(\phi_1\right)\right)$ it holds
            \[ \partial \phi_1 (x) = \left\{\nabla \phi_1(x)\right\}.\]
            (see e.g. Proposition~2.80 in~\cite{Dhara2011})

        \item \label{thm: cvxSubdiffChainRule} Let $F: \R \to \R$ be continuously differentiable and assume $\phi_1$ to be real-valued, then for each $x \in \R^n$ one has
            \begin{equation} \label{eq: chainRule}
                \partial_C \left(F \circ \phi_1 \right)(x) = \nabla F \left(\phi_1(x)\right) \partial \phi_1(x),
            \end{equation}
            where $ \partial_C \left(F \circ \phi_1 \right)(x)$ denotes the Clarke subdifferential\footnote{It is known that the Clarke subdifferential coincides for convex functions with the convex subdifferential (see Proposition~2.2.7 in~\cite{Clarke1990}). While applying the proposed chain rule in this work it is always assured that also $F \circ \phi_1$ is a convex function. Hence, \eqref{eq: chainRule} then breaks down to pure convex analysis.} of $F \circ \phi_1$ at~$x$. (see Theorem~2.3.9(ii) in~\cite{Clarke1990})
    \end{enumerate} 
\end{theorem} 

Beyond that, the following two properties, known as local boundedness and closedness of the graph of the convex subdifferential, will be crucial for our convergence analysis.
\begin{theorem} \label{thm: cvxSubdiff_localBoundedClosed}
    Let $\phi:\R^n \to \overline{ \R}$ be proper, lsc and convex. Then, the following holds: 
    \begin{enumerate}
        \item \label{thm: cvxSubdiffLocalBounded} For a nonempty and compact subset $K \subseteq \interior \left(\dom (\phi)\right)$ the image set 
            \[ \partial \phi (K):= \bigcup_{x \in K} \partial \phi(x) \]
            is nonempty and compact. (see Proposition~2.85 in~\cite{Dhara2011})
        \item \label{thm: cvxSubdiffClosedGraph} For any sequence $\left\{ \left(x^k,s^k\right) \right\}_{k \in \N} \subseteq \R^n \times \R^n$ with $s^k \in \partial \phi \left( x^k\right)$ for all $k \in \N$ such that $\left(x^k,s^k\right)$ converges to some $(x,s) \in \R^n \times \R^n$,  $s \in \partial \phi(x)$ follows. (see Theorem~2.84 in~\cite{Dhara2011})
    \end{enumerate}
\end{theorem}

Given a nonempty set $C \subseteq \R^n$, the \emph{tangent cone} of $C$ at $x\in C$ is given by
\[ T_C(x):=\Limsup_{t \downarrow 0} \frac{C-x}{t}, \]
where for a set-valued map $ F: \R^n \rightrightarrows \R^m$   
\[\Limsup_{x \to \overline{x}} F(x):= \left\{y \in \R^m \ \middle| \ \exists x^k \rightarrow \overline{x}, y^k \rightarrow y  \text{ such that } y^k \in F\left(x^k\right) \ \forall k \in \N \right\}\]
describes the \emph{Painlev\'{e}-Kuratowski outer/upper limit} of $F$ at $\overline{x}$ provided that $F\left(\overline{x}\right)$ is nonempty.
Assuming, in addition, $C$ to be convex, $T_C \left(x\right)$ becomes a closed and convex cone (see Theorem~2.35 in~\cite{Dhara2011}). Furthermore, for a convex set $C \subseteq \R^n$, the
\emph{normal cone} of $C$ in $x \in C$ is defined by
\[ N_C(x):= \left\{ s \in \R^n \ \middle| \ s^T(y-x) \leq 0 \ \forall y \in C\right\}.\] 
Moreover, for convex $C$, tangent and normal cone are polar to each other, that means
\begin{equation} \label{eq: NormalToTangentCone}
    N_C(x)= \left(T_C(x)\right)^{\circ} \quad \text{and} \quad T_C(x)=\left(N_C(x)\right)^{\circ} \quad \forall x \in C,
\end{equation}
where for an arbitrary set $K \subseteq \R^n$, the expression $K^{\circ}$ denotes the corresponding \emph{polar cone} defined by
\[K^{\circ}:=\left\{ y \in \R^n \ \middle| \ y^Tx \leq 0 \ \forall x \in K\right\}.\]
(see Proposition~2.37 in~\cite{Dhara2011})
In addition, for convex $C \subseteq \R^n$, one has 
\begin{equation} \label{eq: normalCone_indicator}
    N_C(x)= \partial \delta_C(x)
\end{equation}
for any $x \in C$, where $ \linebreak[2] \delta_C: \nolinebreak \R^n \to \nolinebreak \overline{\R}$ denotes the \emph{indicator function} of the set $C$ (see page~89 in~\cite{Dhara2011}).
This relation together with the closedness of the graph of the convex subdifferential reveals that for closed $C$ the normal cone is robust in the sense that
\begin{equation} \label{eq: normalCone_Robustness}
    N_C\left(\overline{x}\right)= \Limsup_{x \xrightarrow{C} \overline{x}} N_C(x)
\end{equation}
holds for all $\overline{x} \in C$, where we indicate with $x \xrightarrow{C} \overline{x}$ that only sequences contained in~$C$ and converging to $\overline{x}$ are taken into account.

Finally, consider the optimization problem
\begin{equation} \label{eq: convexOptimization}
    \min_{x \in \R^n} \phi(x) \quad \text{s.t.} \quad x \in C
\end{equation}
with a convex function $\phi:\R^n \to \R$ and a convex set $C\subseteq \R^n$.
\begin{theorem}[optimality condition for convex optimization, cf. Theorem~3.1 in~\cite{Dhara2011}] \label{thm: cvx_OptCond}
    A point $x \in \R^n$ is a solution to the convex optimization problem~\eqref{eq: convexOptimization} if and only if
    \[ 0 \in \partial \phi(x)+N_C(x).\]
\end{theorem}

Let us close with stating a separation theorem being an important tool for proving equivalence of certain constraint qualifications later on.
\begin{theorem}[separation theorem, cf. Theorem~2.26(ii) in~\cite{Dhara2011}] \label{thm: separationTheorem}
    Let $C_1,C_2 \subseteq \R^n$ be nonempty, convex and disjoint sets. Then there exists some $ \linebreak[2] a \in \nolinebreak \R^n \setminus \nolinebreak \{0\}$ with
    \[ a^T x_1 \leq a^T x_2 \qquad \forall x_1 \in C_1, \ x_2 \in C_2.\]
\end{theorem}

\subsection{Optimality conditions}
Here, we derive necessary optimality conditions for problem~\eqref{eq: problemFormulation} following primarily \cite{Pang2017} and then extending the approach of~\cite{Dhara2011}. 
To this end, let us write
\begin{equation} \label{eq: equalityConstraints}
    \ell : \R^n \rightarrow \R^p, \qquad \ell (x) := Ax-b
\end{equation}
for the equality constraints of~\eqref{eq: problemFormulation}, and
\begin{equation} \label{eq: feasibleSet}
    X:= C \cap c^{-1} \left( \{ (-\infty,0]^m \} \right) \cap \ell^{-1} \left( \{0\} \right)
\end{equation}
for the feasible set of~\eqref{eq: problemFormulation}.

Now, assume $x^\ast \in \R^n$ to be an arbitrary local minimum to~\eqref{eq: problemFormulation}. It is well known that $x^\ast$ satisfies
\[ f'\left(x^\ast;d\right)\geq 0 \qquad \forall d \in T_X \left(x^\ast \right)\]
or, equivalently, by exploiting convexity of the feasible set $X$,
\[ f'\left(x^\ast;x-x^\ast \right) \geq 0 \qquad \forall x \in X \]
(see e.g.~\cite{Pang2021}). 
Any point meeting this last condition is called a directional stationary point or, for short, \emph{d-stationary point} of~\eqref{eq: problemFormulation}.
Taking into account the DC~structure of the objective function, d-stationarity can be rewritten as
\[ g'\left(x^\ast;x-x^\ast \right) \geq h' \left( x^\ast;x-x^\ast \right) \qquad \forall x \in X. \]
Now, utilizing the relation~\eqref{eq: Subdiff_dirDerivative} leads to
\[ g'\left(x^\ast;x-x^\ast \right) \geq s^T\left(x-x^\ast\right) \qquad \forall x \in X, \ s \in \partial h \left( x^\ast \right).\]
Since d-stationarity is both necessary and sufficient for optimality in the convex setting (see e.g. again~\cite{Pang2021}), it follows
\[x^\ast \in \argmin_{x \in X} \left\{ g(x)-s^Tx\right\} \qquad \forall s \in \partial h \left(x^\ast \right)\]
which, in turn, due to the optimality condition from \cref{thm: cvx_OptCond}, is equivalent to 
\[ 0 \in \partial g\left( x^\ast \right)-\{s\} + N_X \left( x^\ast \right) \qquad \forall s \in \partial h \left(x^\ast \right)\]
proving
\begin{equation} \label{eq: dStationarity_DC}
    \partial h \left( x^\ast \right) \subseteq \partial g \left( x^\ast \right) + N_X \left( x^\ast \right).
\end{equation}
Note that the above considerations show that the concept of a d-stationary point is equivalent to the inclusion~\eqref{eq: dStationarity_DC}. However, in practice it is usually quite difficult to verify d-stationarity. Hence, one often resorts to the weaker notion of a \emph{critical point} of~\eqref{eq: problemFormulation}, that is a feasible point satisfying
\begin{equation} \label{eq: criticality}
    \partial h \left( x^\ast \right) \cap \left( \partial g \left( x^\ast \right) + N_X \left( x^\ast \right) \right) \neq \emptyset.
\end{equation}
Clearly, with the above explanations in mind, this still poses a necessary optimality condition for~\eqref{eq: problemFormulation}.
In order to obtain a suitable expression for the normal cone $N_X \left( x^\ast \right)$ which exploits the special structure of the convex set $X$, one can apply the following Slater-type constraint qualification.

\begin{definition}[StCQ] \label{def: SCQ}
    We say that the \emph{Slater-type Constraint Qualification} (StCQ) holds for~\eqref{eq: problemFormulation} whenever there exists some $\widetilde{x} \in \ri (C) $ such that $c_i \left( \widetilde{x} \right) < 0 $ holds for all $i=1,\dots,m$ as well as $A \widetilde{x}=b$. 
\end{definition}

With this constraint qualification at hand one can derive the following alternative optimality condition.

\begin{proposition} \label{prop: optCondDC_generalizedKKT}
    Assume that StCQ is satisfied for~\eqref{eq: problemFormulation}. Then $x^\ast \in \R^n$ is a critical point of~\eqref{eq: problemFormulation} if and only if there exist $\lambda^\ast \in \R^m, \ \mu^\ast \in \R^p$ such that
    \begin{subequations} \label{eq: KKT}
        \begin{align}
            & \partial h \left( x^\ast \right) \cap \bigg( \partial g \left( x^\ast \right)  + \left\{ A^T \mu^{\ast} \right\} + \sum_{i=1}^m \lambda^{\ast}_i \partial c_i \left( x^\ast \right) + N_C \left( x^\ast \right)\bigg) \neq \emptyset, \label{eqSys: KKT_optimality} \\
            & c_i\left(x^\ast\right) \leq 0, \quad \lambda_i^\ast \geq 0,  \quad \lambda^{\ast}_i c_i \left( x^\ast \right) = 0 \quad \forall i=1,\dots, m, \label{eqSys: KKT_complementary}\\
            & A x^\ast=b, \quad x^\ast \in C.
        \end{align}
    \end{subequations} 
\end{proposition}
\begin{proof}
    This immediately follows from the definition of a critical point,~\eqref{eq: criticality}, as under StCQ it holds for all feasible points $x \in X$
    \[ N_X(x) = N_{ c^{-1} \left( \left\{ (-\infty,0]^m \right\} \right)}(x)+N_{\ell^{-1} \{0\}}(x)+N_C(x) \]
    with
    \begin{align*}
        N_{c^{-1} \left( \left\{ (-\infty,0]^m \right\} \right)} (x)&= \left\{ \sum_{i=1}^m \lambda_i \omega_i \in \R^n \ \middle| \ \omega_i \in \partial c_i(x), \  \lambda_i \geq 0, \ \lambda_i c_i(x)=0, \ i=1,\dots,m \right\},\\
        N_{\ell^{-1} \left( \{0\} \right)}(x) &=  \left\{ A^T \mu \in \R^n \ \middle| \ \mu \in \R^p \right\} 
    \end{align*}
    \enlargethispage{1\baselineskip} 
    (see Proposition~3.3 and Theorem~3.12 in~\cite{Dhara2011}).
\end{proof}

Any triple $\left( x^\ast, \lambda^\ast, \mu^\ast \right) \in \R^n \times \R^m \times \R^p$ meeting the conditions~\eqref{eq: KKT} is called a \emph{generalized Karush-Kuhn-Tucker (KKT)~point} of~\eqref{eq: problemFormulation}. Thereby, the requirement $ \lambda_i^\ast c_i \left(x^\ast \right) =0  $ for all $i=1,\dots,m$ is referred to as \emph{complementary slackness}. Further, observe that~\eqref{eqSys: KKT_complementary} can be equivalently expressed as $\min \left\{-c\left(x^\ast \right), \lambda^\ast \right\} = 0$, where here and in the following $\min\{y,z\}$ for two vectors $y,z \in \R^m$ (and the analog holds true for $\max \{y,z\}$) is computed componentwise, meaning
\[ \left(\min\{y,z\}\right)_i:= \min \left\{y_i,z_i\right\} \qquad i=1,\dots,m.\]
Let us note that the term of a generalized KKT~point is not used consistently in the literature. In~\cite{Pang2017} it coincides with that of a d-stationary point (see~\eqref{eq: dStationarity_DC}) although no Lagrange multiplier appears therein. However, in \cite{LeThi2014,PhamDinh2014} they call $x^\ast \in \R^n$ a KKT~point of~\eqref{eq: problemFormulation} whenever there exist some $\lambda^\ast \in \R^m, \ \mu^\ast \in \R^p$ with
\begin{align*}
    &0 \in \partial_C f \left( x^\ast \right)+ \left\{A^T \mu^\ast \right\} + \sum_{i=1}^m \lambda_i^\ast \partial c_i \left(x^\ast\right) +N_C \left(x^\ast\right),\\
    &c_i\left(x^\ast \right) \leq 0, \quad \lambda_i^\ast \geq 0,  \quad \lambda^{\ast}_i c_i \left( x^\ast \right) = 0 \quad \forall i=1,\dots, m, \\
    & A x^\ast=b, \quad x^\ast \in C
\end{align*}
where $\partial_C f \left( x^\ast \right)$ denotes again the Clarke subdifferential~\cite{Clarke1990,Pang2021} of the objective function~$f$ at~$x^\ast$. Looking into the Clarke subdifferential it is apparent that the notion of a KKT~point given by~\cite{LeThi2014,PhamDinh2014} is stronger than the one used in this work. Actually, this comes from exploiting the DC~structure and explains our notion of a \emph{generalized} KKT~point.

\section{The algorithm and its convergence properties} \label{sec: psALMDC}

Our approach is based on the classical safeguarded augmented Lagrangian method (see e.g. \cite{Kanzow2017,Birgin2012,Birgin2014}). But instead of applying the technique directly to our problem~\eqref{eq: problemFormulation}, we first replace the concave part $-h$ of the objective function by a linearization. 
This way, we avoid a DC~structure in our subproblems but obtain convex optimization problems instead, which, in general, reduces the computational effort of solving the subproblems.
In order to approximate $-h$, we follow  the classical DC~Algorithm~\cite{LeThi1996} from unconstrained DC~optimization which means, at iteration $k \in \N_0$ with current iterate $x^k$, we pick a subgradient $s^k \in \partial h \left(x^k\right)$ and replace $-h$ by its affine majorant
\[ x \mapsto - h \left(x^k\right) -\left(s^{k}\right)^T \left(x-x^k\right). \]
In addition, we only penalize the equality and inequality constraints in~\eqref{eq: problemFormulation}, while keeping the abstract constraint $x \in C$ explicitly in our subproblems. Hence, after approximating the objective function and augmenting only the equality and inequality constraints, the augmented Lagrangian function in each iteration $k$ reads
\begin{equation*}
    \begin{multlined}
        L_{\rho}^{(k)} (x, \lambda, \mu):=  g(x)-h\left(x^k\right) - \left(s^k \right)^T \left(x-x^k\right)+ \mu^T (Ax-b) + \frac{\rho}{2} \Vert Ax-b \Vert^2 \\
        + \frac{1}{2\rho} \left( \| \max \{0, \lambda + \rho c(x) \} \|^2 - \Vert \lambda \Vert^2 \right),
    \end{multlined}
\end{equation*}
which means the augmented Lagrangian function changes from iteration to iteration. Hence, we call this an adaptive (safeguarded) augmented Lagrangian approach. 
Adding a proximal term in the subproblems results in the following proximal safeguarded augmented Lagrangian method for DC~problems, psALMDC for short.
\renewcommand{\labelenumi}{\theenumi}

\begin{alg}{(psALMDC)} \label{alg: sALMDC}
\item \label{alg: initialisation} Let $ \left\{Q_k\right\}_{k \in \N_0} \subseteq \R^{n \times n}$ be a sequence of symmetric positive definite matrices satisfying $\linebreak \underline{\lambda}\leq \nolinebreak \lambda_{\min} \left(Q_k \right) \leq \lambda_{\max} \left(Q_k\right) \leq \overline{\lambda}$ for all $k \in \N_0$ and some $\underline{\lambda},\overline{\lambda}>0$.
    Choose $\left(x^0,v^0,u^0\right) \in \R^n\times \R^p \times \nolinebreak \R^m_+ , \linebreak[3] \ \sigma_0,\epsilon_0>0, \ M,N \in \N \ (N<\nolinebreak M), \ \alpha,\beta,\gamma,\theta \in (0,1), \linebreak[2] \ \overline{\eta}>\nolinebreak 1$. Set $\rho_0:= \left( \sigma_0\right)^\gamma, \ K:=I_1:=0, \ k:=0$.
\item \label{alg: subprob} Select $s^k \in \partial h\left(x^k \right)$. Compute 
    \[x^{k+1} = \argmin_{x \in C} \left\{ L_{\rho_k}^{(k)} \left( x,u^k,v^k\right) + \frac{\sigma_k}{2}\| x-x^k \|_{Q_k}^2\right\}. \]
\item \label{alg: LagrangeMultiplier} Set 
    \begin{align*}
        \mu^{k+1}&:=v^k+\rho_k \left(Ax^{k+1}-b\right), \\
        \lambda^{k+1}&:= \max \left\{ 0, u^k+\rho_k c\left(x^{k+1} \right) \right\}.
    \end{align*}

\item \label{alg: termination} IF $\sigma_k Q_k \left( x^{k+1}-x^k \right)=0$ AND $Ax^{k+1}=b$ AND $\min \left\{-c \left(x^{k+1}\right), \lambda^{k+1} \right\}=0$: STOP.

\item \label{alg: feasibilityCheck}
    \begin{enumerate}
        \item \label{alg: feasibilityCheck_succ}
                    IF $ \| Ax^{k+1}-b \| \leq \theta \| Ax^k-b \|$ AND 
                    $\Big\| \min \left\{ -c \left(x^{k+1} \right), \ \frac{u^k}{\rho_k} \right\} \Big\| \leq \theta \Big\| \min \left\{ -c \left(x^{k} \right), \ \frac{u^{k-1}}{\rho_{k-1}} \right\} \Big\|$
                set 
                \[\sigma_{k+1}:=\sigma_k, \ \rho_{k+1}:= \rho_k, \ \epsilon_{k+1}:=\epsilon_k, \]

            \item \label{alg: feasibilityCheck_unsucc} ELSE set
                \begin{align*}
                    \eta &:= \begin{cases}
                        \overline{\eta} & \text{if } \| x^{k+1}-x^k \|^{\alpha}\geq \epsilon_k, \\
                        1 & \text{otherwise,}
                    \end{cases}\\
                    \sigma_{k+1}&:= \begin{cases}
                        \overline{\eta} \sigma_k & \text{if } x^{k+1}=x^k, \\
                        \max \left\{ \frac{1}{\| x^{k+1}-x^k \|^{\alpha}},\eta\sigma_k \right\} & \text{otherwise},
                    \end{cases}\\
                    \rho_{k+1}&:= \left( \sigma_{k+1}\right)^{\gamma}.
                \end{align*}
                \begin{tabbing}
                    \hspace{0.5cm}\= IF \= $ \| x^{k+1}-x^k \|^{\alpha} < \epsilon_k$\\
                    \>\>$ K \leftarrow K+1.$\\
                    \> IF $ \sigma_{k+1} = \frac{1}{\| x^{k+1}-x^k \|^{\alpha}}$\\
                    \>\>$ I_1 \leftarrow I_1+1.$\\
                    \> IF $ K \geq M$ AND $ I_1 \geq N$ set\\
                    \>\> $\epsilon_{k+1}:= \beta \epsilon_k$,\\
                    \>\> $K  \leftarrow 0$,\\
                    \>\> $I_1 \leftarrow 0$,\\
                    \> ELSE set\\
                    \>\>$ \epsilon_{k+1}:=\epsilon_k.$
                \end{tabbing}
        \end{enumerate}
    \item \label{alg: auxiliarySequences} Set 
        \begin{align*}
            v^{k+1}&:= \begin{cases}
                v^k & \text{if } Ax^{k+1}=b,\\
                \frac{\left(v^k\right)^T\left(Ax^{k+1}-b\right)}{\| Ax^{k+1}-b \|^2} \left(Ax^{k+1}-b\right) & \text{otherwise},
            \end{cases}\\
            I&:= \left\{ i \in \{1,\dots,m\} \ \middle| \ c_i\left(x^{k+1}\right)>0\right\}, \\
            u_i^{k+1}&:= \begin{cases}
                \frac{\left(u_I^k\right)^Tc_I \left(x^{k+1}\right)}{\| c_I \left(x^{k+1} \right) \|^2}c_i \left(x^{k+1} \right) & \text{if } i \in I, \\
                u_i^k & \text{if } i \in \{1,\dots,m\} \setminus I.
            \end{cases}
        \end{align*}

    \item Set $k \leftarrow k+1$, and go to \ref{alg: subprob}. 
\end{alg}

Before dedicating ourselves to the convergence theory, let us give some more insights in the motivation of the algorithm. First, choosing the next iterate $x^{k+1}$ in \ref{alg: subprob} as a minimizer of the modified augmented Lagrangian function $L_{\rho_k}^{(k)} \left( \cdot, u^k,v^k \right)$ over the convex set $C$ without the proximal term and taking into account the necessary optimality condition leads to
\begin{equation}
    \begin{aligned}[t]
        s^k &\in 
        \partial g \left(x^{k+1} \right) +  \left\{A^T v^k + \rho_{k} A^T\left(Ax^{k+1}-b\right) \right\} + \sum_{i=1}^m \max \left\{0,u_i^k+\rho_k c_i\left(x^{k+1} \right) \right\} \partial c_i\left(x^{k+1} \right) + N_C \left(x^{k+1} \right) \\
        &= \partial g \left(x^{k+1} \right) + \left\{A^T\mu^{k+1} \right\} + \sum_{i=1}^m \lambda_i^{k+1} \partial c_i \left( x^{k+1} \right) + N_C \left(x^{k+1} \right)	  \label{eq: optCond_subprobWithoutProx}
    \end{aligned}
\end{equation}
with $s^k \in \partial h \left(x^k\right)$ by our choice. Hence, as we aim to determine a generalized KKT point~\eqref{eq: KKT}, we would like to have $x^{k+1}=x^k$ (at least approximately) in the end. To yield the latter (asymptotically), it is a common technique to add a proximal term. This also ensures existence and uniqueness of the solution of each subproblem as the objective function is strongly convex and the minimization is carried out with respect to a convex set. Using $\Vert \cdot \Vert_{Q_k}$ within the proximal term instead of the Euclidean norm  appears to be advantageous in some numerical experiments in order to relativize the effect that the proximal parameter might grow rapidly. 
Note that the conditions imposed in~\ref{alg: initialisation} on the sequence $\left\{Q_k \right\}_{k \in \N_0}$ of symmetric positive definite matrices correspond to $\left\{Q_k \right\}_{k \in \N_0}$ and $\left\{Q_k^{-1} \right\}_{k \in \N_0}$ being bounded.

In order to establish the convergence theory of the classical safeguarded augmented Lagrangian method~\cite{Birgin2014}, it is crucial that the sequence of penalty parameters $\left\{ \rho_k \right\}_{k \in \N_0}$ stays bounded only if there is sufficient progress in feasibility and complementarity, whereas otherwise $\left\{\rho_k\right\}_{k_\in \N_0}$ grows to infinity.
It turns out that for our algorithm this property is not only required for the sequence of penalty parameters  $\left\{ \rho_k \right\}_{k \in \N_0}$, but also for that of proximal parameters $\left\{ \sigma_k \right\}_{k \in \N_0}$. Having in mind the previous paragraph with the optimality condition obtained from \ref{alg: subprob} given by
\begin{equation} \label{eq: optCond_subprob}
    s^k - \sigma_k Q_k \left(x^{k+1}-x^k \right) \in \partial g \left(x^{k+1} \right) + \left\{A^T\mu^{k+1} \right\} + \sum_{i=1}^m \lambda_i^{k+1} \partial c_i \left( x^{k+1} \right) + N_C \left(x^{k+1} \right)
\end{equation}
one would like to have $\sigma_k Q_k \left(x^{k+1}-x^k \right) \approx 0$ in the end, conferring~\eqref{eq: optCond_subprobWithoutProx}. 
Consequently, one has to ensure that the sequence $\left\{ \sigma_k \right\}_{k \in \N_0}$ of proximal parameters does not grow too fast. An adaptive update depending on the decay of $ \| x^{k+1} - x^k \| $ seems reasonable. 
Moreover, the sequence of penalty parameters~$\left\{\rho_k \right\}_{k \in \N_0}$ has to grow slower than the one of proximal parameters $\left\{ \sigma_k \right\}_{k \in \N_0}$, which suggests the update formula for $\rho_{k+1}$ in \ref{alg: feasibilityCheck_unsucc}.
Note that in \ref{alg: feasibilityCheck_unsucc} both the parameter $\eta$ and the counters $K$ and~$I_1$ depend on the iteration index $k$, but, to simplify the notation, we suppress this dependence in our notation. We stress that these counters play a central role for a careful update of $\epsilon_k$, and they will essentially be used in the proof of \cref{lem: proximalGrowth} only.

For the classical safeguarded augmented Lagrangian method~\cite{Birgin2014}, the only requirement imposed on the auxiliary sequences $\left\{v^k\right\}_{k \in \N_0} \subseteq \R^p$ and $\left\{u^k\right\}_{k \in \N_0} \subseteq \R^m_+$ for the construction of Lagrange multipliers is their boundedness.
In our context, however, the boundedness of these multipliers is not enough, and we have to choose the corresponding sequences in a particular way. More precisely, we require the sequences to satisfy the two inequalities
(see the proof of \cref{lem: NearMonotonicity})
\begin{align}
    \left(v^{k+1}-v^k\right)^T\left(Ax^{k+1}-b \right) &\leq 0, \label{eq: constructionMot_vk}\\
    \Big\| \max \left\{0, u^{k+1}+\rho_{k+1}c \left( x^{k+1} \right) \right\} \Big\|^2 - \| u^{k+1}\|^2 &\leq \Big\| \max \left\{0, u^{k}+\rho_{k+1}c \left( x^{k+1} \right) \right\} \Big\|^2 - \| u^{k}\|^2 \label{eq: constructionMot_uk}
\end{align}
for each $k \in \N_0$.
To motivate the precise update formula for $v^{k+1}$ in \ref{alg: auxiliarySequences}, we stick with the previous estimate for the Lagrange multiplier whenever the corresponding iterate $x^{k+1}$ already satisfies the equality constraints. Otherwise the equality constraints are violated, which means $Ax^{k+1}-b \neq 0$, and then our update formula results from considering the optimization problem 
\begin{equation} \label{eq: motivation_vk}
    \min_{v \in \R^p} \frac{1}{2} \Vert v \Vert^2 \quad \text{s.t.} \quad \left(v -v^k \right)^T \left(Ax^{k+1}-b \right) =0
\end{equation}
for determining $v^{k+1}$. 
This problem can be solved analytically and leads to the update of $v^k$ from \ref{alg: auxiliarySequences}. Note that $v^k$ is also feasible for~\eqref{eq: motivation_vk}, so we immediately obtain
\begin{equation} \label{eq: vk_monotonicity}
    \| v^{k+1} \| \leq \| v^k \|
\end{equation}
for all $k \in \N_0$. 
This ensures the boundedness of the sequence $\left\{v^k\right\}_{k \in \N_0}$. Let us note that in the optimization problem~\eqref{eq: motivation_vk} we place equality constraints though one might expect inequality constraints due to the requirement~\eqref{eq: constructionMot_vk}. The reason for this is that considering inequality constraints would change the unique solution of the optimization problem to $v^{k+1}=0$ whenever one has $-\left(v^k \right)^T \left(Ax^{k+1}-b\right) \leq 0$. As a consequence, also all subsequent members of the sequence would be equal the zero vector which, in general, does not seem to yield a good estimate for the Lagrange multiplier of the generalized KKT~point. 

Our construction of the auxiliary sequence $\left\{u^k\right\}_{k \in \N_0}$ is mainly motivated by the construction of~$\left\{v^k\right\}_{k \in \N_0}$ and results in a non-negative sequence meeting the requirement~\eqref{eq: constructionMot_uk} with equality and, in addition,
\[ \| u^{k+1} \| \leq \| u^k \|\]
for all $k \in \N_0$, ensuring once again the boundedness of $\left\{u^k\right\}_{k \in \N_0}$.

Let us note that, due to the monotonicity of the sequences $\left\{\| v^k \|\right\}_{k \in \N_0}$ and $\left\{\| u^k \|\right\}_{k \in \N_0}$, \cref{alg: sALMDC} breaks down to a pure penalty method whenever we pick $v^0:=u^0:=0$.

With our considerations regarding the proximal term we have already anticipated in parts the motivation of the termination criterion in \ref{alg: termination}. Nevertheless, let us go through the reasoning with the following lemma.

\begin{lemma}
    Suppose \cref{alg: sALMDC} stops at some iteration $k$. Then, the triple $\left(x^{k+1},\lambda^{k+1},\mu^{k+1} \right)$ is a generalized KKT~point of~\eqref{eq: problemFormulation}.
\end{lemma}
\begin{proof}
    Note that at the point of termination one has $x^{k+1} = x^k$ and hence, $s^k \in \partial h \left(x^{k+1} \right)$ holds. Thus, the optimality condition~\eqref{eq: optCond_subprob} implies
    \[ \partial h \left(x^{k+1} \right) \cap \bigg( \partial g \left(x^{k+1} \right) +\left\{ A^T \mu^{k+1}\right\} + \sum_{i=1}^m \lambda_i^{k+1} \partial c_i  \left(x^{k+1} \right) + N_C \left(x^{k+1} \right) \bigg) \neq \emptyset.  \]
    Furthermore, $ \min \left\{-c\left(x^{k+1}\right),\lambda^{k+1} \right\}=0$ is equivalent to
    \[ c \left(x^{k+1} \right) \leq 0, \quad \lambda^{k+1}\geq 0, \quad \left(\lambda^{k+1}\right)^Tc \left(x^{k+1}\right)=0.  \]
    By construction, $x^{k+1} \in C$ holds for all $k \in \N_0$, and satisfaction of the equality constraints is guaranteed by the termination criterion. Altogether, this proves that the triple $\left(x^{k+1},\lambda^{k+1},\mu^{k+1} \right)$ is indeed a generalized KKT~point of~\eqref{eq: problemFormulation}.
\end{proof}

Theoretically, one could replace the first termination condition in \ref{alg: termination} with the equivalent postulation $x^{k+1} = x^k$. But for practical implementation, where one checks for satisfaction of the stopping criterion up to some tolerance, it will definitely make a difference whether to take the factor $\sigma_k Q_k$ into account or not, especially as $\sigma_k$ might grow significantly. Hence, keeping in mind the optimality condition~\eqref{eq: optCond_subprob}, it seems advisable to numerically incorporate the additional factor. Thus, we stick with this formulation also for our theoretical considerations.\\

In the subsequent convergence theory, we assume without loss of generality that infinite sequences $ \left\{\left(x^k,s^k,\sigma_k,\rho_k,\epsilon_k,u^k,v^k \right) \right\}_{k \in \N_0}$ and $ \left\{\lambda^k,\mu^k\right\}_{k \in \N}$ are generated by \cref{alg: sALMDC}. 

In order to establish our final convergence result, some preliminary lemmata are proven initially. The first result regarding the auxiliary sequences $\left\{ v^k \right\}_{k \in \N_0}, \ \left\{u^k \right\}_{k \in \N_0}$ for the construction of Lagrange multipliers summarizes our previous explanations to motivate their update formulae. 

\begin{lemma} \label{lem: auxiliarySequences}
    The sequences $\left\{ v^k \right\}_{k \in \N_0} \subseteq \R^p, \ \left\{u^k \right\}_{k \in \N_0} \subseteq \R^m_+$ are bounded and satisfy~\eqref{eq: constructionMot_vk} and~\eqref{eq: constructionMot_uk}, respectively (even with equality).
\end{lemma}
\begin{proof}
    To begin with the sequence $\left\{v^k \right\}_{k \in \N_0}$, in case $Ax^{k+1}=b$ holds for some $k \in \N_{0}$ the update formula in~\ref{alg: auxiliarySequences} reads $v^{k+1}=v^k$. Hence, on one hand, \eqref{eq: constructionMot_vk} is satisfied trivially with equality and, on the other hand, the same prevails for \eqref{eq: vk_monotonicity}. Thus, consider the case $Ax^{k+1} \neq b$ for some $k \in \N_{0}$. Then, $v^{k+1}$ gets determined as solution to problem~\eqref{eq: motivation_vk}. Obviously, the constraint ensures that \eqref{eq: constructionMot_vk} holds with equality. Besides, like already pointed out, feasibility of $v^k$ to problem~\eqref{eq: motivation_vk} yields anew~\eqref{eq: vk_monotonicity}. Altogether,  the whole sequence $\left\{\| v^k \| \right\}_{k \in \N_0}$ to be monotonically decreasing has been proven, which reveals boundedness of $\left\{v^k\right\}_{k \in \N_0}$.

    Turning to $\left\{u^k\right\}_{k \in \N_0}$, we first note that the sequence being (componentwise) non-negative follows inductively from the update formula in~\ref{alg: auxiliarySequences} by the choice of the index set~$I$.
    Next, recognize that for $k \in \N_0$ and $i \in \{1,\dots,m\}$ one has $u_i^{k+1}+\rho_{k+1} c_i \left(x^{k+1} \right) >\nolinebreak 0$ if and only if $u_i^k + \rho_{k+1} c_i \left(x^{k+1} \right) >0$ as the sequence  $\left\{u^k\right\}_{k \in \N_0}$ is (componentwise) non-negative and the update formula in~\ref{alg: auxiliarySequences} reveals that  $u_i^{k+1}\neq u_i^k$ occurs only if $c_i \left(x^{k+1} \right) >\nolinebreak 0$.
    Hence, introducing
    \begin{align*}
        J&:= \left\{ i \in \{1,\dots,m\} \ \middle| \ u_i^k + \rho_{k+1} c_i\left(x^{k+1}\right) > 0 \right\} \\
        & = \left\{ i \in \{1,\dots,m\} \ \middle| \ u_i^{k+1} + \rho_{k+1} c_i \left(x^{k+1} \right) > 0 \right\},
    \end{align*}
    one obtains
    \[ \Big\| \max \left\{0,u^{k+1}+\rho_{k+1} c\left(x^{k+1} \right) \right\} \Big\|^2 - \| u^{k+1} \|^2 = \| u_J^{k+1} \|^2 + 2 \rho_{k+1} \left(u_J^{k+1} \right)^T c_J\left(x^{k+1}\right)  + \rho_{k+1}^2 \| c_J\left(x^{k+1}\right) \|^2 - \| u^{k+1} \|^2 \]
    as well as
    \[
        \Big\| \max \left\{0,u^{k}+\rho_{k+1} c\left(x^{k+1} \right) \right\} \Big\|^2 - \| u^{k} \|^2= \| u_J^{k} \|^2 + 2 \rho_{k+1} \left(u_J^{k} \right)^T c_J\left(x^{k+1}\right) + \rho_{k+1}^2 \| c_J\left(x^{k+1}\right) \|^2 - \| u^{k} \|^2.
    \]
    Therefore, to show that~\eqref{eq: constructionMot_uk} holds with equality it apparently suffices to prove
    \begin{equation} \label{eq: uk_scalarproduct}
        \left(u_J^{k+1}\right)^Tc_J\left(x^{k+1} \right) = \left(u_J^k\right)^Tc_J\left(x^{k+1}\right)
    \end{equation}
    together with
    \begin{equation} \label{eq: uk_normDiff}
        \| u_J^{k+1} \|^2 - \| u^{k+1} \|^2 = \| u_J^k\|^2 - \| u^k\|^2.
    \end{equation}
    To begin with~\eqref{eq: uk_scalarproduct}, the update formula in~\ref{alg: auxiliarySequences} gives primarily $\left(u_I^{k+1}\right)^T c_I\left(x^{k+1}\right) = \left(u_I^k\right)^Tc_I\left(x^{k+1}\right)$. However, observing that $I \subseteq J$ and $u_i^{k+1}=u_i^k$ for $i \in J \setminus I$, equation~\eqref{eq: uk_scalarproduct} follows.
    Denoting for a subset $S \subseteq \{1,\dots,m\}$ with $\neg S$ its respective complement, that is $\neg S:= \{1,\dots,m\} \setminus S$, equation~\eqref{eq: uk_normDiff} is equivalent to $\| u_{\neg J}^{k+1} \|^2 = \| u_{\neg J}^k \|^2$. But, since $\neg J \subseteq \neg I$ and $u_i^{k+1} = u_i^k$ for all $i \in \neg I$, the latter clearly holds.

    Verifying boundedness of the sequence $\left\{u^k\right\}_{k \in \N_0}$ gets once more done by establishing monotonicity of the sequence $\left\{ \| u^k\| \right\}_{k \in \N_0}$. 
    To this end, recognize that the update formula in~\ref{alg: auxiliarySequences}  together with the Cauchy-Schwarz inequality yields for each $k \in \N_0$
    \begin{equation*}
        \begin{aligned}[b]
            \| u^{k+1} \|^2 &= \| u_I^{k+1} \|^2 +\| u_{\neg I}^{k+1} \|^2 \\
            &= \left( \frac{\left\vert \left( u_I^k\right)^T c_I\left(x^{k+1}\right) \right\vert}{\| c_I \left(x^{k+1}\right) \|^2} \right)^2 \| c_I\left(x^{k+1} \right) \|^2 + \| u_{\neg I}^k\|^2 \\
            &\leq \| u_I^k \|^2 + \| u_{\neg I}^k \|^2 = \| u^k\|^2.
        \end{aligned}
            \qedhere
    \end{equation*}
\end{proof}

In order to prove our final convergence result, we do not require the specific expressions of the auxiliary sequences $\left\{ v^k \right\}_{k \in \N_0}, \ \left\{u^k \right\}_{k \in \N_0}$ given in \ref{alg: auxiliarySequences}. Instead, only the two properties~\eqref{eq: constructionMot_vk} and~\eqref{eq: constructionMot_uk} will be exploited. Hence, other choices differing from the presented one are possible, too.
In particular, theoretically in each case also any constant sequence would serve the purpose. However, it is known from the classical safeguarded augmented Lagrangian method~\cite{Birgin2014} that in practice choosing $u^k$ and~$v^k$ as preferably good approximations to the conventional Lagrange multiplier updates of $\lambda^k$ and $\mu^k$, respectively, has proven one's worth. Consequently, adaptive choices seem to be more advantageous. With the just said at first glance it might be more promising than our hitherto existing approach to choose $u^k$ and $v^k$ as projections of the Lagrange multiplier updates $\lambda^k$ and $\mu^k$ onto a suitable convex set of vectors satisfying the requirements given in \cref{lem: auxiliarySequences}.
However, involving these considerations in some of our numerical experiments has not led to significant improvements compared to the implementation with the update suggested in \cref{alg: sALMDC}. In addition, such an alternative update is concomitant with a higher computational effort leading to longer runtime. Hence, we stick with the update formula proposed in \cref{alg: sALMDC}.

The crucial properties \eqref{eq: constructionMot_vk} and \eqref{eq: constructionMot_uk} allow us to verify the following inequality.

\begin{lemma} \label{lem: NearMonotonicity}
    Let
    \[
        L_{\rho}(x,\lambda,\mu):= g(x)-h(x)+\mu^T(Ax-b)+\frac{\rho}{2}\Vert Ax-b \Vert^2 + \frac{1}{2\rho} \left( \Big\| \max \{0,\lambda+\rho c(x) \} \Big\|^2- \Vert \lambda \Vert^2 \right)
    \]
    be the augmented Lagrangian function of~\eqref{eq: problemFormulation}
    and
    \begin{equation} \label{eq: SequenceMonotonicity}
        \begin{aligned}[t]
            \ell_{k+1} := & \, L_{\rho_k} \left(x^{k+1},\lambda^{k+1},\mu^{k+1}\right)-\rho_{k} \| Ax^{k+1}-b \|^2 -\frac{1}{2\rho_k}\left[ \Big\| \max \left\{0, \lambda^{k+1}+\rho_k c \left(x^{k+1} \right) \right\} \Big\|^2- \| \lambda^{k+1} \|^2 \right. \\ 
            &-\left. \left( \Big\| \max \left\{0, u^{k}+\rho_k c \left(x^{k+1} \right) \right\} \Big\|^2- \| u^{k} \|^2 \right)\right]
        \end{aligned}
    \end{equation}
    for $k \in \N$.
    In case $\rho_{k}=\rho_{k-1}$, it holds
    \[ \ell_{k} \geq \ell_{k+1} + \frac{\sigma_k}{2} \| x^{k+1} - x^k \|_{Q_k}^2.\]
\end{lemma}

\begin{proof}
    We will prove the assertion by showing that for all $k \in \N$ it holds
    \begin{align*}
        &L_{\rho_k} \left(x^k,\lambda^k,\mu^k\right)-\rho_{k-1} \| Ax^k-b \|^2
        -\frac{1}{2\rho_k}\left[ \Big\| \max \left\{0, \lambda^k+\rho_k c \left(x^k \right) \right\} \Big\|^2- \| \lambda^k \|^2  \right. \\
        &\quad -\left(  \left. \Big\| \max \left\{0, u^{k-1}+\rho_k c \left(x^k \right) \right\} \Big\|^2- \| u^{k-1} \|^2 \right) \right]\\
        & \geq \ell_{k+1}+\frac{\sigma_k}{2}\| x^{k+1}-x^k \|^2_{Q_k}.
    \end{align*}
    Involving the assumption $\rho_k=\rho_{k-1}$ the claim follows. To verify the above inequality we give evidence of the two estimates
    \begin{equation}
        L_{\rho_k}\left(x^k,u^k,v^k\right) \geq \ell_{k+1}+\frac{\sigma_k}{2}\| x^{k+1}-x^k \|^2_{Q_k} \label{eqSys: NearMonotonicity_lowerCase}
    \end{equation}
    and
    \begin{equation}
        \begin{aligned}[t]
            L_{\rho_k}\left(x^k,u^k,v^k\right) \leq & \, L_{\rho_k} \left(x^k,\lambda^k,\mu^k\right)-\rho_{k-1} \| Ax^k-b \|^2  -\frac{1}{2\rho_k}\left[ \Big\| \max \left\{0, \lambda^k+\rho_k c \left(x^k \right) \right\} \Big\|^2- \| \lambda^k \|^2  \right. \\
            &-\left(  \left. \Big\| \max \left\{0, u^{k-1}+\rho_k c \left(x^k \right) \right\} \Big\|^2- \| u^{k-1} \|^2 \right) \right] \label{eqSys: NearMonotonicity_upperCase}
        \end{aligned}
    \end{equation}
    for each $k \in \N$.

    First, observe that, due to the respective definitions, $ L_{\rho_k}\left(x^k,u^k,v^k\right) = L_{\rho_k}^{(k)}\left(x^k,u^k,v^k\right)$ holds. Taking into account the construction of $x^{k+1}$ in \ref{alg: subprob}, it follows that
    \[ L_{\rho_k}^{(k)}\left(x^k,u^k,v^k\right) \geq L_{\rho_k}^{(k)}\left(x^{k+1},u^k,v^k\right) + \frac{\sigma_k}{2} \| x^{k+1}- x^k \|_{Q_k}^2. \]
    Since $L_{\rho}^{(k)} ( x, \lambda,\mu)$ differs from $L_{\rho}(x,\lambda,\mu)$ by replacing the function $-h$ with an affine majorant, one has
    \[ L_{\rho}^{(k)} (x,\lambda,\mu) \geq L_{\rho} (x,\lambda,\mu) \]
    for all $(x,\lambda,\mu) \in \R^n \times \R^m \times \R^p$, in particular for $ \left(x^{k+1},u^k,v^k\right)$. Utilizing the definition of the Lagrange multiplier $\mu^{k+1}$ in \ref{alg: LagrangeMultiplier} yields
    \begin{align*}
        L_{\rho_k} \left(x^{k+1},u^k,v^k \right) =\ell_{k+1}.
    \end{align*}
    Putting all estimates together gives~\eqref{eqSys: NearMonotonicity_lowerCase}.

    Exploiting once again the definition of the Lagrange multiplier $\mu^k$ and afterwards the properties~\eqref{eq: constructionMot_vk} and~\eqref{eq: constructionMot_uk}, we obtain
    \begin{align*}
        L_{\rho_k} \left(x^k,u^k,v^k\right) =& \, L_{\rho_k} \left(x^k,\lambda^k,\mu^k \right) - \rho_{k-1} \| Ax^k-b \|^2 + \left(v^k-v^{k-1} \right)^T \left(Ax^k-b\right)\\
        & - \frac{1}{2\rho_k} \left[ \Big\| \max \left\{0, \lambda^k+ \rho_k c \left(x^k \right) \right\} \Big\|^2 - \| \lambda^k \|^2  - \left(  \Big\| \max \left\{ 0, u^k + \rho_k c \left(x^k \right) \right\} \Big\|^2 - \| u^k \|^2 \right) \right] \\
        \leq & \, L_{\rho_k} \left(x^k,\lambda^k,\mu^k\right)-\rho_{k-1} \| Ax^k-b \|^2
        -\frac{1}{2\rho_k}\left[ \Big\| \max \left\{0, \lambda^k+\rho_k c \left(x^k \right) \right\} \Big\|^2 - \| \lambda^k \|^2\right. \nonumber \\
        & -\left(  \left. \Big\| \max \left\{0, u^{k-1}+\rho_k c \left(x^k \right) \right\} \Big\|^2- \| u^{k-1} \|^2 \right) \right]
    \end{align*}
    which proves~\eqref{eqSys: NearMonotonicity_upperCase}.
    Note that the last inequality can actually be sharpened to an equality due to \cref{lem: auxiliarySequences}.
\end{proof}

The above lemma precisely yields that the sequence $\left\{ \ell_k \right\}_{k > \kappa}$ defined by~\eqref{eq: SequenceMonotonicity}
is monotonically decreasing if the penalty parameter $\rho_k$ stays constant from the index $\kappa$ on.
This observation is crucial to establish our convergence theory. Hence, we are interested under which circumstances the sequence of penalty parameters $\left\{\rho_k\right\}_{k \in \N_0}$ shows this behavior. In fact, the next result reveals that $\left\{ \sigma_k \right\}_{k \in \N_0}$ and accordingly, $\left\{ \rho_k \right\}_{k \in \N_0}$ remain bounded (or to be more precise: constant up to finitely many exceptions) if and only if \ref{alg: feasibilityCheck_unsucc} occurs only finitely many times.

\begin{lemma} \label{lem: proximalGrowth}
    Assume that the sequence $\left\{x^k\right\}_{k \in \N_0}$ stays bounded. Then the sequence of proximal parameters~$\left\{\sigma_k \right\}_{k \in \N_0}$ goes to infinity if and only if \ref{alg: feasibilityCheck_unsucc} is entered infinitely many times.

    In this case also the sequence of penalty parameters $\left\{\rho_k \right\}_{k \in \N_0}$ goes to infinity.
\end{lemma}
\begin{proof}
    Let us note first that as an immediate consequence of the construction of the proximal parameter~$\sigma_k$, the corresponding sequence is monotonically increasing.

    Obviously, if case \ref{alg: feasibilityCheck_unsucc} occurs only finitely many times, $\left\{\sigma_k \right\}_{k \in \N_0}$ eventually stays constant and hence bounded. Thus, suppose that \ref{alg: feasibilityCheck_unsucc} arises infinitely many times and denote the corresponding index set with $J \subseteq \N_0$.
    We proceed with distinguishing whether the set $K_0:= \left\{ k \in J \ \middle| \ x^{k+1}=x^k \right\}$ contains infinitely many indices or not. 

    %
    %
    %

    A. \emph{Assume that $K_0$ is an infinite index set.} By construction, one has $\sigma_{k+1}= \overline{\eta}\sigma_k$ for all $k \in K_0$. Hence, as $\overline{\eta}>1$ by our choice and $\left\{\sigma_k \right\}_{k \in \N_0}$ is monotonically increasing, the unbounded growth of $\left\{\sigma_k \right\}_{k \in \N_0}$ follows immediately.

    B. \emph{Assume that $K_0$ contains only a finite number of indices.} Let us denote  $k_0:= \max \left\{k \ \middle| \ k \in K_0\right\}$. Then one has $\sigma_{k+1}= \max \left\{ \frac{1}{\| x^{k+1}- x^k \|^{\alpha}}, \eta \sigma_k \right\}$ for all $k \in J, \ \linebreak[3] k > \nolinebreak k_0$.
    Now, we continue by distinguishing whether the monotonically decreasing and positive sequence $\left\{ \epsilon_k \right\}_{k \in \N_0}$ stays bounded away from zero or not.

    B.I) \emph{Assume that $\epsilon_k \rightarrow 0$ for $k \rightarrow \infty$.} Apparently, this requires that the counter $K$ exceeds $M$ infinitely many times after being reset to zero. Hence, we can find a subsequence $\left\{ \| x^{k+1}- x^k \|^{\alpha}\right\}_{\mathcal{K}}$, $ \mathcal{K}\subseteq J$, satisfying $\| x^{k+1}- x^k \|^{\alpha} < \nolinebreak \epsilon_k \linebreak[3]$ for all $k \in \mathcal{K}$. But due to our assumption this implies $ \| x^{k+1}- x^k \|^{\alpha} \rightarrow_{\mathcal{K}} 0$ which, in turn, leads to $\sigma_k \rightarrow \infty$ for $k \rightarrow \infty$ since clearly one has $\sigma_{k+1} \geq \frac{1}{\| x^{k+1}- x^k \|^{\alpha}}$ at least for all $k \in J, \ k > k_0$.

    B.II) \emph{Assume that $\left\{ \epsilon_k\right\}_{k \in \N_0} $ stays bounded away from zero.} By construction this means that there exists $\kappa \in \N_0$ such that $\epsilon_k = \epsilon_{\kappa}>0$ for all $k \geq \kappa$. Furthermore, this implies that the counters $K$ and $I_1$ cannot simultaneously exceed their bounds $M$ and $N$, respectively, infinitely many times after being reset to zero. In the following we differentiate whether the counter $K$ finally sticks to a value less than~$M$ or not.

    B.II.1. \emph{Assume that $K$ eventually remains constant with a value less than $M$.} Consequently, there exists $\overline{\kappa} \in \N_0$ ($\overline{\kappa}\geq \kappa$) such that $\| x^{k+1}- x^k \|^{\alpha} \geq \epsilon_k = \epsilon_{\kappa}$ for all $k \in J, \ k \geq \overline{\kappa}$ (where $\epsilon_{\kappa}$ denotes the lower bound of the sequence $\left\{ \epsilon_k \right\}_{k \in \N_0}$). But this implies $\sigma_{k+1} \geq \overline{\eta} \sigma_k$ for all $k \in J, \ k \geq \overline{\kappa}$.
    Hence,  it gets immediately clear that $\left\{\sigma_k \right\}_{k \in \N_0}$ goes to infinity as $J$ is an infinite index set, $\overline{\eta} > 1$ and $\left\{\sigma_k \right\}_{k \in \N_0}$ is monotonically increasing.

    B.II.2. \emph{Assume, on the contrary, that $K$ does not stick to a value less than $M$.} In this case both, $K$ remaining constant with a value larger than $M$ and $K$ tending to infinity, is possible. But due to our assumption for case~B.II) the counter $I_1$ necessarily has to stay at a constant level less than $N$ and does not get updated anymore. Again, we distinguish two cases, namely $K$ finally remaining constant as well and $K$ tending to infinity.

    B.II.2(a) \emph{Assume that $K$ eventually sticks to a value larger than $M$.} Then we can proceed as in \linebreak case~B.II.1.  to obtain the unbounded growth of $\left\{\sigma_k \right\}_{k \in \N_0}$.

    B.II.2(b) \emph{Assume that $K \rightarrow \infty$.} Still we have to distinguish two cases, namely whether the set $\overline{K}:= \nolinebreak \left\{ k \in J, \ k \geq \kappa \ \middle| \ \| x^{k+1}- x^k \|^{\alpha} \geq \epsilon_k = \epsilon_{\kappa} \right\}$ is finite or not.

    B.II.2.(b)(i) \emph{Assume that $\overline{K}$ is an infinite set.} By construction, one has $\sigma_{k+1} \geq \overline{\eta} \sigma_k$ for all $k \in \overline{K}$.
    Hence, along the lines of case~A., one obtains the unbounded growth of $\left\{\sigma_k\right\}_{k \in \N_0}$.

    B.II.2.(b)(ii) \emph{Assume that $\overline{K}$ contains only finitely many indices.} Then there exists $\hat{\kappa}\geq \nolinebreak \kappa$ such that $ \| x^{k+1}- x^k \|^{\alpha} < \epsilon_k = \epsilon_{\kappa}$ for all $k \in \nolinebreak[3] J, \ \linebreak[2]  k \geq \nolinebreak[2] \hat{\kappa}$. 
    Taking without loss of generality $\hat{\kappa} > k_0$ ensures 
    \[\sigma_{k+1}= \max \left\{ \frac{1}{\| x^{k+1}- x^k \|^{\alpha}}, \sigma_k \right\} \qquad \forall k\in J, \ k \geq \hat{\kappa}.\] 
    But, as our assumption for case~B.II.2. leads to $I_1$ finally remaining at a value less than $N$, the update formula eventually reads as $\sigma_{k+1}= \sigma_k$ for all $k \geq \overline{k}$ with $ \overline{k}:= \hat{\kappa}+N-1$.
    Note that, taking into account also the update formula in \ref{alg: feasibilityCheck_succ}, we do not have to restrict ourselves to indices $k \in J$ this time.
    Obviously, this yields $\sigma_k= \sigma_{\overline{k}}$ remaining constant for all $k \geq \overline{k}$. 
    Hence, also $ \left\{ \rho_k \right\}_{k \geq \overline{k}}$ is a constant sequence.
    \cref{lem: NearMonotonicity} therefore shows
    \begin{equation} \label{eq: MonotonicityConstantPenalty}
        \ell_k \geq \ell_{k+1} + \frac{\sigma_k}{2} \| x^{k+1} - x^k \|^2_{Q_k}
    \end{equation}
    for all $k > \overline{k}$ with $\ell_k, \ k \in \N,$ defined in the said lemma.
    As $\left\{x^k \right\}_{k \in \N_0}$ is assumed to be bounded and $ \left\{u^k\right\}_{k \in \N_0}, \ \left\{v^k\right\}_{k \in \N_0}$ are bounded due to \cref{lem: auxiliarySequences}, $\left\{\lambda^k\right\}_{k \in \N_0}$ and $ \left\{ \mu^k \right\}_{k \in \N_0}$ stay bounded, too. Together with the boundedness of $\left\{\rho_k \right\}_{k \in \N_0}$ and continuity of the functions $g, \ h$ and $c$, one obtains boundedness of $\left\{ \ell_k \right\}_{k \in \N_0}$, which due to monotonicity (from index $\overline{k}+1$ on) is therefore convergent. Taking the limit $k \rightarrow_{k > \overline{k}} \infty$ in~\eqref{eq: MonotonicityConstantPenalty} and keeping in mind that $\left\{\sigma_k \right\}_{k \in \N_0}$ eventually remains constant one yields $\| x^{k+1}-x^k \|_{Q_k} \rightarrow 0$ for $k \rightarrow \infty$. Furthermore, exploiting that 
    \[\| x^{k+1}-x^k \|_{Q_k} \geq \sqrt{\lambda_{\min} \left(Q_k\right)} \| x^{k+1}-x^k \| \geq \sqrt{\underline{\lambda}} \| x^{k+1}-x^k \| \]
    holds by our choice of the sequence $\left\{Q_k \right\}_{k \in \N_0}$ subject to the assumptions in~\ref{alg: initialisation}, $\| x^{k+1}-x^k\| \rightarrow 0$ for $k \rightarrow \infty$ follows.
    But as $\sigma_{k+1} \geq \frac{1}{\| x^{k+1} - x^k \|^{\alpha}}$ at least for all $k \in J, \ k > k_0$ this would imply an unbounded growth of $\left\{\sigma_k \right\}_{k \in \N_0}$, a contradiction. Hence, case~B.II.2(b)(ii) is a mere hypothetical case which can never occur. \\

    All cases together show the unbounded growth of $\left\{\sigma_k \right\}_{k \in \N_0}$ in the event of \ref{alg: feasibilityCheck_unsucc} being entered infinitely many times. 

    Furthermore, by the update formulas for $\rho_{k+1}$ in \ref{alg: feasibilityCheck_succ} and \ref{alg: feasibilityCheck_unsucc} it gets immediately clear that the sequence of penalty parameters $\left\{ \rho_k \right\}_{k \in \N_0}$ inherits the asymptotic behavior from the sequence of proximal parameters $\left\{\sigma_k \right\}_{k \in \N_0}$.
\end{proof} 

%

Whenever the penalty parameter goes to infinity in the context of safeguarded augmented Lagrangian methods, it is quite usual to require a suitable constraint qualification in order to ensure boundedness of the Lagrange multipliers. So does the presented method. Hence, in the following we introduce three qualification conditions for problem~\eqref{eq: problemFormulation} which turn out to be equivalent in our setting.

\begin{definition} \label{def: CQs}
    Let $\widehat{x}$ denote a feasible point of~\eqref{eq: problemFormulation}.
    \begin{enumerate}
        \item \label{def: CQs_mSCQ} We say that the \emph{modified Slater Constraint Qualification} (mSCQ) holds for~\eqref{eq: problemFormulation} in case there exists $\widetilde{x} \in C$ with 
            \begin{itemize}
                \item $c_i \left(\widetilde{x}\right) < 0 \ \forall i=1,\dots, m,$
                \item $A\widetilde{x}=b$ and
                \item $0 \in \interior \left( \ell (C) \right)$,
            \end{itemize}
            where $\ell$ is defined by~\eqref{eq: equalityConstraints}.
            Such a point $\widetilde{x}$ is then called a \emph{Slater point}.
        \item \label{def: CQs_EMFCQ} The point $\widehat{x}$ is said to satisfy the \emph{modified Extended Mangasarian Fromovitz Constraint Qualification} (mEMFCQ) whenever 
            \begin{itemize}
                \item $0 \in \interior \left( \ell (C) \right)$ and
                \item there exists $d \in T_C \left( \widehat{x} \right)$ such that $ Ad=0$ and $ c_i'\left( \widehat{x};d\right)<0 \ \forall i \in I\left( \widehat{x}\right)$ holds.
            \end{itemize}
            Thereby, $T_C \left( \widehat{x} \right)$ denotes the tangent cone from convex analysis and
            \[
                I \left( \widehat{x}\right) := \left\{ i \in \{1,\dots,m \} \ \middle| \ c_i \left( \widehat{x} \right) = 0 \right\}
            \]
            the active index set of $\widehat{x}$.
        \item We say that \emph{No Nonzero Abnormal Multiplier Constraint Qualification} (NNAMCQ) holds at $\widehat{x}$ if there is no nonzero vector $\left( \lambda, \mu \right) \in \R^m_+ \times \R^p$ satisfying
            \begin{equation} \label{eq: NNAMCQ}
                0 \in\left\{ A^T\mu \right\} + \sum_{i=1}^m \lambda_i \partial c_i\left( \widehat{x} \right) + N_C \left( \widehat{x} \right) \quad \text{and} \quad \lambda^Tc \left(\widehat{x}\right) = 0.
            \end{equation}
    \end{enumerate}
\end{definition}

Note that the two notions of Slater-type constraint qualifications used in this work (see \cref{def: SCQ} and \cref{def: CQs}~\ref{def: CQs_mSCQ}) are not equivalent to each other. In fact, none of them relates to the other. 

Adding the condition $0 \in \interior \left( \ell (C)\right)$ to the usual SCQ and EMFCQ, respectively, is required to obtain equivalence with NNAMCQ. Indeed, neglecting the said stipulation one cannot deduce that NNAMCQ is satisfied at some point $\widehat{x} \in \R^n$ for which EMFCQ (in the classical sense) holds or whenever SCQ (in the classical sense) is satisfied (whereas the corresponding reverse implications still hold true). This is illustrated by~\cref{ex: CQs}. 
For common definitions of SCQ and EMFCQ see, for example,~\cite{LeThi2014}.
Modifying SCQ in the proposed way has already been done in~\cite{Li1997,Zeng2023} in the context of vector optimization. For the sake of completeness, let us note that in~\cite{Zhang2014} the author points out that in a common non-convex and nonsmooth setting, NNAMCQ is in general weaker than EMFCQ whereas in the case of the functions involved being smooth (but still possibly non-convex) and abstract constraints being absent (i.e. $C= \R^n$) both constraint qualifications coincide. Moreover, it is well known that whenever considering convex feasible sets, the notions of SCQ and EMFCQ are equivalent (see e.g.~\cite{Stein2004}, although the feasible set deviates slightly from the one regarded in this work).

In the end, let us note that, technically speaking, we do not just add the condition $0 \in \interior \left( \ell (C)\right)$ to EMFCQ but replace the usual condition of the matrix~$A$ having full rank~$p$ by the proposed one. However, it is clear that the latter assumption gets implied by the prior one. Hence, mEMFCQ is indeed a stronger constraint qualification than the usual EMFCQ.

Let us now turn to the announced example, illustrating that SCQ and also EMFCQ in the classical sense do not suffice to guarantee NNAMCQ.

\begin{example}\label{ex: CQs} (SCQ or, equivalently, EMFCQ does not imply NNAMCQ)\\
    Consider problem~\eqref{eq: problemFormulation} with $n=2$, ($f$ an arbitrary DC~function) and
    \begin{align*}
        &c\left(x_1,x_2\right):=-x_2,\\
        &\ell \left(x_1,x_2\right):=x_1,\\
        &C:=\left\{ \left(0,x_2\right)^T \ \middle| \ x_2 \geq 0\right\}.
    \end{align*}
    Then, clearly $\widetilde{x}:=(0,1)^T$ is a Slater point of~\eqref{eq: problemFormulation} in the classical sense, that means a feasible point satisfying the inequality constraint(s) strictly (for a definition of the common SCQ see e.g.~\cite{LeThi2014,Mordukhovich2013}). In addition, one has $A:=(1,0)$ which is of full rank. Moreover, at the feasible point $\widehat{x}:=(0,0)^T$ one has $T_C \left(\widehat{x}\right)=C$. Hence, choosing $d:=(0,1)^T \in T_C \left(\widehat{x}\right)$ one obtains $Ad=0$ as well as $c'\left(\widehat{x};d\right)=(0,-1)\cdot (0,1)^T=-1<0$ (note that the inequality constraint is active at~$\widehat{x}$). Altogether this shows that EMFCQ in the classical sense holds at~$\widehat{x}$ (for a definition of EMFCQ of the common type see again~\cite{LeThi2014}). However, as $N_C \left(\widehat{x} \right) = \left\{x \in \R^2 \ \middle| \ x_2 \leq 0\right\}$, choosing $\lambda=0 \in \R_+, \ \mu=1$ and $\nu=(-1,0)^T\in N_C \left( \widehat{x} \right)$ one obtains $ 0 \in \left\{A^T\mu + \lambda \nabla c \left(\widehat{x}\right)\right\}+N_C \left(\widehat{x}\right)$ together with $\lambda c \left(\widehat{x}\right)=0$, showing that a nonzero abnormal multiplier of~\eqref{eq: problemFormulation} exists. Thus, NNAMCQ is not satisfied at~$\widehat{x}$.

    Let us note for the sake of completeness, that here, $\ell(C)=\{0\}$ and hence, $0 \notin \nolinebreak \interior\left(\ell(C)\right)$. 
    Besides, it holds $\widetilde{x} \in \ri (C) =\left\{ \left(0,x_2\right)^T \ \middle| \ x_2>0 \right\}$. Thus, also StCQ from \cref{def: SCQ} does not imply NNAMCQ.
\end{example}

Now, we want to verify that modifying the constraint qualifications in the proposed way indeed ensures not only equivalence of mSCQ and mEMFCQ but also with NNAMCQ in the sense below. 
To our knowledge the following two lemmata have not been proven yet but use techniques which are quite common for proving equivalence of constraint qualifications that are similar to the presented ones (see e.g.~\cite{Li1997,Stein2004}). Hence, we put their proofs in \cref{append: proof_CQs}. 

\begin{lemma} \label{lem: CQs_stepOne}
    Let $\widehat{x}$ be feasible for~\eqref{eq: problemFormulation}. 
    Then, the following statements are equivalent:
    \begin{enumerate}
        \item \label{lem: CQ_mEMFCQ_one} mEMFCQ holds in~$\widehat{x}$.
        \item \label{lem: CQ_NNAMCQ_one} NNAMCQ holds in~$\widehat{x}$.
    \end{enumerate}
\end{lemma}

\begin{lemma}\label{lem: CQs}
    The subsequent assertions are equivalent for~\eqref{eq: problemFormulation}: 
    \begin{enumerate} 
        \item \label{lem: CQ_mSCQ} mSCQ is satisfied for~\eqref{eq: problemFormulation}.
        \item \label{lem: CQ_mEMFCQ_NNAMCQ_all} mEMFCQ or, equivalently, NNAMCQ holds in all feasible points of~\eqref{eq: problemFormulation}.
    \end{enumerate}
\end{lemma}

Now we are in position to prove our final convergence result.

\begin{theorem}
    Assume that $\left\{x^k\right\}_{k \in \N_0}$ remains bounded and the feasible set of~\eqref{eq: problemFormulation} is nonempty. Then, the following statements hold:
    \begin{enumerate}
        \item \label{thm: convergence_bounded} If $\left\{\sigma_k \right\}_{k \in \N_0}$ (or, equivalently, $ \left\{\rho_k \right\}_{k \in \N_0}$) remains bounded, then, for each accumulation point $x^\ast $ of $\left\{x^k\right\}_{k \in \N_0}$ there exist $\lambda^\ast, \ \mu^\ast$ such that $\left(x^\ast, \lambda^\ast,\mu^\ast\right)$ is a generalized KKT~point of~\eqref{eq: problemFormulation}. In particular, there exists $\overline{K} \subseteq \N_{0}$ such that for $k \rightarrow_{\overline{K}} \infty$ it holds \mbox{ $\left(x^k,\lambda^k,\mu^k \right) \rightarrow_{\overline{K}} \left(x^\ast, \lambda^\ast,\mu^\ast\right)$}.
        \item \label{thm: convergence_unbounded} If $ \left\{\sigma_k \right\}_{k \in \N_0}$ (or, equivalently, $ \left\{\rho_k \right\}_{k \in \N_0}$) goes to infinity, then, there exists an accumulation point~$x^\ast$ of $\left\{x^k\right\}_{k \in \N_0}$ which is feasible and, provided that mSCQ holds for~\eqref{eq: problemFormulation}, together with some $\lambda^\ast, \ \mu^\ast$ satisfies the generalized KKT~conditions for~\eqref{eq: problemFormulation}. In particular, there exists $\overline{K} \subseteq \N_{0}$ such that $\left(x^k,\lambda^k,\mu^k \right) \rightarrow_{\overline{K}} \left(x^\ast, \lambda^\ast,\mu^\ast\right)$.
    \end{enumerate}
\end{theorem}
\begin{proof}
    \ref*{thm: convergence_bounded} First note that, according to \cref{lem: proximalGrowth}, $\left\{\sigma_k \right\}_{k \in \N_0}$ remains bounded corresponds to~\ref{alg: feasibilityCheck_unsucc} being entered only finitely many times. Hence, there exists some index $ \kappa \in \N_0$ from which on only \ref{alg: feasibilityCheck_succ} occurs. In particular, this yields
    \begin{equation} \label{eq: constantPenaltyProximal}
        \sigma_k=\sigma_{\kappa}, \quad \rho_k = \rho_{\kappa} \quad \forall k \geq \kappa
    \end{equation}
    and, noting that $\theta \in (0,1)$,
    \begin{equation} \label{eqSys: limitFeasibility}
        Ax^k-b \rightarrow 0, \qquad
        \min \left\{ -c \left(x^{k+1}\right),\frac{u^k}{\rho_k} \right\} \rightarrow 0
    \end{equation}
    as $k \rightarrow \infty$.
    Now, let $x^\ast$ denote an arbitrary accumulation point of $\left\{x^k\right\}_{k \in \N_0}$ and $\overline{K}\subseteq \N_0$ an index set with $x^{k+1} \rightarrow_{\overline{K}} x^\ast$. 
    Since $x^k \in C$ for all $k$, we obtain $x^\ast \in C$ by the assumed closedness of the set $C$. Furthermore, since $\left\{u^k\right\}_{k \in \N_0}$ is bounded according to \cref{lem: auxiliarySequences}, we may assume without loss of generality the convergence of the subsequence $\left\{u^k\right\}_{k \in \overline{K}}$ to some $u^\ast \in \R^m_+$ which inherits the non-negativity from the members of the sequence. Hence, taking the limit $k \rightarrow_{\overline{K}} \infty$ of $\left\{ Ax^{k+1}-b \right\}$ and $ \min \left\{ - c \left(x^{k+1} \right), \frac{u^k}{\rho_k} \right\}$ and using~\eqref{eqSys: limitFeasibility} as well as~\eqref{eq: constantPenaltyProximal} yields both feasibility of $x^\ast$ and complementary slackness of $\left( x^\ast, \frac{u^\ast}{\rho_{\kappa}} \right)$ with $\frac{u^\ast}{\rho_{\kappa}}\geq 0$. This implies that $\left( x^\ast, u^\ast \right)$ satisfies complementary slackness as well.

    Similarly, by the boundedness of $\left\{v^k \right\}_{k \in \N_0}$ due to \cref{lem: auxiliarySequences}, we may assume without loss of generality that $v^k \rightarrow_{\overline{K}} v^\ast$ for some $v^\ast \in \R^p$. To verify that $\left(x^\ast, u^\ast, v^\ast \right)$ is a generalized KKT~point of~\eqref{eq: problemFormulation}, it remains to show~\eqref{eqSys: KKT_optimality}. To this end, we exploit \cref{lem: NearMonotonicity} which, together with~\eqref{eq: constantPenaltyProximal}, demonstrates
    \begin{equation} \label{eq: monotonicityEllWithProx}
        \ell_k \geq \ell_{k+1} + \frac{\sigma_{\kappa}}{2} \| x^{k+1} - x^k \|_{Q_k}^2 \quad \forall k > \kappa
    \end{equation}
    with $\ell_k, \ k \in \N,$ defined in the said lemma.
    Noting that the update formula for $\lambda^{k+1}$ in~\ref{alg: LagrangeMultiplier} can be written as
    \[ \lambda^{k+1} = u^k - \rho_k \min \left\{-c \left(x^{k+1}\right),\frac{u^k}{\rho_k} \right\}, \]
    taking the limit $k \rightarrow_{\overline{K}} \infty$ for both $\lambda^{k+1}$ and $\mu^{k+1}$ in~\ref{alg: LagrangeMultiplier} gives $\mu^{k+1} \rightarrow_{\overline{K}} v^\ast$ and $ \lambda^{k+1} \rightarrow_{\overline{K}} \nolinebreak u^\ast $, respectively, by using~\eqref{eqSys: limitFeasibility} and~\eqref{eq: constantPenaltyProximal}. Consequently, the monotonically decreasing sequence $\left\{ \ell_k \right\}_{k > \kappa}$ is convergent on a subsequence. Therefore, the whole sequence $\left\{ \ell_k \right\}_{k \in \N}$ converges. 
    This together with~\eqref{eq: monotonicityEllWithProx} implies $ \frac{\sigma_k}{2}\| x^{k+1}-x^k \|_{Q_k}^2 \rightarrow 0$ for $k \rightarrow \infty$. 
    Since
    \[ \frac{\sigma_k}{2} \| x^{k+1}-x^k \|_{Q_k}^2 \geq \frac{\sigma_{\kappa}}{2} \lambda_{\min} \left(Q_k \right) \| x^{k+1}-x^k \|^2 \geq \frac{\sigma_{\kappa}}{2} \underline{\lambda} \| x^{k+1}-x^k \|^2 \qquad \forall k \geq \kappa \]
    due to~\eqref{eq: constantPenaltyProximal} and our assumptions on the sequence $ \left\{Q_k \right\}_{k \in \N_0}$, we obtain $\| x^{k+1}-x^k \| \rightarrow \nolinebreak 0$ for $k \rightarrow \infty$. Thus, both $\left\{x^{k+1} \right\}_{k \in \overline{K}}$ and $\left\{x^k \right\}_{k \in \overline{K}}$ converge to $x^\ast$. 
    As $x^{k+1}$ is the (unique) minimizer of \linebreak $L^{(k)}_{\rho_k} \left( \cdot, u^k,v^k \right) + \nolinebreak \frac{\sigma_k}{2} \| \cdot - x^k\|^2_{Q_k}$ over the convex set $C$, the optimality condition from \cref{thm: cvx_OptCond} together with the calculus rules from \cref{thm: cvxSubdiff_calculus} and the multiplier updates from~\ref{alg: LagrangeMultiplier} give
    \begin{align}
        0 \in & \, \partial \left( L_{\rho_k}^{(k)} \left(\cdot,u^k,v^k\right) + \frac{\sigma_k}{2} \| \cdot - x^k \|_{Q_k}^2 \right) \left(x^{k+1} \right) + N_C \left(x^{k+1} \right) \nonumber \\
        =& \, \partial g \left(x^{k+1} \right) + \left\{-s^k + A^T \left(v^k+\rho_k \left(Ax^{k+1}-b \right)\right) \right\} + \sum_{i=1}^m \max \left\{ 0, u_i^k + \rho_k c_i\left(x^{k+1}\right) \right\} \partial c_i \left( x^{k+1} \right)\nonumber \\
        & + \left\{ \sigma_k Q_k \left( x^{k+1}-x^k \right) \right\} + N_C \left( x^{k+1} \right) \label{eq: optCondSubproblem_long}\\
        =& \, \partial g \left(x^{k+1} \right) + \left\{-s^k+A^T\mu^{k+1} \right\} + \sum_{i=1}^m \lambda_i^{k+1} \partial c_i \left( x^{k+1} \right) + \left\{ \sigma_k Q_k \left(x^{k+1}-x^k \right) \right\} + N_C \left(x^{k+1} \right). \label{eq: optcondSubproblem}
    \end{align}
    Consequently, there exist $t^{k+1} \in \partial g \left(x^{k+1}\right)$ and $ \omega_i^{k+1} \in \partial c_i \left(x^{k+1} \right)$ for each $i=1,\dots,m$ and $k \in \N_0$ with
    \begin{equation} \label{eq: optCond_Cone}
        - \bigg( t^{k+1}-s^k + A^T \mu^{k+1} + \sum_{i=1}^m \lambda_i^{k+1} \omega_i^{k+1} + \sigma_k Q_k \left( x^{k+1}-x^k \right) \bigg) \in N_C \left(x^{k+1} \right).
    \end{equation}
    Due to the local boundedness of the convex subdifferential (cf.~\cref{thm: cvxSubdiff_localBoundedClosed}\ref{thm: cvxSubdiffLocalBounded}) in combination with the assumed boundedness of the sequence $\left\{x^k\right\}_{k \in \N_0}$, the sequences $\left\{t^{k+1} \right\}_{k \in \N_0}, \ \left\{s^k \right\}_{k \in \N_0}, \ \left\{\omega_i^{k+1} \right\}_{k \in \N_0}, \ \linebreak[2] i=\nolinebreak 1,\dots,m,$ are bounded. Therefore, without loss of generality one may assume convergence of each sequence on the corresponding subsequence with indices $k \in \overline{K}$ to some limits $t^\ast \in \partial g \left(x^\ast \right), \ \linebreak[3] s^\ast \in \nolinebreak \partial h \left(x^\ast\right) \linebreak[3]$ and $\omega_i^\ast \in \partial c_i \left(x^\ast \right), \ i=1,\dots, m$, respectively. Thereby, the inclusions are guaranteed by the closedness of the graph of the convex subdifferential (cf.~\cref{thm: cvxSubdiff_localBoundedClosed}\ref{thm: cvxSubdiffClosedGraph}) and convergence of $\left\{x^{k+1}\right\}_{k \in \overline{K}}$ and $\left\{x^k \right\}_{k \in \overline{K}}$ to $x^\ast$. 
    Note that $\| x^{k+1}-x^k\| \rightarrow 0$ for $k \rightarrow \infty$ implies also $\sigma_k Q_k \left(x^{k+1}-x^k \right) \rightarrow 0$ since one has
    \begin{align*}
        \sigma_k \| Q_k \left(x^{k+1}-x^k \right) \| & \leq \sigma_{\kappa} \| Q_k \|  \| x^{k+1}-x^k \| \\
        & = \sigma_{\kappa} \lambda_{\max} \left(Q_k \right) \| x^{k+1}-x^k \| \\
        & \leq \sigma_{\kappa} \overline{\lambda} \| x^{k+1}-x^k \|
    \end{align*}
    for all $k \geq \kappa$.
    Hence, taking the limit $k \rightarrow_{\overline{K}} \infty$ in~\eqref{eq: optCond_Cone} and exploiting the robustness~\eqref{eq: normalCone_Robustness} of the normal cone from convex analysis, we obtain
    \[ - \bigg( t^\ast -s^\ast + A^T v^\ast + \sum_{i=1}^m u_i^\ast \omega_i^\ast \bigg) \in N_C \left(x^\ast \right), \]
    hence,
    \[ 0 \in \partial g \left(x^\ast\right) - \partial h \left(x^\ast\right) + \left\{A^T v^\ast\right\} + \sum_{i=1}^m u_i^\ast \partial c_i \left(x^\ast\right) + N_C \left(x^\ast\right). \]
    Thus, the claim follows with $\lambda^\ast := u^\ast$ and $\mu^\ast := v^\ast$.\\

    \ref*{thm: convergence_unbounded} Recall from \cref{lem: proximalGrowth} that $\left\{\sigma_k \right\}_{k \in \N_0}$ tends to infinity if and only if \ref{alg: feasibilityCheck_unsucc} occurs infinitely many times.

    \emph{Step 1}: First, we show that the full sequence $ \left\{ \| x^{k+1} - x^k \| \right\}_{k \in \N_0}$ converges to zero. Since $\left\{x^k \right\}_{k \in \N_0}$ is bounded, $\left\{\| x^{k+1} - x^k \| \right\}_{k \in \N_0}$ is also bounded. Hence, it suffices to show that each convergent subsequence of $\left\{\| x^{k+1} - x^k \| \right\}_{k \in \N_0}$ converges to zero. Therefore, choose an arbitrary convergent subsequence $\left\{\| x^{k+1} - x^k \| \right\}_{k \in \overline{K}}, \ \linebreak[2] \overline{K} \subseteq \nolinebreak[3] \N_0$, and assume without loss of generality that also $\left\{x^{k+1} \right\}_{k \in \overline{K}}$ and $\left\{x^k\right\}_{k \in \overline{K}}$ converge to some limits $x^\ast \in \R^n$ and $\widetilde{x}^\ast \in \R^n$, respectively.
    Actually, one even has $x^\ast, \widetilde{x}^\ast \in C$ since each iterate $x^k$ belongs to the set $C$ by construction and $C$ is assumed to be closed.  
    By~\eqref{eq: optCondSubproblem_long}, for each $k \in \N_0$, there exist subgradients $t^{k+1} \in \partial g \left( x^{k+1} \right), \  \omega_i^{k+1} \in \nolinebreak \partial c_i \left(x^{k+1} \right)$ for all $i=1,\dots, m$ such that
    \begin{equation}
        \begin{aligned}[t]
            - & \bigg(  t^{k+1}-s^k+A^T \left(v^k+\rho_k \left(Ax^{k+1}-b \right) \right)+ \sum_{i=1}^m \max \left\{0,u_i^k+\rho_k c_i \left(x^{k+1} \right) \right\} \omega_i^{k+1}  \\
            &\quad + \sigma_k Q_k \left(x^{k+1}-x^k \right) \bigg) \in N_C \left(x^{k+1} \right).\label{eq: optCondSubrpoblem_longVec}
        \end{aligned}
    \end{equation}
    As $N_C \left(x^{k+1} \right)$ is a cone, we also have 
    \begin{equation} \label{eq: optCond_scaledWithSigma}
        \begin{aligned}[t]
            - \frac{1}{\sigma_k} & \bigg(  t^{k+1}-s^k+A^T \left(v^k+\rho_k \left(Ax^{k+1}-b \right) \right)+ \sum_{i=1}^m \max \left\{0,u_i^k+\rho_k c_i \left(x^{k+1} \right) \right\} \omega_i^{k+1}  \\
            & \quad + \sigma_k Q_k \left(x^{k+1}-x^k \right) \bigg) \in N_C \left(x^{k+1} \right).
        \end{aligned}
    \end{equation}
    Since $\left\{Q_k \right\}_{k \in \N_0}$ is bounded by assumption, we may assume without loss of generality convergence of the subsequence $\left\{ Q_k \right\}_{k \in \overline{K}}$ to some symmetric positive definite matrix $Q^\ast \in \R^ {n \times n }$.
    Taking the limit $k \rightarrow_{\overline{K}} \infty$ in~\eqref{eq: optCond_scaledWithSigma}, we get
    \[ Q^\ast \left( \widetilde{x}^\ast - x^\ast \right) \in N_C \left(x^\ast \right) \]
    exploiting that $\left\{\sigma_k \right\}_{k \in \N_0}$ itself tends to infinity, $\left\{t^{k+1} \right\}_{k \in \N_0}, \ \left\{s^k\right\}_{k \in \N_0}, \ \left\{\omega_i^{k+1} \right\}_{k \in \N_0}, \linebreak[3] \ i= \nolinebreak 1,\dots,m,$ are bounded due to the boundedness of $\left\{x^k\right\}_{k \in \N_0}$ and the local boundedness property of the convex subdifferential, $\left\{u^k\right\}_{k \in \N_0}, \ \linebreak[3] \left\{v^k \right\}_{k \in \N_0}$ are bounded by \cref{lem: auxiliarySequences}, $\frac{\rho_k}{\sigma_k}= \left( \sigma_k \right)^{ \gamma-1}$ with $\gamma<1$ and hence tends to zero as $\left\{\sigma_k \right\}_{k \in \N_0}$ goes to infinity, $c_i$ is continuous for each $i=1,\dots,m$, and last but not least the normal cone from convex analysis is robust.
    The definition of the convex normal cone yields $ \left( Q^\ast \left( \widetilde{x}^\ast - x^\ast \right)\right)^T \left(z-x^\ast\right) \leq 0$ for all $z \in C$. Taking $z:=\widetilde{x}^\ast \in C$ and exploiting the fact that $Q^\ast$ is (symmetric) positive definite, $x^\ast = \widetilde{x}^\ast$ follows.
    Consequently, we have $ \| x^{k+1}-x^k \| \rightarrow_{\overline{K}} 0$.

    \emph{Step 2}: Next, we deduce that 
    \begin{equation} \label{eq: proofConvergence_LambdaSubsequence}
        \sigma_k Q_k \left(x^{k+1}-x^k \right) \rightarrow 0
    \end{equation} 
    holds at least on a subsequence. 
    By Step~1, we know that $\| x^{k+1}-x^k \| \rightarrow 0$ for $k \rightarrow \infty$.
    From the proof of \cref{lem: proximalGrowth}, however, we see that 
    $\| x^{k+1}-x^k \|^{\alpha} \rightarrow 0$
    can occur only if we have $x^{k+1}=x^k$ for an infinite number of indices or $\epsilon_k \rightarrow 0$. (As for each distinction of case~B.II), which can occur, the sequence $\left\{  \| x^{k+1}-x^k \|^{\alpha} \right\}_{k \in \N_0}$ would be bounded away from zero (at least on a subsequence).)
    In the first case, that is subsequent iterates coincide infinitely many times, \eqref{eq: proofConvergence_LambdaSubsequence} trivially holds true for the respective subsequence.
    Thus, let us consider the second case.
    Since $\epsilon_k$ is necessarily reduced infinitely many times and each decrease involves the counter~$I_1$ reaching or exceeding $N$ after being reset to zero, we can extract a subsequence such that $\sigma_{k+1}=\frac{1}{\| x^{k+1}-x^k \|^{\alpha}}$ holds. Let us denote the corresponding infinite index set by $\Lambda'$. Then, for each $k \in \Lambda'$ one has
    \begin{align*}
        \sigma_k \| Q_k \left( x^{k+1}-x^k \right) \| &\leq \sigma_{k+1} \| Q_k \| \| x^{k+1}-x^k \|\\
        & = \lambda_{\max} \left(Q_k \right) \| x^{k+1}-x^k \|^{1-\alpha}\\
        & \leq \overline{\lambda} \| x^{k+1}-x^k \|^{1- \alpha}, 
    \end{align*}
    where we additionally exploit the monotonicity of $\left\{\sigma_k \right\}_{k \in \N_0}$ as well as our postulations for $\left\{Q_k \right\}_{k \in \N_0}$ in~\ref{alg: initialisation}.
    Hence, taking advantage of $\| x^{k+1}-x^k \|\rightarrow 0$ together with $\alpha \in (0,1)$ gives
    $ \sigma_k Q_k \left( x^{k+1}-x^k\right) \rightarrow_{\Lambda'}\nolinebreak0$.

    In the following, let
    \[ \Lambda \in \left\{ S \subseteq \N \ \middle| \ \eqref{eq: proofConvergence_LambdaSubsequence} \text{ holds for } k \rightarrow_S \infty \right\}. \]

    \emph{Step 3}:
    Here, we prove the existence of a feasible accumulation point of $\left\{x^k\right\}_{k \in \N_0}$. 
    Since $\left\{x^k\right\}_{k \in \N_0}$ is bounded, we may assume without loss of generality that $x^{k+1} \rightarrow_{\Lambda} \nolinebreak \widehat{x}$ for some $\widehat{x} \in \R^n$. Note that $\widehat{x} \in C$ since $C$ is a closed set and $x^k \in C $ holds for all $k \in \N_0$. Moreover, $\left\{x^k\right\}_{k \in \Lambda}$ also converges to $\widehat{x}$ since $\| x^{k+1}-x^k \| \rightarrow 0$. By~\eqref{eq: optCondSubrpoblem_longVec} and $N_C \left(x^{k+1} \right)$ being a cone, we obtain
    \begin{align*}
        - \frac{1}{\rho_k} & \bigg(  t^{k+1}-s^k+A^T \left(v^k+\rho_k \left(Ax^{k+1}-b \right) \right)+ \sum_{i=1}^m \max \left\{0,u_i^k+\rho_k c_i \left(x^{k+1} \right) \right\} \omega_i^{k+1}\\
        & \quad + \sigma_k Q_k \left(x^{k+1}-x^k \right) \bigg) \in N_C \left(x^{k+1} \right). 
    \end{align*}
    Utilizing the previously discussed boundedness of $\left\{w_i^{k+1} \right\}_{k \in \N_0}$ for each $i=1,\dots, m$, we may again assume without loss of generality the convergence of $ \left\{\omega_i^{k+1} \right\}_{k \in \Lambda}$ to some $\widehat{\omega}_i \in \partial c_i \left( \widehat{x} \right)$ for each $i=1,\dots,m$ by the closedness property of the convex subdifferential.
    Now, taking the limit $k \rightarrow_{\Lambda} \infty$ and exploiting the facts that $\rho_k \rightarrow \infty$ according to \cref{lem: proximalGrowth} due to the assumed unboundedness of $\left\{\sigma_k \right\}_{k \in \N_0}$ and $\sigma_k Q_k \left( x^{k+1}-x^k \right) \rightarrow_{\Lambda} 0$ by Step~2, the boundedness of $\left\{t^{k+1} \right\}_{k \in \N_0}, \ \left\{s^k \right\}_{k \in \N_0}, \ \linebreak[2] \left\{v^k\right\}_{k \in \N_0}, \ \left\{ u^k \right\}_{k \in \N_0}$, the continuity of each $c_i$, and the robustness of the normal cone, we obtain 
    \[ - \bigg( A^T \left(A\widehat{x}-b \right)+ \sum_{i=1}^m \max \left\{0,c_i \left(\widehat{x} \right) \right\} \widehat{\omega}_i \bigg) \in N_C \left( \widehat{x}\right). \] 
    But this shows
    \begin{align*}
        0 \in & \left\{ A^T \left(A \widehat{x}-b \right)\right\} + \sum_{i=1}^m \max \left\{ 0, c_i \left(  \widehat{x} \right) \right\} \partial
        c_i \left( \widehat{x} \right) + N_C \left( \widehat{x} \right) \\
        =  & \, \partial \left( \frac{1}{2} \| A \cdot - b \|^2 + \frac{1}{2} \| \max \left\{0, c ( \cdot ) \right\} \|^2 + \delta_C ( \cdot ) \right) \left( \widehat{x}\right) 
    \end{align*}
    where the equality follows from some elementary calculus rules for the convex subdifferential (cf.~\cref{thm: cvxSubdiff_calculus}) noting, in particular, that all functions $c_i$ are real-valued. However, due to \cref{thm: cvx_OptCond} this in turn is equivalent to
    \[ \widehat{x} \in \argmin_{x \in \R^n } \left\{ \frac{1}{2} \| Ax-b \|^2 + \frac{1}{2} \| \max \left\{0, c(x) \right\} \|^2+ \delta_C(x)\right\}. \]
    Since the feasible set of~\eqref{eq: problemFormulation} is nonempty, the minimal value of the objective function above is zero and hence, we then obtain feasibility of $\widehat{x}$ for~\eqref{eq: problemFormulation}.

    \emph{Step 4}:
    In the next step we will show that under the assumption of mSCQ the sequence $ \left\{\left(\lambda^{k+1}, \mu^{k+1} \right) \right\}_{k \in \Lambda}$ of Lagrange multipliers stays bounded. To this end, let us denote with $I\left( \widehat{x} \right) := \left\{ i \in \{1,\dots, m \} \ \middle| \ c_i \left( \widehat{x} \right) = 0 \right\}$ the active index set of $\widehat{x}$. Then, for each index $i \notin I \left( \widehat{x} \right)$  by feasibility of $\widehat{x}$, continuity of $c_i$ and $x^{k+1} \rightarrow_{\Lambda} \widehat{x}$, we have $c_i \left(x^{k+1} \right) < 0$ for all sufficiently large $k \in \Lambda$. Considering the update formula of $\lambda^{k+1}$ in~\ref{alg: LagrangeMultiplier} and taking into account that $\left\{u^k\right\}_{k \in \N_0}$ remains bounded according to \cref{lem: auxiliarySequences}, whereas $\left\{\rho_k \right\}_{k \in \N_0}$ tends to infinity, we obtain $\lambda_i^{k+1} = \nolinebreak 0$ for all sufficiently large $k \in \Lambda$ and $i \notin I \left( \widehat{x} \right)$. Hence, $\lambda_i^{k+1} \rightarrow_{\Lambda} 0=: \widehat{\lambda}_i$ follows for each $i \notin I \left( \widehat{x} \right)$. As a consequence, for all sufficiently large $k \in \Lambda$, \eqref{eq: optCondSubrpoblem_longVec} reduces to
    \begin{equation} \label{eq: optcondSubproblemVec}
        - \bigg( t^{k+1}-s^k + A^T \mu^{k+1}+ \sum_{i \in I \left(\widehat{x}\right)} \lambda_i^{k+1} \omega_i^{k+1} + \sigma_k Q_k \left(x^{k+1}-x^k \right) \bigg) \in N_C \left(x^{k+1} \right)
    \end{equation}
    using the definition of Lagrange multipliers. Now, assume, by contradiction, that the sequence \linebreak $\left\{ \left(\lambda_i^{k+1}, \mu_j^{k+1} \right)_{i \in I \left( \widehat{x} \right),j=1,\dots,p} \right\}_{ k \in \Lambda}$ is unbounded and at the same time mSCQ holds for~\eqref{eq: problemFormulation}. In view of \cref{lem: CQs} this is equivalent to $\widehat{x}$ satisfying NNAMCQ. 
    Without loss of generality, we may assume that $ \frac{\left(\lambda_{I \left(\widehat{x}\right)}^{k+1}, \mu^{k+1} \right)}{\| \left(\lambda_{I \left(\widehat{x}\right)}^{k+1}, \mu^{k+1} \right) \|} \rightarrow_{\Lambda} \nolinebreak \left(\widetilde{\lambda}_{I \left( \widehat{x}\right)}, \widetilde{\mu} \right) $ with $\left(\widetilde{\lambda}_{I \left(\widehat{x}\right)}, \widetilde{\mu} \right) \neq 0$ and $ \widetilde{\lambda}_i \geq 0$ for all $i \in I \left( \widehat{x} \right)$ by the (component-wise) non-negativity of the sequence $\left\{\lambda^{k+1} \right\}_{k \in \N_0}$.
    Setting $\widetilde{\lambda}_i=0$ for all $i \notin I \left( \widehat{x} \right)$, $\widetilde{\lambda}$ together with $  \widehat{x} $ satisfy the complementary slackness condition $\widetilde{\lambda}_i \geq 0, \ c_i \left( \widehat{x} \right) \leq 0$ and $\widetilde{\lambda}_i c_i \left( \widehat{x}\right)=0$ for all $i=1,\dots, m$.
    Exploiting once more that $N_C \left(x^{k+1}  \right)$ is a cone, \eqref{eq: optcondSubproblemVec} gives
    \[
        - \frac{1}{\Big\| \left( \lambda_{I \left(\widehat{x}\right)}^{k+1}, \mu^{k+1} \right) \Big\|}  \bigg( t^{k+1}-s^k + A^T \mu^{k+1}+ \sum_{i \in I \left(\widehat{x}\right)} \lambda_i^{k+1} \omega_i^{k+1}  + \sigma_k Q_k \left(x^{k+1}-x^k \right) \bigg) \in N_C \left(x^{k+1} \right)
    \]
    for all sufficiently large $k \in \Lambda$. Taking $k \rightarrow_{\Lambda} \infty$ yields
    \[ - \bigg(A^T  \widetilde{\mu}+ \sum_{i \in I \left(\widehat{x} \right)} \widetilde{\lambda}_i \widehat{\omega}_i \bigg) \in N_C \left( \widehat{x} \right) \] 
    where one utilizes that $ \left\{ \left( \lambda^{k+1}_{I \left( \widehat{x}\right)}, \mu^{k+1} \right)\right\}_{k \in \Lambda}$ is unbounded whereas, following our previous discussion, $ \left\{t^{k+1} \right\}_{k \in \N_0}$ and $ \left\{ s^k \right\}_{k \in \N_0}$ stay bounded, \mbox{$\omega_i^{k+1} \rightarrow_{\Lambda} \widehat{\omega} \in \nolinebreak \partial c_i \left( \widehat{x} \right) $} for all $i=1,\dots, m$, $\sigma_k Q_k \left(x^{k+1}-x^k \right) \rightarrow_{\Lambda} 0$ and the normal cone from convex analysis is robust. But as NNAMCQ is assumed to hold at $\widehat{x}$, this would imply $ \left( \widetilde{\lambda}_{I \left( \widehat{x} \right)}, \widetilde{\mu} \right)=0$, a contradiction. Hence, the boundedness of the sequence $ \left\{ \left( \lambda^{k+1}, \mu^{k+1} \right) \right\}_{k \in \Lambda}$ is proven.

    \emph{Step 5}:
    Finally, we prove that our accumulation point $\widehat{x}$ together with some suitable Lagrange multipliers is indeed a generalized KKT~point of~\eqref{eq: problemFormulation}. By Step~4 we may assume without loss of generality that the sequences $\left\{\lambda^{k+1}_{I \left( \widehat{x}\right)}\right\}_{k \in \Lambda}, \ \left\{\mu^{k+1}\right\}_{k \in \Lambda}$ converge to some $\widehat{\lambda}_{I \left( \widehat{x}\right)}$ and $\widehat{\mu}$, respectively, where one has $\widehat{\lambda}_{I \left( \widehat{x}\right)} \geq 0$ by the non-negativity of the sequence $\left\{\lambda^{k+1} \right\}_{k \in \N_0}$. Note that already $\lambda^{k+1}_i \rightarrow_{\Lambda} 0$ for every $i \notin I \left( \widehat{x} \right)$ has been shown which, together with already verified feasibility of $\widehat{x}$, shows the complementary slackness condition. As discussed previously, the sequences of subgradients $\left\{t^{k+1} \right\}_{k \in \N_0}, \ \left\{s^k\right\}_{k \in \N_0}$ are bounded and hence, one can assume once more without loss of generality the convergence of $\left\{t^{k+1} \right\}_{k \in \Lambda}$ to some $\widehat{t} \in \R^n$ and $\left\{s^k \right\}_{k \in \Lambda}$ to $\widehat{s} \in \R^n$, respectively. In fact, by the closedness of the graph of the convex subdifferential we have $\widehat{t} \in \partial g \left( \widehat{x} \right)$ and $\widehat{s} \in \partial h \left( \widehat{x} \right)$ as both, $\left\{x^{k+1} \right\}_{k \in \Lambda}$ as well as $\left\{x^k \right\}_{k \in \Lambda}$ converge to $\widehat{x}$ by Step~1. Now, taking the limit $k \rightarrow_{\Lambda} \infty$ in~\eqref{eq: optcondSubproblemVec} and exploiting again $\sigma_k Q_k \left( x^{k+1}-x^k \right) \rightarrow_{\Lambda} 0$ as well as the robustness of the normal cone, we get
    \[ - \bigg( \widehat{t} - \widehat{s}+ A^T \widehat{\mu} + \sum_{i=1}^m \widehat{\lambda}_i \widehat{\omega}_i \bigg) \in N_C \left( \widehat{x} \right),  \]
    which shows
    \[ \partial h \left( \widehat{x} \right) \cap \bigg( \partial g \left( \widehat{x} \right) + \left\{ A^T \widehat{\mu} \right\} + \sum_{i=1}^m \widehat{\lambda}_i \partial c_i \left( \widehat{x} \right) + N_C \left( \widehat{x} \right) \bigg) \neq \emptyset. \]
    Altogether, this shows that $\left( \widehat{x},\widehat{\lambda},\widehat{\mu}\right)$ is in fact a generalized KKT~point of~\eqref{eq: problemFormulation}.
\end{proof}

Note that the postulation of $\left\{x^k\right\}_{k \in \N_0}$ remaining bounded is not too strong in our context and automatically satisfied if $C$ is compact since, by construction, all iterates $x^k$ belong to $C$. Alternatively, note that we may add suitable lower and upper bounds in case $C$ is not compact.

\section{Numerical experiments and applications} \label{sec: numerics}
This section presents some numerical results for our \cref{alg: sALMDC} together with a comparison with existing methods. Since our method is designed for the solution of constrained DC~programs with both DC~components being nonsmooth, we concentrate on this class of DC~programs.
More precisely, we consider problems from location planning~\cite{Eilon1971,Chen1998,Tuy2016} as well as for sparse signal recovery~\cite{Yin2015}. For the latter one, we extend the approach presented in~\cite{Yin2015} by a modification of the problem which is based on the idea of~\cite{Gotoh2018} for promoting sparsity.

Comparisons of our method are carried out with the classical DC~Algorithm (DCA)~\cite{LeThi1996} and the proximal bundle method for DC~programming (PBMDC) presented in~\cite{deOliveira2019}. To apply DCA to problem~\eqref{eq: problemFormulation}, we consider the (formally) unconstrained reformulation
\[ \min_{x \in \R^n} \widetilde{f}(x):= \widetilde{g}(x)-h(x)\]
with $\widetilde{f}, \widetilde{g}:\R^n \to \overline{\R}, \ \widetilde{g}:=g+\delta_X$ and $X$ denoting the feasible set of~\eqref{eq: problemFormulation} given by~\eqref{eq: feasibleSet}. At iteration $k$ one replaces the concave part $-h$ by an affine majorant. Minimizing the resulting model function in order to determine the next iterate $x^{k+1}$ leads to the constrained optimization problem
\begin{equation} \label{eq: DCA_subprob}
    \min_{x \in \R^n} \left\{ g(x) - h\left(x^k \right) - \left( s^k\right)^T\left(x-x^k\right)\right\} \quad \text{s.t.} \quad Ax=b, \ c(x)\leq 0, \ x \in C,
\end{equation}
where $s^k \in \partial h \left(x^k \right)$ with current iterate $x^k$. Hence, the performance of the algorithm highly depends on whether there is an efficient numerical method for solving these constrained subproblems at hand. Moreover, in order to guarantee existence and uniqueness of a solution of~\eqref{eq: DCA_subprob}, one adds the same strongly convex term to both DC~components $g$ and $h$ before substituting the concave part. To be more precise, we use $\frac{1}{2}\Vert \cdot \Vert^2$ in what follows.
Hence, instead of~\eqref{eq: DCA_subprob}, we obtain
\begin{equation} \label{eq: DCA_subprobProx}
    \min_{x \in \R^n} \left\{ g(x) - h\left(x^k \right) - \left( s^k\right)^T\left(x-x^k\right) + \frac{1}{2} \| x-x^k \|^2\right\} \quad \text{s.t.} \quad Ax=b, \ c(x)\leq 0, \ x \in C
\end{equation}
(after adding some constant terms). The resulting method is known in the literature as a variant of DCA, called proximal linearized algorithm for difference of convex functions~\cite{deOliveira2019, Souza2015}.

On the other hand, PBMDC \cite{deOliveira2019} can be applied directly to problem~\eqref{eq: problemFormulation}. The key idea is to approximate, at each iteration $k$, not only the concave part $-h$ by an affine majorant, but also the first DC~component~$g$ by a cutting plane model. Adding a proximal term to the resulting model function leads again to a convex, constrained subproblem. A trial point gets computed as solution of this optimization problem and becomes the new iterate whenever a sufficient decrease in the objective value can be ensured. Otherwise, the bundle to construct the cutting plane model of~$g$ gets enriched while no update of the iterate is carried out. As a consequence, also the performance of PBMDC depends, in principle, on an efficient numerical method being available to solve the convex, constrained subproblems. However, in what follows, the feasible set always constitutes a closed, convex polytope and hence, can be characterized by finitely many equality and inequality restrictions. Thus, the subproblems occurring in PBMDC can be equivalently transformed into quadratic programs \cite{deOliveira2019}. For this class of optimization problems several established numerical solution methods are known. In particular, \textsc{matlab} provides with \texttt{quadprog}\footnote{see \url{https://de.mathworks.com/help/optim/ug/quadprog.html} \label{fn: quadprog}} a selection of such algorithms.

We have used the \textsc{matlab} implementation of PBMDC which is freely available by its developer at \url{https://www.oliveira.mat.br/solvers}. In addition, DCA as well as psALMDC have been implemented in \textsc{matlab}, too. PBMDC has been run with its default parameters, except switching the value for \texttt{pars.QPdual} from \texttt{true} to \texttt{false}, as otherwise the constraints had been disregarded. 
In addition, we have fixed a small bug in the implementation in order to avoid errors when reaching the maximum size of a bundle\footnote{To be precise, when extracting the relevant part of the Lagrange multiplier after solving the quadratic program with \texttt{quadprog}, line~308 should be corrected to \texttt{mu= mu(nineq+1:end)}.}.

For DCA the only parameter to choose is the termination tolerance which we have set to $10^{-3}$. That means, DCA terminates as soon as $ \| x^k-x^{k+1} \| \leq 10^{-3} $ is satisfied. As already mentioned, choosing a suitable routine for solving the subproblems highly depends on the underlying structure. Hence, we comment on this for each application separately in the corresponding subsections.

Regarding our solver psALMDC from \cref{alg: sALMDC}, we have chosen the parameters as $M:=20, \ \linebreak[2] N:=\nolinebreak 5, \ \alpha:=0.5, \ \beta:=0.9, \ \gamma:=0.9, \ \theta:=0.8$ and $ \overline{\eta}:=10$. Moreover, the termination criterion in~\ref{alg: termination} gets numerically realized by checking whether the conditions
\begin{equation} \label{eq: psALMDC_numericTermination}
    \sigma_k \| Q_k \left(x^{k+1}-x^k\right) \| \leq \delta_1, \quad \| Ax^{k+1}-b \| \leq \delta_2, \quad \Big\| \min \left\{-c \left(x^{k+1}\right),\lambda^{k+1} \right\} \Big\| \leq \delta_2
\end{equation}
are satisfied for some given tolerances $\delta_1, \delta_2 >0$ depending on the application. 
In our applications, the abstract set $C$ is always equal to $\R^n$. As a consequence, the subproblems in~\ref{alg: subprob} become unconstrained ones. To this class of optimization problems one can, in principle, apply bundle methods like those proposed in~\cite{Kanzow2002,Lemarechal2014,Kiwiel1983}.
However, in order to improve the efficiency of the overall algorithm, it seems advantageous to exploit the structure of the respective subproblems for selecting a suitable solution method. For more details, see the corresponding subsection.

All codes were run with \textsc{matlab} version R2024a executed on a computer with \linebreak \mbox{8xIntel\textregistered \(\text{Core}^{\text{TM}}\)i7-7700 CPU \(@\) 3.60 GHz processor}, 31.1~GiB RAM and an openSUSE Leap 15.6 (64-bit) system.  

\subsection{Location planning} \label{sec: locationPlanning}
The consideration of how to place certain facilities to serve the demand of participants whose locations are known emerges at many problems. While developing a wider area settlement the question may arise of how to place supply facilities for daily needs, like supermarkets, health care offices, educational institutions or department stores, in such a way that the sum of the most favorable distances from the residents' homes to the respective facilities or the resulting transportation costs gets minimized. Thereby, optimization can be carried out for example with respect to travel distance, time or total costs. 
The mathematical formulation leads to the following generalized or multisource Weber's problem~\cite{Tuy2016,Chen1998}. Given $N$ points $a^j \in \R^2, \ j=1, \dots, N,$ (e.g., the coordinates of warehouses) the task is to find $p$ points $x^i \in \R^2, \ i=1,\dots, p,$ (e.g., the coordinates of storage depots) solving
\[ \min_{X \in \R^{2 \times p}} f(X):= \sum_{j=1}^N w_j \min_{i=1,\dots,p} \| x^i-a^j \| \quad \text{s.t} \quad x^i \in S, \ i=1,\dots, p, \]
where $w_j>0, j=1,\dots,N,$ are weights (e.g., representing the transportation costs per unit of distance for the demand of warehouse $j$), $S$ denotes a convex region, usually a rectangle, and $X:=\left(x^1,\dots,x^p\right)$.
The objective function $f:\R^{2 \times p} \to \R$ can be rewritten as a DC~function $f=g-h$ with convex component functions $g,h: \R^{2 \times p} \to \R$,
\begin{align*}
    g(X)&:= \sum_{j=1}^N \sum_{i=1}^p w_j \| x^i-a^j \|,\\
    h(X)&:= \sum_{j=1}^N w_j \max_{k=1,\dots,p} \sum_{\substack{i=1,\dots,p \\ i \neq k}} \| x^i -a^j \|
\end{align*}
(see e.g.~\cite{Chen1998,Tuy2016}).

For our numerical tests we consider the 50-Customer Problems which can be found in~\cite{Chen1998,Eilon1971}. Given 50~demand points\footnote{The corresponding data can be found in Table~I of~\cite{Chen1998} as well as in Appendix~4.1(a) of~\cite{Eilon1971}.} $a^j \in \R^2, \ j=1,\dots, 50,$ the task is to place $p\in \{1,2,3\}$ facilities $x^i \in \R^2, \ i=1,\dots, p,$ in the rectangle $[0,10] \times [0,10]$ (the region over which also the data points are spread). 
For each number of facilities to be placed, we carry out 100~test runs with different initial guesses $X^0 \in \R^{2 \times p}$ of uniformly distributed random numbers ranging in $(0,10)$. 

In terms of psALMDC, we describe the feasible set via inequality constraints and set $C:=\R^{2 \times p}$ in the general problem formulation~\eqref{eq: problemFormulation}.
Besides, psALMDC requires some further initial guesses (see~\cref{alg: sALMDC}). In what follows, we identify $ X=  \left(x^1,\dots,x^p\right) \in\R^{2 \times p}$ with $X=\left(\left(x^1\right)^T,\dots,\left(x^p\right)^T\right)^T  \in \R^{2p}$ without change in notation (and similarly we proceed with other quantities in various dimensions).
We take $u^0:= \nolinebreak 4 \cdot \nolinebreak \mathbf{1}_{4p}$ where $\mathbf{1}_{n}$ denotes a vector of length $n$ containing only ones as entries. (Note that, for each $x^i, \ i=1,\dots, p,$ one has four inequalities to be satisfied.) Moreover, we choose $\sigma_0:=\epsilon_0=0.1$ and $Q_k:=10^{-3} \cdot \mathbf{I}_{2p}$ for all $k \in \N_0$ with $\mathbf{I}_{2p}$ denoting the $2p \times 2p$ identity matrix as well as $\delta_1:=\delta_2:=10^{-3}$ (see~\eqref{eq: psALMDC_numericTermination}). Additionally, we switch $\alpha$ from the default value to $0.9$\footnote{In principle, also the default value works. But adapting $\alpha$ has led to a speed up in convergence.}. In order to solve the unconstrained, convex, but nonsmooth subproblems in~\ref{alg: subprob} of \cref{alg: sALMDC} in iteration $k$, namely (neglecting constant terms)
\[
    \begin{aligned}
        \min_{X \in \R^{2 \times p}} \Bigg\{ \sum_{j=1}^N \sum_{i=1}^p w_j  \| x^i-a^j \| - \sum_{i=1}^p  \Big( \left(s^{i,k}\right)^T\left(x^i-x^{i,k}\right)
                &+\frac{1}{2\rho_{k}}\Big\| \max \left\{ 0,u^{k,i}+\rho_k  c^i\left(x^i\right) \right\} \Big\|^2 \\
        &+ \frac{q\sigma_k}{2}\| x^i-x^{i,k} \|^2 \Big) \Bigg\}
    \end{aligned}
\]
with $q:=10^{-3}$,
we first take advantage of the separability with respect to $i$ and then exploit that there is only a finite number of points in which the resulting objective functions $ \phi^{i,k}: \R^2 \to \R $
\[
    \phi^{i,k}(x):= \sum_{j=1}^N w_j \| x-a^j \| - \left(s^{i,k}\right)^T\left(x-x^{i,k}\right) +\frac{1}{2\rho_{k}}\Big\| \max \left\{ 0,u^{k,i}+\rho_k  c^i\left(x\right) \right\} \Big\|^2  +\frac{q\sigma_k}{2}\| x^i-x^{i,k} \|^2 
\]
are not differentiable. (Note that with $\phi^{i,k}$ we suppress dependencies others than the variable with respect to which minimization is carried out.) These points of non-differentiability are precisely the known data points $a^j, \ j=1,\dots,50$. Adapting the idea presented in~\cite{Beck2015} of modifying Weiszfeld's method~\cite{Weiszfeld2009} in order to avoid getting stuck in such points, we first determine $j_{i,k}:= \argmin_{j=1,\dots,50} \phi^{i,k}\left(a^j\right)$ and then check whether $a^{j_{i,k}}$ is a minimizer of $\phi^{i,k}$. In this case, we set $x^{i,k+1}:=a^{j_{i,k}}$. Otherwise, we continue with computing a descent direction $d^{i,k}$ of $\phi^{i,k}$ at $a^{j_{i,k}}$ and a trial point $\widetilde{x}^{i,k+1}:=a^{j_{i,k}}+t_k d^{i,k}$ with a stepsize $t_k$ (by Vardi and Zhang~\cite{Vardi2001}) which ensures $\phi^{i,k} \left(\widetilde{x}^{i,k+1}\right)< \phi^{i,k} \left(a^{j_{i,k}}\right)$. Now, we are in position to apply a (gradient) descent method to solve the $i$th subproblem, starting at the trial point $\widetilde{x}^{i,k+1}$ and remaining in the region of differentiability of the objective function. The methods for determining whether $a^{j_{i,k}}$ is already a minimizer and for constructing a descent direction otherwise can be found in~\cite{Beck2015}. As descent method, we take the \textsc{matlab} solver \texttt{fminunc}\footnote{see \url{https://de.mathworks.com/help/optim/ug/fminunc.html}} with the algorithm \texttt{trust-region} and the option \texttt{SpecifyObjectiveGradient} set to \texttt{true}.

For solving the subproblems of DCA (see~\eqref{eq: DCA_subprobProx}) the safeguarded augmented Lagrangian method~\cite{Birgin2014} gets applied. This way, the optimization problems in the subroutine are pretty close to the ones occurring in psALMDC. Hence, we utilize the same solution strategy of executing a descent method starting at a point which ensures that we stay in the region of differentiability during the iterations\footnote{Actually, the said strategy could be directly applied to the constrained optimization problem~\eqref{eq: DCA_subprobProx}. However, testing this approach with \textsc{matlab} solver \texttt{fmincon} as descent method for constrained optimization problems, we ended up with (approximately) the same solutions obtained with the use of the safeguarded augmented Lagrangian method but with a tendency of longer runtimes.}.  

In \cref{tab: locationPlanning} we illustrate, for each number of facilities $p$ to be placed, how many of the 100~instances each algorithm terminates with the coordinates of the optimal depot locations, that means the corresponding final objective value coincides with the optimal one from Appendix~4.1(b) of~\cite{Eilon1971} after rounding to the third decimal place. In addition, we present for each method the runtime taking the mean over those instances in which the algorithm has succeeded to identify the optimal location. 

\begin{table}[b]
    \centering
    \caption{Results for the location planning problem according to the number $p$ of facilities to be placed} \label{tab: locationPlanning}
    \begin{subtable}[c]{0.4\textwidth}
        \subcaption{Number of successful instances of 100~test runs}
        \begin{tabular}{l >{$}r<{$} >{$}r<{$} >{$}r<{$}}
            \hline\noalign{\smallskip}
            & p=1 & p=2 & p=3 \\ 
            \noalign{\smallskip}\hline\noalign{\smallskip}
            psALMDC & \mathbf{100} & \mathbf{9} & \mathbf{36} \\ 
            PBMDC & 99 & 3 & 27 \\ 
            DCA & \mathbf{100} & 5 & 27 \\ 
            \noalign{\smallskip}\hline
        \end{tabular}
    \end{subtable} \hspace{.5cm}
    \begin{subtable}[c]{0.4\textwidth}
        \subcaption{CPU time in seconds, averaged over the successful instances} \label{tab: locationPlanning_Time}
        \begin{tabular}{l >{$}r<{$} >{$}r<{$} >{$}r<{$}}
            \hline\noalign{\smallskip}
            & p=1 & p=2 & p=3 \\ 
            \noalign{\smallskip}\hline\noalign{\smallskip}
            psALMDC & \mathbf{0.0054} & \mathbf{0.0258} & 0.4104 \\ 
            PBMDC & 0.0493 & 0.0804 & \mathbf{0.1534} \\ 
            DCA & 0.0703 & 0.2231 & 0.5357 \\ 
            \noalign{\smallskip}\hline
        \end{tabular}
    \end{subtable}	
\end{table}

Regardless of the number of facilities $p$ to be placed, psALMDC has the highest number of successful test runs, followed by DCA which is slightly ahead of PBMDC. It is not surprising that for $p=1$ (nearly) all methods are able to determine the optimal location in every single instance as there is only one local optimum of the underlying minimization problem, whereas with growing number of facilities, also the number of (known) local optima increases (see Table~4.1 of~\cite{Eilon1971}). 
The reason why PBMDC fails for $p=1$ in determining the optimal location for one instance is that it terminates beforehand with an error indicating that the subproblem could not be solved\footnote{The precise error reads \texttt{Numerical issues when solving the master program}.}. This precise error also occurs four times with $p=2$. 
In terms of the CPU~time (see \cref{tab: locationPlanning_Time}) for $p=1$ and $p=2$ psALMDC is also on average the fastest algorithm, followed by PBMDC and then DCA. However, for $p=3$ PBMDC takes the lead over psALMDC with DCA being still slightly behind. Altogether, the runtime averaged over the successful instances seems to be reasonably low for each of the three algorithms.  
finally, let us note that for $p=1$ psALMDC requires for all but one instance only two iterations to determine the optimal location which is less than one-tenth of the number of iterations required by the remaining two algorithms.

\subsection{Sparse signal recovery} \label{sec: sparseSignalRecovery}
A major task in compressed sensing consists in retrieving a signal $\overline{x} \in \R^n$ from a measurement $b:=A \overline{x} \in \R^m$ with $m \ll n$, where $ A \in \R^{m \times n}$ denotes the sensing (measurement) matrix. In order to save storage space and perhaps computation time for further processing, there is a particular interest in finding a sparse representation of the signal, that is a solution $ x \in \R^n$ of the linear equation $ Ax=b$ with $x$ being as sparse as possible.
This task results in finding a solution of the optimization problem
\begin{equation} \label{eq: sparseSignal_l0}
    \min_{x \in \R^n} \Vert x \Vert_0 \qquad \text{s.t.} \qquad Ax=b,
\end{equation}
where $\Vert \cdot \Vert_0$ denotes the $\ell_0$-norm that counts the number of nonzero entries of a vector. However, non-convexity together with discontinuity of the $\ell_0$-norm causes severe difficulties (see~\cite{Gotoh2018} for more details). Hence, one switches to alternative formulations which still allow for sparsity of solutions. Subsequently, we present two different approaches, one following~\cite{Yin2015} and involving constrained $\ell_{1-2}$ minimization, the other resulting from~\cite{Gotoh2018} taking into account the difference of the $\ell_1$- and the largest-$k$-norm.

Based on the observation that the difference $\Vert \cdot \Vert_1 - \Vert \cdot \Vert_2$ can serve as a sparsity metric, the authors in~\cite{Yin2015} replace the optimization problem~\eqref{eq: sparseSignal_l0} for retrieving a sparse signal from the measurement by
\begin{equation} \label{eq: sparseSignal_l12}
    \min_{x \in \R^n} \left\{\Vert x \Vert_1 - \Vert x \Vert_2 \right\} \qquad \text{s.t.} \qquad Ax=b.
\end{equation}

The second approach adopts an idea from~\cite{Gotoh2018}. There, within the context of optimization problems with cardinality constraints, the $\ell_0$ norm gets replaced by the difference $\Vert \cdot \Vert_1 - \Vert \cdot \Vert_{[k]}$, where $\Vert \cdot \Vert_{[k]}$ denotes the largest-$k$-norm, that is the sum over the $k$ largest (in absolute value) entries of a vector. This is motivated by the fact that for an $s$-sparse vector $x \in \R^n$, which means $\Vert x \Vert_0=s$, one obviously obtains $\Vert x \Vert_1 - \Vert x \Vert_{[k]} =0$ if and only if $k \geq s$. Therefore, the second optimization problem under consideration reads
\begin{equation} \label{eq: sparseSignal_kNorm}
    \min_{x \in \R^n} \left\{ \Vert x \Vert_1 - \Vert x \Vert_{[k]}\right\} \qquad \text{s.t.} \qquad Ax=b
\end{equation}
with a predetermined $k \in \{1,\dots,n\}$. Note that all algorithms taken into account during the numerical examples require the efficient computation of a subgradient of the second DC~component $\Vert \cdot \Vert_{[k]}$. However, \cite{Gotoh2018,Watson1992} provide exactly what is needed for that matter.

Our numerical test settings are geared to the ones in Section~5.2 of~\cite{Yin2015}. We first try to recover different test signals $\overline{x} \in \R^n$ with varying sparsity $s$ from the measurement $b:=A\overline{x}$ with several kinds of sensing matrices $A \in \R^{m \times n}$, namely random Gaussian and random partial discrete cosine transform~(DCT) as well as randomly oversampled partial DCT~matrices. 
Throughout our experiments, whenever addressing problem~\eqref{eq: sparseSignal_kNorm}, we choose $k=s$ as the parameter for the largest-$k$-norm. Of course, one could argue that in applications the sparsity $s$ of the signal to be determined from a measurement is in general not known a priori. Hence, choosing $k$ with respect to, for example, the rank of the matrix~$A$ might be more advisable. But, as the task here is to recover a known signal, we allow for $k=s$.

For the random Gaussian matrices, each column of $A$ contains quasi-randomly generated values drawn from a multivariate normal distribution with mean vector $\mathbf{0_m}$ and covariance matrix $\frac{1}{m}\mathbf{I}_m$. The columns $a_i, \ i=1,\dots,n,$ of the random partial DCT matrices $A$ are constructed via
\[ a_i = \frac{1}{\sqrt{m}} \cos \left(2\pi i \xi\right),\]
where $\xi \in \R^m$ contains entries quasi-randomly drawn from the uniform distribution in the interval~$(0,1)$. Similarly, the columns $a_i, \ i =1\dots,n,$ of the randomly oversampled partial DCT matrices $A$ are chosen as
\begin{equation} \label{eq: sensingMatrix_coherent}
    a_i = \frac{1}{\sqrt{m}}\cos \left(\frac{2 \pi i}{F} \xi \right),
\end{equation}
where $\xi \in \R^m$ gets selected like above and $F \in \N$ defines the refinement factor which is closely related to the conditioning of the matrix $A$ (see~\cite{Yin2015} for details). While random Gaussian and random partial DCT matrices are incoherent with high probability and hence, are suitable for compressed sensing, randomly oversampled partial DCT matrices feature significantly higher coherence, leading to ill-conditioning of the sensing matrices~\cite{Yin2015}. 

The nonzero entries of each $s$-sparse test signal $\overline{x}\in \R^n$ are drawn from the standard normal distribution. Thereby, the support of $\overline{x}$, $\operatorname{supp}\left(\overline{x}\right):= \left\{ 1 \leq i \leq n \ \middle| \ \overline{x}\neq 0 \right\}$, is taken as a random index set with elements derived from the discrete uniform distribution on the interval $[1,n]$.
While considering randomly oversampled partial DCT matrices, one has to ensure that the entries of $\operatorname{supp} \left( \overline{x} \right)$ are sufficiently separated in order to enable recovery of the test signal~\cite{Yin2015}. Hence, we impose the additional condition
\begin{equation} \label{eq: signal_minimumSeparation}
    \min_{j,k\in \operatorname{supp} \left( \overline{x} \right)} \vert j-k \vert \geq L
\end{equation}
on the randomly generated support of $\overline{x}$, where $L:=2F$ determines the minimum separation of two consecutive nonzero elements of the vector (for more details on how to choose $L$, see~\cite{Yin2015}).
As initial guess $x^0 \in \R^n$ for our test runs, we take a perturbation $ x^0:= \overline{x}+\zeta$ of the signal to be reconstructed with $\zeta \in \R^n$ drawn from a multivariate normal distribution with mean vector $\mathbf{0}_n$ and covariance matrix~$\frac{1}{2}\mathbf{I}_n$.

Additionally, psALMDC requires some further initialization (see~\cref{alg: sALMDC}). To this end, we take $v^0:=m \mathbf{1}_m, \sigma_0:=100$ and $\epsilon_0:=0.1$. Furthermore, we set $Q_k:=10^{-4} \mathbf{I}_n$ for all $k \in \N_0$, $\delta_1:=1$ as well as $\delta_2:=10^{-4}$ in case of problem~\eqref{eq: sparseSignal_l12} and $\delta_2:=10^{-5}$ for solving~\eqref{eq: sparseSignal_kNorm}. 
For solving the unconstrained, convex, but nonsmooth subproblems of psALMDC in~\ref{alg: subprob} of \cref{alg: sALMDC}, which read in iteration $k$ (omitting constant terms)
\[ \min_{x \in \R^n} \phi^k(x):= \Vert x \Vert_1 - \left( s^k \right)^T x + \left(v^k \right)^T (Ax-b) +\frac{\rho_k}{2} \Vert Ax-b \Vert^2 + \frac{q\sigma_k}{2} \| x-x^k \|^2  \]  
with $q:=10^{-4}$, we exploit that $\phi^k- \Vert \cdot \Vert_1$ is differentiable with Lipschitz-continuous gradient, and a proximal point of the $\ell_1$-norm at any point $x \in \R^n$,
\[ \operatorname{prox}_{\Vert \cdot \Vert_1}(x) := \argmin_{y \in \R^n} \left\{ \Vert y \Vert_1 + \frac{1}{2} \Vert y-x \Vert^2\right\},\]
is easy to compute (see e.g. Example~6.8 in~\cite{Beck2017}). Hence, the Fast Iterative Shrinkage-Thresholding Algorithm (FISTA, see Chapter~10.7.1 in~\cite{Beck2017}) with backtracking is the solution method of our choice.

Like in the previous section we apply the safeguarded augmented Lagrangian method~\cite{Birgin2014} for solving the subproblems~\eqref{eq: DCA_subprobProx} occurring within DCA, in order to yield optimization problems in the subroutine which are pretty similar to the ones arising within psALMDC. Hence, it suggests itself to once again deploy FISTA with backtracking for tackling the said composite problems of the subroutine.

Preliminary experiments revealed that in case psALMDC terminates with a non-optimal solution uncommon high runtime may occur on-again-off-again, especially while considering coherent sensing matrices.
Also with DCA very few instances occurred for which the method has not converged within a reasonable time, however, not in context with coherent but incoherent measurement matrices.
Hence, we impose a general (and quite generous) time limit of one hour for these two methods which we extend to two hours when taking into account randomly oversampled DCT~matrices for problem~\eqref{eq: sparseSignal_l12}.
Let us note that we do not restrict the runtime of PBMDC as on the one hand, there are few instances for which the algorithm succeeds in recovering the test signal but exceeds the mentioned limit and on the other hand, in none of the preliminary test runs PBMDC surpasses this limit excessively. \\

Executing the numerical experiments, we first analyze the capacity of all three solvers psALMDC, DCA and PBMDC to reconstruct a sparse signal from a measurement involving (with high probability) incoherent sensing matrices, these are random Gaussian as well as random partial DCT~matrices. To this end, we fix $n:=256, \ m:=64$ and consider for every sparsity level $s \in \{10,12,14,\dots,28\}$ 100~test runs for each of our two approaches~\eqref{eq: sparseSignal_l12} and~\eqref{eq: sparseSignal_kNorm}. Thereby, all three algorithms are rendered the same data which varies with each instance in the test signal as well as in the measurement matrix.

In \cref{im: signalRecovery_incoherentSucc} we depict for both problem formulations and each sparsity level in how many of the 100~instances the respective algorithm has been able to recover a reconstruction~$x^{\ast}$ of the test signal $\overline{x}$ with a relative error of at most $10^{-3}$, that is $\frac{\| x^{\ast}-\overline{x}\|}{\| \overline{x}\|} \leq 10^{-3}$. 
Besides, in \cref{im: signalRecovery_incoherentTime} we report the CPU~time averaged for each algorithm over the instances rated as a success. Additionally, we provide for every method the median of the CPU~time taking into account once more only the test runs in which the corresponding test signal could be reconstructed appropriately. For the sake of completeness, let us note that in \cref{im: signalRecovery_DCTtime} for psALMDC and sparsity level $s=28$ the mean and the median of the CPU~time constitutes 554~seconds, having two successful instances with highly varying runtime.

\begin{figure}[t]
    \centering
    \begin{subfigure}{.45\textwidth}
        \centering
        \includegraphics[keepaspectratio=true, width=\textwidth]{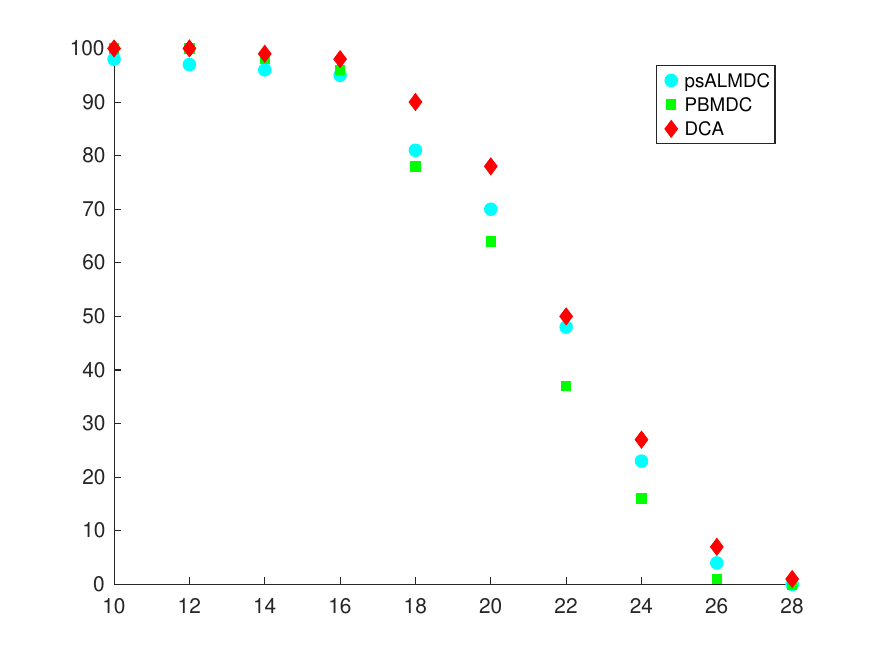}
        \caption{Problem~\eqref{eq: sparseSignal_l12} with Gaussian sensing matrix}
    \end{subfigure}
    \begin{subfigure}{.45\textwidth}
        \centering
        \includegraphics[keepaspectratio=true, width=\textwidth]{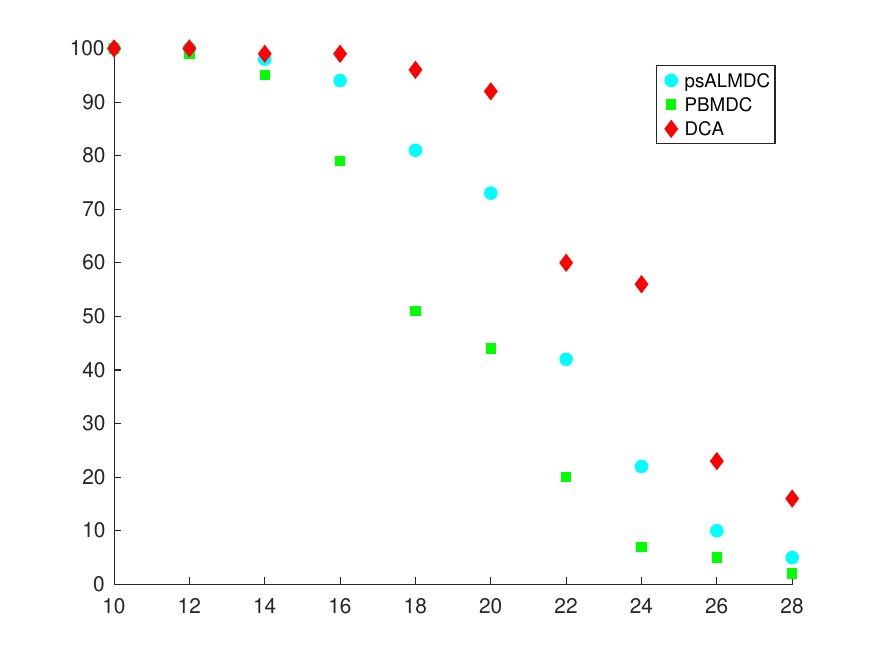}
        \caption{Problem~\eqref{eq: sparseSignal_kNorm} with Gaussian sensing matrix}
    \end{subfigure}
    \par
    \begin{subfigure}{.45\textwidth}
        \centering
        \includegraphics[keepaspectratio=true, width=\textwidth]{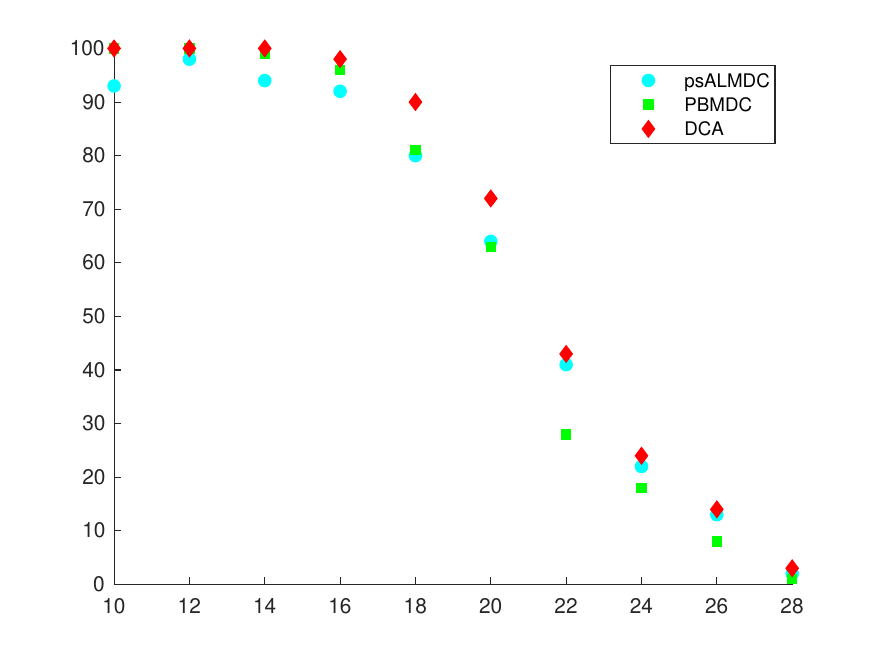}
        \caption{Problem~\eqref{eq: sparseSignal_l12} with DCT~sensing matrix}
    \end{subfigure}
    \begin{subfigure}{.46\textwidth}
        \centering
        \includegraphics[keepaspectratio=true, width=\textwidth]{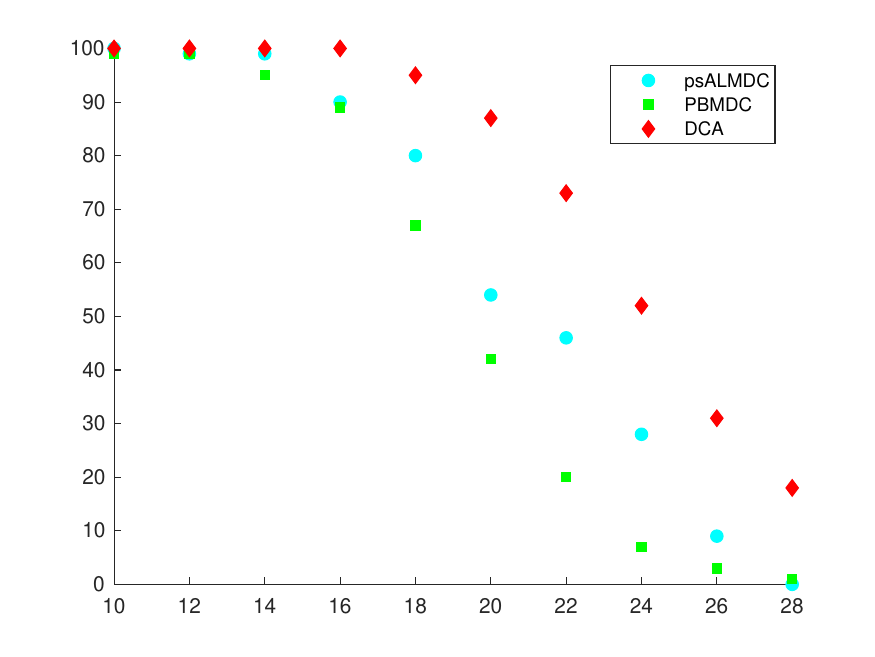}
        \caption{Problem~\eqref{eq: sparseSignal_kNorm} with DCT~sensing matrix}
    \end{subfigure}
    \caption{Number of successful instances of 100 test runs with respect to different sparsity levels, arranged according to various test settings} \label{im: signalRecovery_incoherentSucc}
\end{figure}

\begin{figure}[t]
    \centering
    \begin{subfigure}{.45\textwidth}
        \centering
        \includegraphics[keepaspectratio=true, width=\textwidth]{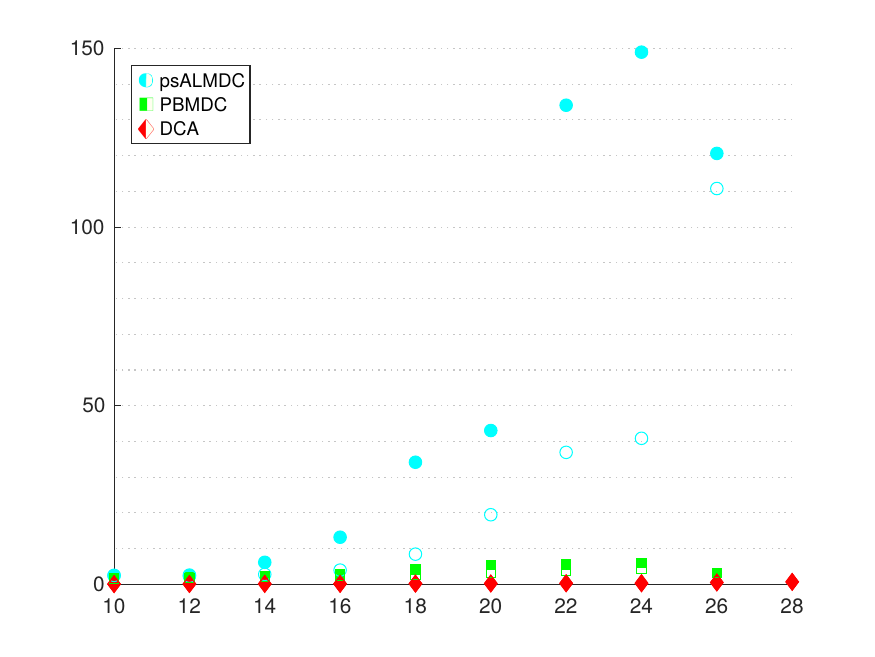}
        \caption{Problem~\eqref{eq: sparseSignal_l12} with Gaussian sensing matrix}
    \end{subfigure}
    \begin{subfigure}{.45\textwidth}
        \centering
        \includegraphics[keepaspectratio=true, width=\textwidth]{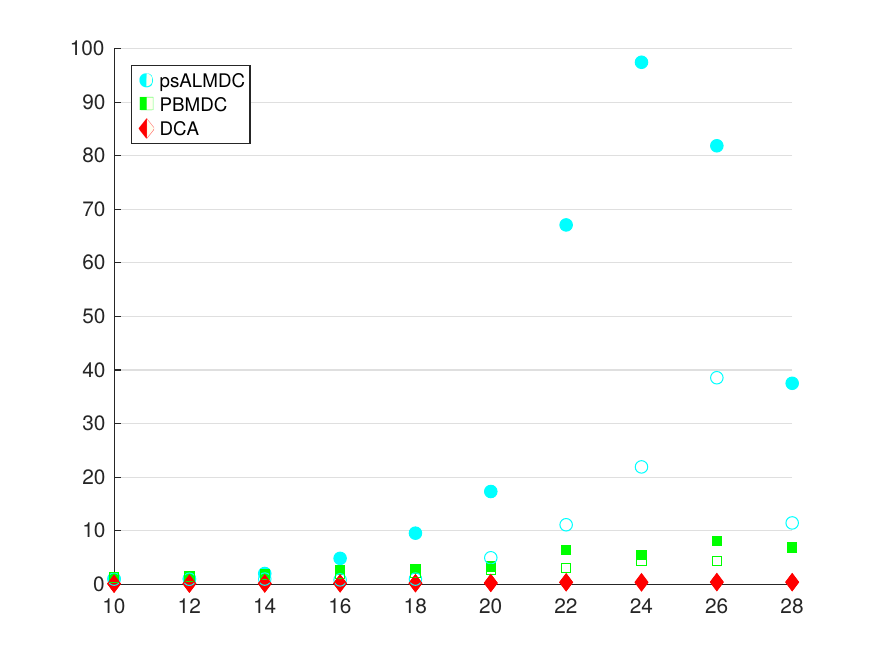}
        \caption{Problem~\eqref{eq: sparseSignal_kNorm} with Gaussian sensing matrix}
    \end{subfigure}
    \par
    \begin{subfigure}{.45\textwidth}
        \centering
        \includegraphics[keepaspectratio=true, width=\textwidth]{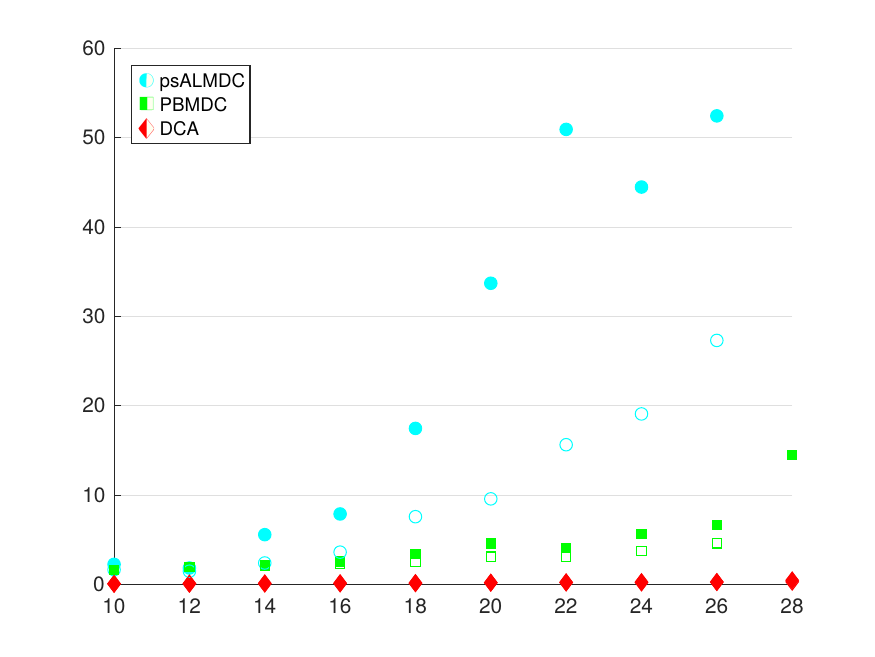}
        \caption{Problem~\eqref{eq: sparseSignal_l12} with DCT~sensing matrix (extract) \label{im: signalRecovery_DCTtime}}
    \end{subfigure}
    \begin{subfigure}{.47\textwidth}
        \centering
        \includegraphics[keepaspectratio=true, width=\textwidth]{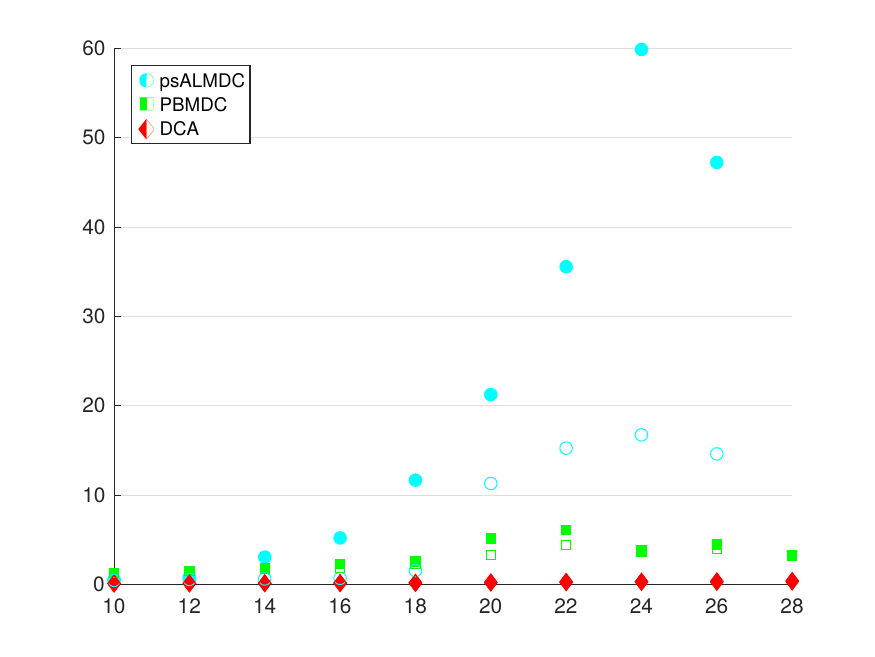}
        \caption{Problem~\eqref{eq: sparseSignal_kNorm} with DCT~sensing matrix}
    \end{subfigure}
    \caption{CPU time in seconds with respect to different sparsity levels, averaged (filled markers) as well as taken the median (marker silhouettes) over the successful instances only, arranged according to various test settings} \label{im: signalRecovery_incoherentTime}
\end{figure}

Clearly, DCA outperforms psALMDC as well as PBMDC in terms of both, success rates and CPU~time. In particular, for each test setting and each sparsity level the CPU~time with respect to the successful instances undershoots one second on average. However, while taking Gaussian sensing matrices and the sparsity level $s=26$ of the test signal into account, there is one single instance for which DCA terminates early because the subroutine for solving the subproblem does not converge.

The differences in the performance of psALMDC and PBMDC are less strict.
While considering the success rates in the context of $l_{1-2}$ minimization, that is problem~\eqref{eq: sparseSignal_l12}, PBMDC surpasses psALMDC initially for sparsity levels up to 16 in terms of Gaussian and 18 in terms of DCT sensing matrices but then gets outperformed by psALMDC. With regard to problem~\eqref{eq: sparseSignal_kNorm} involving the largest-$k$-norm psALMDC (nearly) throughout outrages PBMDC. 
However, taking into account the development of the CPU~time with increasing sparsity level PBMDC is the algorithm to show the overall better performance in comparison with psALMDC. Although, psALMDC is able to keep up with PBMDC and in terms of problem~\eqref{eq: sparseSignal_kNorm} to even undercut PBMDC for low sparsity levels~$s$, the gap between the averaged runtime of psALMDC and PBMDC increases visibly when exceeding $s=14$. 
At the beginning of the distinct ascent of the mean runtime of psALMDC the runtime of psALMDC stays for the majority of the successful instances around the corresponding averaged runtime of PBMDC. But at the same time there are single outliers of successful instances with comparably long runtime. Hence, the median of psALMDCs runtime increases slightly more slowly than the respective mean, sticking around the mean runtime of PBMDC for a few further sparsity levels. However, at the latest of sparsity level $s=20$, the runtime of psALMDC climbs in general for successful instances above those of PBMDC.
For PBMDC (and also DCA) the spread of runtime within a single sparsity level appears, if at all, to insignificant extent.\\

Execution and evaluation of the numerical experiments regarding the ill-conditioned problem, that is considering randomly oversampled partial DCT~matrices for sensing in both, \eqref{eq: sparseSignal_l12} as well as~\eqref{eq: sparseSignal_kNorm}, equals the previous tests involving (with high probability) incoherent measurement matrices. But this time, we take $n:=2000, \ m:=100$ and consider refinement factors $F \in \{10,20\}$ in~\eqref{eq: sensingMatrix_coherent} and~\eqref{eq: signal_minimumSeparation}, respectively, as well as sparsity levels $s \in \{5,7,9,...,25\}$ for~\eqref{eq: sparseSignal_l12} and $ s \in \{5,7,9,\dots, 29\} $ for~\eqref{eq: sparseSignal_kNorm}. Once again, we illustrate for each problem~\eqref{eq: sparseSignal_l12} and~\eqref{eq: sparseSignal_kNorm} and each algorithm in \cref{im: signalRecovery_coherentSucc} the success rates according to the sparsity levels and in \cref{im: signalRecovery_coherentTime} the average as well as the median of the CPU~time restricted to the successful instances. 

\begin{figure}[t]
    \centering
    \begin{subfigure}{.45\textwidth}
        \centering
        \includegraphics[keepaspectratio=true, width=\textwidth]{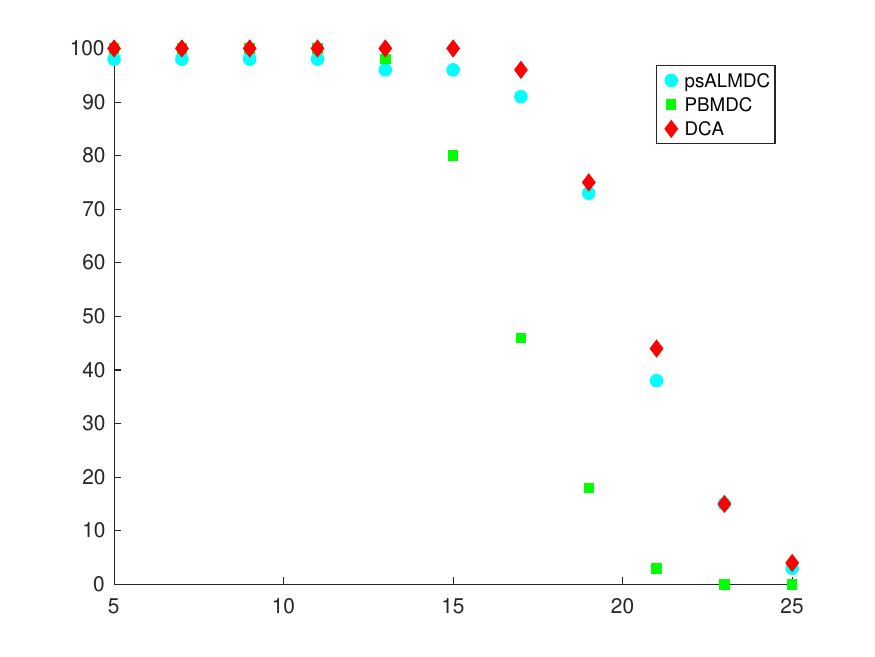}
        \caption{Problem~\eqref{eq: sparseSignal_l12}, $F=10$}
    \end{subfigure}
    \begin{subfigure}{.45\textwidth}
        \centering
        \includegraphics[keepaspectratio=true, width=\textwidth]{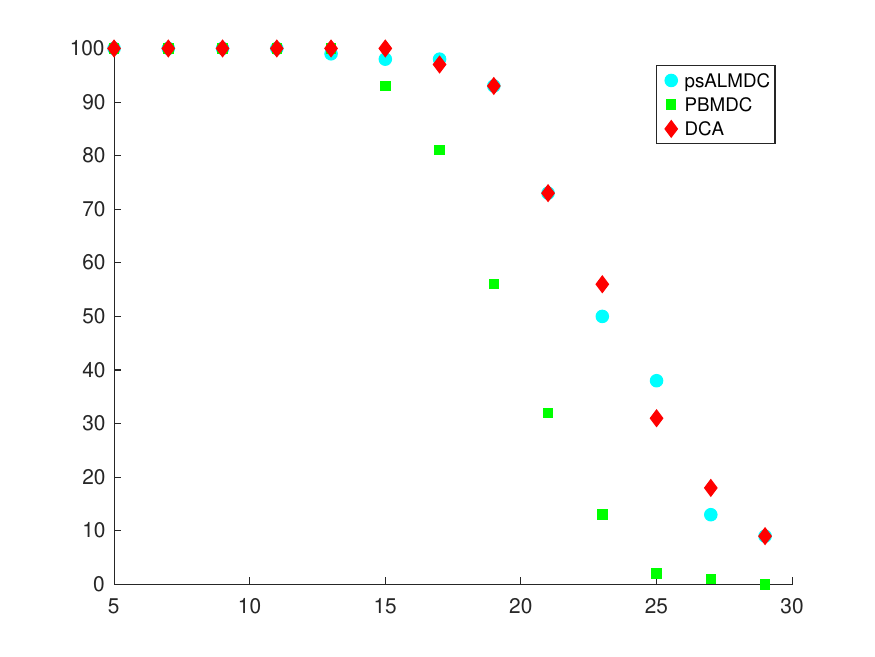}
        \caption{Problem~\eqref{eq: sparseSignal_kNorm}, $F=10$}
    \end{subfigure}
    \par
    \begin{subfigure}{.45\textwidth}
        \centering
        \includegraphics[keepaspectratio=true, width=\textwidth]{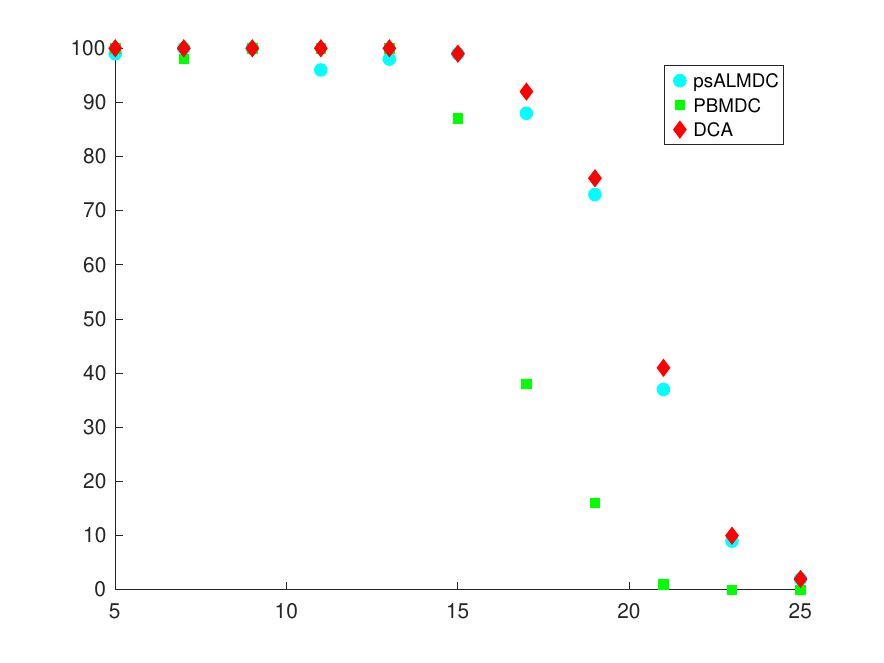}
        \caption{Problem~\eqref{eq: sparseSignal_l12}, $F=20$}
    \end{subfigure}
    \begin{subfigure}{.46\textwidth}
        \centering
        \includegraphics[keepaspectratio=true, width=\textwidth]{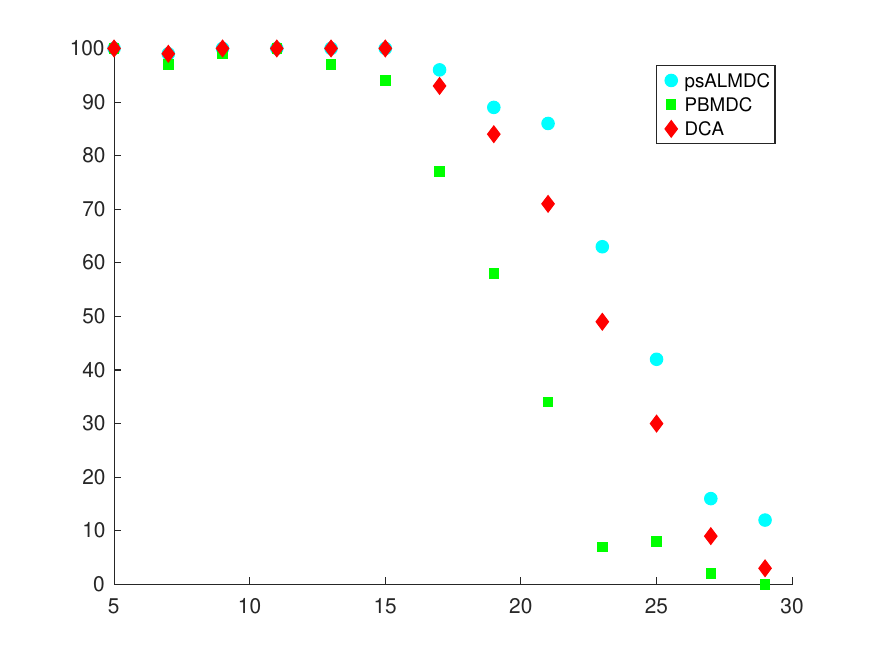}
        \caption{Problem~\eqref{eq: sparseSignal_kNorm}, $F=20$}
    \end{subfigure}
    \caption{Number of successful instances of 100 test runs with respect to different sparsity levels while considering randomly oversampled DCT sensing matrices with different refinement factors $F$} \label{im: signalRecovery_coherentSucc}
\end{figure}

\begin{figure}[t]
    \centering
    \begin{subfigure}{.45\textwidth}
        \centering
        \includegraphics[keepaspectratio=true, width=\textwidth]{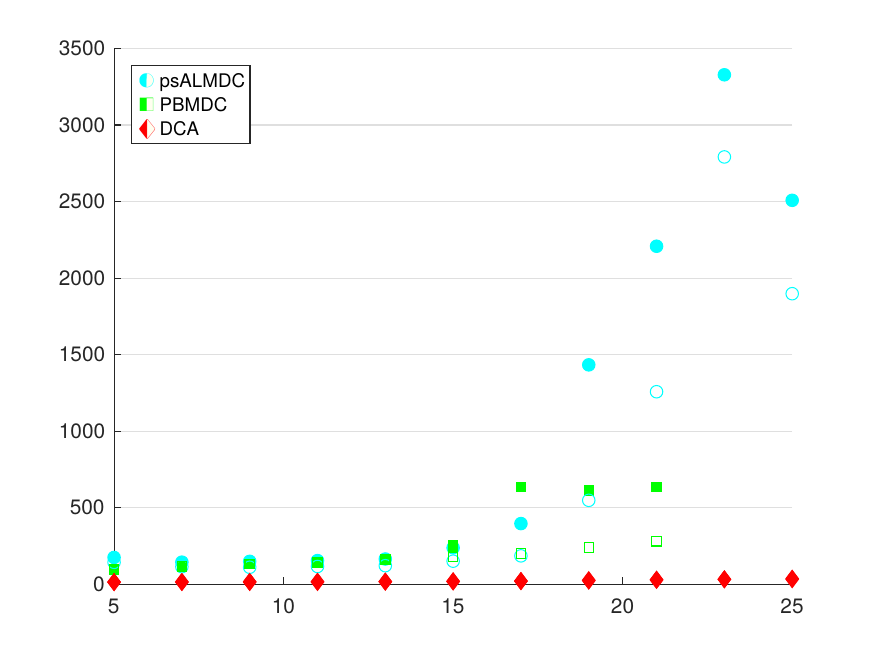}
        \caption{Problem~\eqref{eq: sparseSignal_l12}, $F=10$}
    \end{subfigure}
    \begin{subfigure}{.45\textwidth}
        \centering
        \includegraphics[keepaspectratio=true, width=\textwidth]{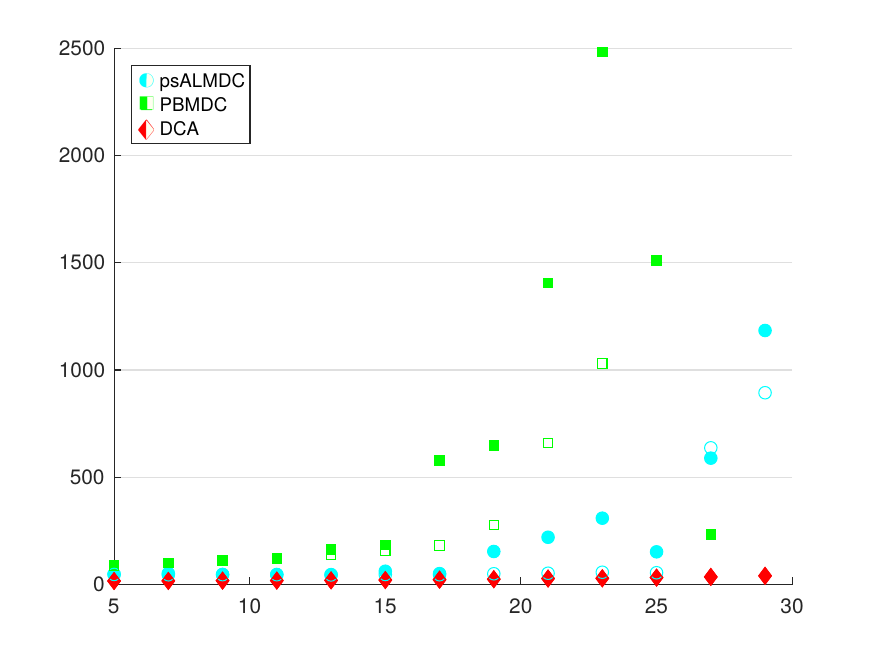}
        \caption{Problem~\eqref{eq: sparseSignal_kNorm}, $F=10$}
    \end{subfigure}
    \par
    \begin{subfigure}{.45\textwidth}
        \centering
        \includegraphics[keepaspectratio=true, width=\textwidth]{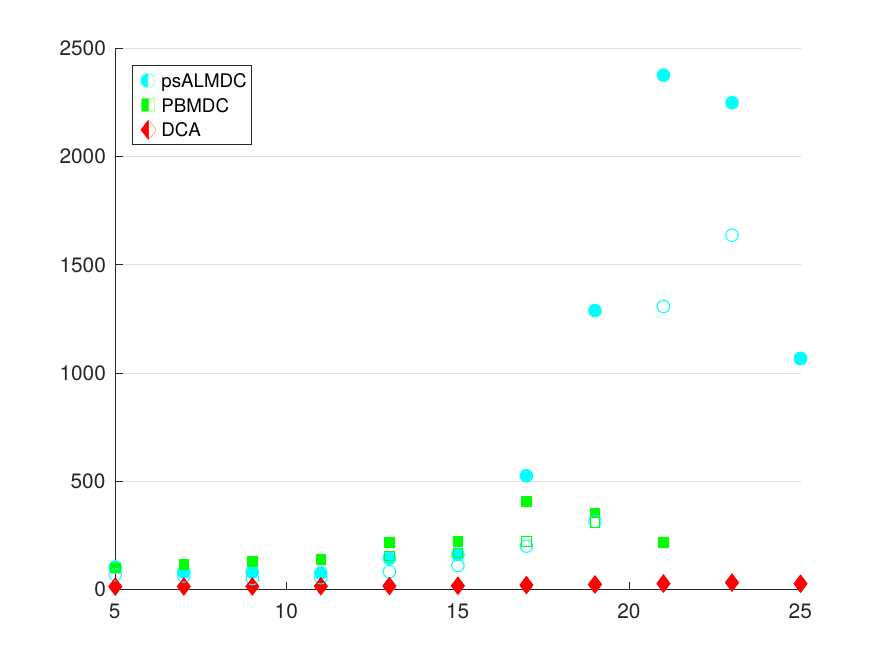}
        \caption{Problem~\eqref{eq: sparseSignal_l12}, $F=20$}
    \end{subfigure}
    \begin{subfigure}{.46\textwidth}
        \centering
        \includegraphics[keepaspectratio=true, width=\textwidth]{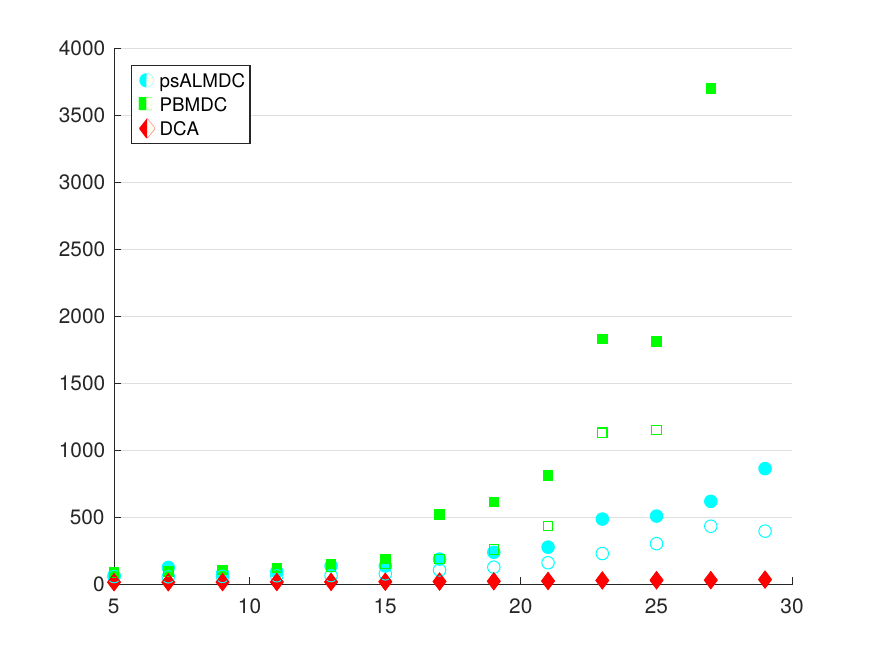}
        \caption{Problem~\eqref{eq: sparseSignal_kNorm}, $F=20$}
    \end{subfigure}
    \caption{CPU time in seconds with respect to different sparsity levels, averaged (filled markers) as well as taken the median (marker silhouettes) over the successful instances only, while considering randomly oversampled DCT sensing matrices with different refinement factors $F$} \label{im: signalRecovery_coherentTime}
\end{figure}

Once again, DCA is the overall outstanding algorithm. However, psALMDC is able to outperform DCA in view of the success rates while considering problem~\eqref{eq: sparseSignal_kNorm}, sparsity levels greater or equal than~17 and a refinement factor of $F=20$ for the randomly oversampled partial DCT matrices. Besides, while further comparing the success rates PBMDC falls behind the two remaining algorithms for most of the scenarios. Only taking into account low sparsity levels of at most 13 in context of problem~\eqref{eq: sparseSignal_l12} PBMDC gets slightly ahead of psALMDC for quite some instances.

Comparing the CPU~time by means of \cref{im: signalRecovery_coherentTime} there is no doubt that DCA by far outnumbers psALMDC and PBMDC. It takes DCA on average less than 45~seconds to reconstruct the test signal from a measurement involving a randomly oversampled DCT sensing matrix.

Collating the CPU~time of psALMDC and PBMDC there is a clear difference between problem~\eqref{eq: sparseSignal_l12} and~\eqref{eq: sparseSignal_kNorm}. In case of problem~\eqref{eq: sparseSignal_kNorm} (nearly) throughout all sparsity levels and both refinement factors it takes PBMDC on average more time to reconstruct the test signal than psALMDC (considering the successful instances). Devoting ourselves to problem~\eqref{eq: sparseSignal_l12} it strikes that psALMDC gets visibly outperformed by PBMDC for high sparsity levels~$s$, at the latest from $s=19$. However, having a closer look at the results with refinement factor $F=20$ one realizes that psALMDC is ahead of PBMDC for (nearly all) sparsity levels up to 15.

In the end, let us note that a visible gap between mean and median of the CPU~time can be observed in \cref{im: signalRecovery_coherentTime} for both psALMDC and PBMDC especially with increasing sparsity level. However, the gaps are in general less distinct than the ones occurring with psALMDC for high sparsity levels in \cref{im: signalRecovery_incoherentTime}. The reasons are similar though.\\

Finally, we want to give a quick contrasting juxtaposition between the two considered methods for reconstruction a sparse test signal, constrained $\ell_{1-2}$ minimization and the one involving the largest-$k$-norm (see problems~\eqref{eq: sparseSignal_l12} and~\eqref{eq: sparseSignal_kNorm}, respectively). 
There is a clear tendency that with the usage of the largest-$k$-norm in the objective function one can obtain higher success rates compared to the problem formulation involving the $\ell_2$-norm instead (see \cref{im: signalRecovery_incoherentSucc} and \ref{im: signalRecovery_coherentSucc}). PBMDC is the only method for which employing problem~\eqref{eq: sparseSignal_l12} leads to better results (nearly) independent of the chosen sparsity level while taking into account random Gaussian matrices. Also while looking at the second class of (with high probability) incoherent sensing matrices, those are random partial DCT~matrices, with PBMDC there occur increasingly instances for which problem~\eqref{eq: sparseSignal_l12} seems to be more favorable in terms of success rates.

Comparing the runtime of each of our two approaches for retrieving a sparse signal differences in between the considered algorithms get visible (see \cref{im: signalRecovery_incoherentTime} and~\ref{im: signalRecovery_coherentTime}). On trend, psALMDC and PBMDC have been able to reconstruct the test signal faster (for the successful instances) while using the largest-$k$-norm (that is problem~\eqref{eq: sparseSignal_kNorm}) compared to taking into account the $\ell_2$-norm (see problem~\eqref{eq: sparseSignal_l12}). Merely, in the context of coherent sensing and higher sparsity levels of at least 17 the average runtime of PBMDC lies for problem~\eqref{eq: sparseSignal_l12} below the one for problem~\eqref{eq: sparseSignal_kNorm}. Surprisingly, DCA shows the reverse behavior. In general, this method requires less runtime for the reconstruction of the test signal while solving problem~\eqref{eq: sparseSignal_l12} rather than problem~\eqref{eq: sparseSignal_kNorm}, at least for sparsity levels within low and mid range. For (very) high sparsity levels, however, the usage of the largest-$k$-norm seems to be preferable against the $\ell_2$-norm also for DCA, no matter which sensing matrix is taken into account.

\section{Concluding remarks} \label{sec: conclusion}
In this work a proximal safeguarded augmented Lagrangian method for minimizing nonsmooth DC~functions over a convex feasible region defined by a combination of linear equality as well as convex inequality constraints and an abstract constraint, called psALMDC, was introduced. Contrary to many existing algorithms, the new method provides both, handling of a general nonsmooth DC~objective function and augmentation of the (in)equality constraints, which are probably the most common restrictions in applications. Combining the basic concepts of the classical DC~Algorithm (DCA,~\cite{LeThi1996}), that is substituting the concave part of the objective function by an affine majorization, and the safeguarded augmented Lagrangian method~\cite{Birgin2014}, the subproblems to be solved in each iteration are convex ones. Thereby, constraints occur within those only if there is indeed an abstract one in the overall problem, which is rare in applications, or if someone intends to keep simple (in)equality constraints, such as box constraints, explicitly. Besides, the overall method includes an implementable termination criterion.

Assuming boundedness of the (primal) iterates, which is argued not to be too restrictive, and supposing a modified Slater constraint qualification to be satisfied for the problem under consideration, the existence of a convergent subsequence of iterates, both in terms of primal as well as Lagrangian dual variables, is proven. The corresponding limit point is shown to be a generalized KKT~point of the regarded problem.

Numerical experiments were conducted in order to judge the practical performance of the new algorithm against existing solution methods. For the purpose of comparison, the classical DCA~\cite{LeThi1996} next to the proximal bundle method for nonsmooth DC~programming (PBMDC,~\cite{deOliveira2019}) were taken into account. As applications served the 50-Customer Problems from~\cite{Chen1998,Eilon1971} and the recovery of sparse signals~\cite{Yin2015,Gotoh2018}.

For the 50-Customer Problems our new method psALMDC shows the overall best performance in terms of success in determining the optimum and in the majority of cases also with regard to CPU~time. 
However, while considering reconstruction of sparse signals the classical DCA is the method to clearly outperform both psALMDC and PBMDC. Nevertheless, the latter two algorithms yield satisfactory success rates, too, which in terms of psALMDC approach the ones of DCA for most of the instances.

So far, hardly any algorithms for tackling constrained DC~optimization problems are known which provide, in the limit, points meeting requirements beyond criticality, like for example d-stationarity. The few that exist require additional assumptions on the structure of the DC~components. Hence, our future research is dedicated to the development of an algorithm for constrained, nonsmooth DC~optimization which gets along without such additional postulations but despite provides stronger convergence results.

\appendix

\section{Proof of Lemmas \NoCaseChange{\ref{lem: CQs_stepOne}} and~\ref{lem: CQs}} \label{append: proof_CQs}
\begin{proof}[Proof of \cref{lem: CQs_stepOne}]
    First, we prove via contradiction that whenever mEMFCQ holds at $\widehat{x}$ also NNAMCQ is satisfied at $\widehat{x}$. To this end, assume that $ 0 \in \nolinebreak[3] \interior \left( \ell (C) \right)$ and there exists $d \in T_C \left( \widehat{x} \right)$ with $Ad=0$ and $c_i' \left( \widehat{x}; d \right)<0$ for all $i \in I \left( \widehat{x} \right)$, as well as a nonzero vector $(\lambda, \mu) \in \R^m_+ \times \R^p$ satisfying
    \[ 0 \in \left\{ A^T\mu \right\} + \sum_{i=1}^m \lambda_i \partial c_i\left( \widehat{x} \right) + N_C \left( \widehat{x} \right) \quad \text{and} \quad \lambda^Tc \left(\widehat{x}\right) = 0. \]
    This implies the existence of some $\omega_i \in \partial c_i \left( \widehat{x} \right) $ for all $i \in I \left( \widehat{x} \right) $ and some $\nu \in N_C \left( \widehat{x} \right)$ with
    \begin{equation} \label{eq: proofCQNAAMCQvec}
        0=  A^T \mu + \sum_{i \in I \left( \widehat{x} \right)} \lambda_i \omega_i + \nu,
    \end{equation}
    where we have already taken into account that $\lambda_i =0$ for all $i \notin I \left(\widehat{x}\right)$ by feasibility of~$\widehat{x}$, non-negativity of $\lambda$ and the complementary slackness condition.
    Taking the inner product with $d$, yields
    \begin{equation} \label{eq: proofCQscalarproduct}
        0=\mu^T Ad + \sum_{i \in I \left( \widehat{x} \right)} \lambda_i \omega_i^T d + \nu^T d.
    \end{equation}
    By assumption, for each $i \in I \left( \widehat{x} \right)$, we have
    \[ 0 > c_i ' \left( \widehat{x}; d \right) = \max_{\widetilde{\omega} \in \partial c_i \left( \widehat{x}\right)} \widetilde{\omega}^Td \geq \omega_i^Td. \]
    In addition, since $N_C \left(\widehat{x}\right) = \left( T_C \left( \widehat{x} \right) \right)^\circ$, it holds $\nu^Td \leq 0$ as well. Hence, combining the last two conclusions with~\eqref{eq: proofCQscalarproduct}, $\lambda$ being non-negative and $Ad=0$ it follows $\lambda_i=0$ for all $i \in I \left( \widehat{x}\right)$ (hence, $\lambda=0$) as well as $\nu^Td=0$. Therefore, \eqref{eq: proofCQNAAMCQvec} reduces to $0 = A^T\mu + \nu$ showing $-A^T \mu \in N_C \left(\widehat{x} \right)$. In particular, $\mu \neq 0$ as we assumed $(\lambda, \mu)$ to be nonzero. By definition of the normal cone from convex analysis, we have
    \[ \left(-A^T \mu\right)^T \left( x - \widehat{x} \right) \leq 0 \quad \forall x \in C \]
    or, equivalently, taking into account the feasibility of $\widehat{x}$ and the definition of $\ell$,
    \[ - \mu^T \ell(x) \leq 0 \quad \forall x \in C. \]
    But since $0 \in \interior \left( \ell(C)\right)$, for sufficiently small $\epsilon >0$ one has $-\epsilon \mu \in \ell(C)$ and thus
    \[ \epsilon \Vert \mu \Vert \leq 0\]
    leading to $\mu =0$, a contradiction.

    The reverse implication is shown via contraposition. To this end, we assume that mEMFCQ is violated at $\widehat{x}$. This means either $0 \notin \nolinebreak \interior \left( \ell (C) \right)$ or  $ 0 \notin S$ with $S \subseteq  \R^p \times \R^{\left\vert I \left(\widehat{x} \right) \right\vert}$ being defined as
    \begin{equation} \label{eq: proofCQs_setS}
        S:= \left\{ (z,\alpha) \ \middle| \ \exists d \in T_C \left(\widehat{x}\right), s\in \R^{\left\vert I \left(\widehat{x} \right) \right\vert},s>0: \ Ad=z, \ c_i'\left(\widehat{x};d \right) +s_i = \alpha_i \ \forall i \in I \left(\widehat{x}\right) \right\}.
    \end{equation}
    Let us consider the case $0 \notin \interior \left( \ell (C) \right)$ first. 
    Under this circumstance, there exists some (nontrivial) vector $v \in \R^p$ such that $\beta v \notin \ell (C)$ for all $\beta>0$ or, equivalently, $\left\{ \beta v \ \middle| \ \beta > 0\right\} \cap \ell(C) = \emptyset$. (For example, one necessarily has $\beta e_i \notin C$ or $-\beta e_i \notin C$ for all $\beta>0$ and at least one $i \in \{1,\dots, p\}$ with $e_i$ denoting the $i$th standard basis vector of $\R^p$.) Since, in addition, both $\ell(C)$ as the image of a convex set under an affine map and $\left\{ \beta v \ \middle| \ \beta>0 \right\}$ are convex, one can apply the separation theorem~\ref{thm: separationTheorem} to obtain the existence of a nontrivial vector $\mu \in \R^p$ with $ \mu^T y \leq \beta \mu^T v $ for all $y \in \ell(C)$ and $\beta>0$. But this implies $\mu^T y \leq 0$ for all $y \in \ell (C)$. 
    Plugging in the definition of $\ell$ and exploiting feasibility of $\widehat{x}$ yields
    \[ 0 \geq \mu^T (Ax-b) = \mu^T A \left(x- \widehat{x}\right) = \left( A^T \mu \right)^T \left(x - \widehat{x} \right) \quad \forall x \in C, \]
    which proves  $A^T \mu \in N_C \left( \widehat{x} \right)$ or, equivalently,
    \[ 0 \in \left\{-A^T \mu \right\} + N_C \left(\widehat{x}\right).\]
    Hence, $(0,-\mu) \in \R^m_+ \times \R^p$ is a nontrivial vector satisfying~\eqref{eq: NNAMCQ}, in contrast to NNAMCQ. 

    Next, consider the case $0 \notin S$ with $S$ defined by~\eqref{eq: proofCQs_setS}. 
    Note that $C$ being convex implies that $T_C \left(\widehat{x}\right)$ is also convex, from which one can easily deduce the convexity of the set $S$. 
    Thus, we can apply the separation theorem~\ref{thm: separationTheorem} to the disjoint sets $\{0\}$ and $S$. It follows that there exists a nontrivial vector $(\mu, \lambda ) \in \R^p  \times \R^{\left\vert I \left( \widehat{x} \right)\right\vert}$ with 
    \[ \mu^Tz+ \lambda^T \alpha \geq 0 \qquad \forall (z,\alpha) \in S, \]
    which gives
    \[ \mu^T Ad + \sum_{i \in I \left( \widehat{x} \right) } \lambda_i \left( c_i'\left( \widehat{x};d \right)+s_i \right) \geq 0 \qquad \forall d \in T_C \left( \widehat{x} \right), s>0. \]
    Since this relation holds for all $s > 0$, we immediately obtain both $\lambda \geq 0$ and
    \[ \mu^T Ad + \sum_{i \in I \left(\widehat{x}\right) } \lambda_i c_i'\left(\widehat{x};d \right) \geq 0 \qquad \forall d \in T_C \left(\widehat{x} \right). \]
    Hence, using~\eqref{eq: Subdiff_dirDerivative}, there exists for each $i \in I \left( \widehat{x}\right)$ some $\omega_i \in \partial c_i \left( \widehat{x}\right)$ such that
    \[ \mu^T Ad + \sum_{i \in I \left( \widehat{x}\right)} \lambda_i \omega_i^T d \geq 0 \qquad \forall d \in T_C \left( \widehat{x} \right). \]
    This proves $- \left( A^T \mu + \sum_{i \in I \left( \widehat{x} \right) } \lambda_i \omega_i \right) \in \left( T_C \left( \widehat{x} \right) \right)^{\circ}$. Exploiting the relationship~\eqref{eq: NormalToTangentCone}, we obtain 
    \[ 0 \in \left\{A^T \mu \right\} + \sum_{i \in I \left( \widehat{x} \right) }\lambda_i \partial c_i \left( \widehat{x} \right) + N_C \left( \widehat{x} \right). \]
    Finally setting $\lambda_i=0$ for all $ i \notin I \left(\widehat{x}\right)$, we see that $\left(\lambda, \mu\right)\in \R^m_+ \times \R^p$ is as nontrivial vector satisfying~\eqref{eq: NNAMCQ}, a contradiction to NNAMCQ.
\end{proof}

\begin{proof}[Proof of \cref{lem: CQs}]
    Let $\widetilde{x} \in \R^n$ denote a Slater point of~\eqref{eq: problemFormulation}. It remains to show that, for an arbitrary feasible point $\widehat{x} \in \R^n$ of~\eqref{eq: problemFormulation}, there exists $d \in T_C \left( \widehat{x} \right)$ with $Ad=0$ and $c_i'\left(\widehat{x};d \right) < 0$ for all $i \in I \left(\widehat{x} \right)$. Hence, for a feasible point $\widehat{x} \in \R^n$, consider $d:= \widetilde{x}-\widehat{x}$. Then, $d \in T_C \left( \widehat{x} \right)$ as can be verified by choosing an arbitrary null sequence $\left\{t_k \right\}_{k \in \N_0} \subseteq (0,1)$, setting $x^k:= \widehat{x}+t_k d$ for $k \in \N_0$ and noting that $\left\{x^k \right\}_{k \in \N_0} \subseteq C$ holds by convexity of $C$, $x^k \rightarrow \widehat{x}$ as $k \rightarrow \infty$ and $\frac{x^k-\widehat{x}}{t_k}=d \rightarrow d$. Moreover, for $i \in I \left( \widehat{x} \right)$, we have
    \[ 0 > c_i \left( \widetilde{x} \right) = c_i \left( \widehat{x}+d \right) \geq c_i \left( \widehat{x} \right) + \omega^T d = \omega^T d \]
    for all $\omega \in \partial c_i \left( \widehat{x} \right)$. Consequently, we get
    \[ 0 > \max_{\omega \in \partial c_i \left( \widehat{x} \right)} \omega^Td = c_i' \left( \widehat{x};d \right) \]
    for all $i \in I \left( \widehat{x} \right)$. Since we clearly have
    \[ Ad = A \widetilde{x} - b - \left( A \widehat{x} -b \right) =0 \]
    by feasibility of $\widetilde{x}$ and $\widehat{x}$, the claim follows.

    Now, suppose mEMFCQ holds at some feasible point of~\eqref{eq: problemFormulation} and mSCQ is violated. Since $0 \in \interior \left( \ell (C) \right)$ by assumption, this means that $0 \notin U$, where $U$ is defined by
    \[ U:= \left\{ (z,\alpha) \in \R^p \times \R^m \ \middle| \ \exists x \in C, s \in \R^m,s>0: \ Ax-b=z, \ c(x)+s=\alpha \right\}.  \]
    A simple calculation shows that $U$ is a convex set. 
    Hence, we can apply the separation theorem~\ref{thm: separationTheorem} to the disjoint sets $\{0\}$ and $U$ to get a nontrivial vector $( \mu,\lambda) \in \R^+ \times \R^m$ satisfying
    \[ \mu^T z + \lambda^T \alpha \geq 0 \qquad \forall (z,\alpha) \in U \]
    or, equivalently,
    \[ \mu^T (Ax-b) + \lambda^T \left(c(x)+s \right) \geq 0 \qquad x \in C, \ s>0. \]
    Since this inequality has to hold for all $s>0$, we obtain $\lambda\geq 0$ and 
    \begin{equation} \label{eq: proofCQs_separationNoSCQ}
        \mu^T (Ax-b) + \lambda^T c(x) \geq 0
    \end{equation}
    for all $ x \in C$.
    Now, consider any feasible point $\widehat{x}$ of~\eqref{eq: problemFormulation}. Then,~\eqref{eq: proofCQs_separationNoSCQ} implies $\lambda^T c \left( \widehat{x}\right) \geq 0$ which immediately shows $\lambda^T c\left( \widehat{x} \right)=0$ due to $\lambda \geq 0$ and $c \left( \widehat{x} \right) \leq 0$. Consequently, each feasible point satisfies~\eqref{eq: proofCQs_separationNoSCQ} with equality and hence, solves
    \[ \min_{x \in C} \left\{ \mu^T (Ax-b) + \lambda^T c(x)\right\}. \]
    Since this is a convex optimization problem, $\widehat{x}$ solving this problem is equivalent to (cf. \cref{thm: cvx_OptCond})
    \[ 0 \in \left\{ A^T \mu \right\} + \sum_{i=1}^m \lambda_i  \partial c_i \left( \widehat{x} \right) + N_C \left( \widehat{x} \right)\]
    where one also exploits the basic subdifferential calculus rules from \cref{thm: cvxSubdiff_calculus}\ref{thm: cvxSubdiff_sumRule} and \ref{thm: cvxSubdiffDifferentiable}. However, this shows that NNAMCQ does not hold at any feasible point of~\eqref{eq: problemFormulation}. Moreover, due to \cref{lem: CQs_stepOne} this is equivalent to mEMFCQ not being satisfied at any feasible point of~\eqref{eq: problemFormulation}, contradicting our assumption.
\end{proof}

\bibliographystyle{jnsao} 
\bibliography{Literatur_sALMDC}

@article{vanAckooij2019,
  author = {W. van Ackooij and W. de Oliveira},
  title = {Non-smooth DC-constrained optimization: constraint qualification and minimizing methodologies},
  journal = {Optimization Methods and Software},
  shortjournal = {Optim. Methods Softw.},
  volume = {34},
  number = {4},
  pages = {890--920},
  year = {2019},
  publisher = {Taylor \& Francis},
  doi = {10.1080/10556788.2019.1595619},
  eprint = {https://doi.org/10.1080/10556788.2019.1595619},
}

@article{vanAckooij2019b,
  author = {W. van Ackooij and W. de Oliveira},
  title = {Addendum to the paper ‘Nonsmooth DC-constrained optimization: constraint qualification and minimizing methodologies’},
  journal = {Optimization Methods and Software},
  shortjournal = {Optim. Methods Softw.},
  volume = {37},
  number = {6},
  pages = {2241--2250},
  year = {2022},
  publisher = {Taylor \& Francis},
  doi = {10.1080/10556788.2022.2063861},
  eprint = {https://doi.org/10.1080/10556788.2022.2063861},
}

@article{AragonArtacho2022,
  title = {The boosted DC algorithm for linearly constrained DC programming},
  author = {Aragón-Artacho, F. J. and Campoy, R. and Vuong, P. T.},
  journal = {Set-Valued and Variational Analysis},
  shortjournal = {Set-Valued Var. Anal.},
  year = {2022},
  volume = {30},
  pages = {1265-1289},
  doi = {10.1007/s11228-022-00656-x},
}

@article{Beck2015,
  title = {Weiszfeld’s method: old and new results},
  author = {Amir Beck and Shoham Sabach},
  journal = {Journal of Optimization Theory and Applications},
  shortjournal = {J. Optim. Theory Appl.},
  year = {2015},
  volume = {164},
  pages = {1-40},
  doi = {10.1007/s10957-014-0586-7},
}

@book{Beck2017,
  author = {Beck, Amir},
  title = {First-Order Methods in Optimization},
  publisher = {Society for Industrial and Applied Mathematics},
  year = {2017},
  doi = {10.1137/1.9781611974997},
  address = {Philadelphia, PA},
  edition = {},
  URL = {https://epubs.siam.org/doi/abs/10.1137/1.9781611974997},
  eprint = {https://epubs.siam.org/doi/pdf/10.1137/1.9781611974997},
}

@article{Birgin2012,
  author = {Ernesto G. Birgin and Damián Fernández and J. M. Martínez},
  title = {The boundedness of penalty parameters in an augmented Lagrangian method with constrained subproblems},
  journal = {Optimization Methods and Software},
  shortjournal = {Optim. Methods Softw.},
  volume = {27},
  number = {6},
  pages = {1001--1024},
  year = {2012},
  publisher = {Taylor \& Francis},
  doi = {10.1080/10556788.2011.556634},
}

@book{Birgin2014,
  title = {Practical Augmented Lagrangian Methods for Constrained Optimization},
  author = {Birgin, E. G. and Martínez, J. M.},
  isbn = {9781611973365},
  series = {Fundamentals of Algorithms},
  year = {2014},
  publisher = {SIAM, Society for Industrial and Applied Mathematics},
}

@article{Chen1998,
  author = {Chen, Pey-Chun and Hansen, Pierre and Jaumard, Brigitte and Tuy, Hoang},
  year = {1998},
  month = {08},
  pages = {548-562},
  title = {Solution of the multisource Weber and conditional Weber problems by D.-C. programming},
  volume = {46},
  journal = {Operations Research},
  shortjournal = {Oper. Res.},
  doi = {10.1287/opre.46.4.548},
}

@book{Clarke1990,
  author = {Clarke, Frank H.},
  title = {Optimization and Nonsmooth Analysis},
  publisher = {Society for Industrial and Applied Mathematics},
  year = {1990},
  doi = {10.1137/1.9781611971309},
  URL = {https://epubs.siam.org/doi/abs/10.1137/1.9781611971309},
  eprint = {https://epubs.siam.org/doi/pdf/10.1137/1.9781611971309},
}

@article{deOliveira2019,
  doi = {10.1007/s10898-019-00755-4},
  year = {2019},
  volume = {75},
  number = {2},
  pages = {523-563},
  author = {de Oliveira, Welington},
  title = {Proximal bundle methods for nonsmooth {DC} programming},
  shortjournal = {J. Glob. Optim.},
  journal = {Journal of Global Optimization},
}

@book{Dhara2011,
  title = {Optimality Conditions in Convex Optimization: A Finite-Dimensional View},
  author = {Dhara, A. and Dutta, J.},
  year = {2012},
  publisher = {Taylor \& Francis},
}

@book{Eilon1971,
  title = {Distribution Management: Mathematical Modelling and Practical Analysis},
  author = {Eilon, S. and Watson-Gandy, C.D.T. and Christofides, N.},
  isbn = {9780852641910},
  lccn = {72870126},
  year = {1971},
  publisher = {Griffin},
}

@article{Gotoh2018,
  author = {Gotoh, Jun-ya and Takeda, Akiko and Tono, Katsuya},
  year = {2018},
  month = {05},
  pages = {141-176},
  title = {DC formulations and algorithms for sparse optimization problems},
  volume = {169},
  journal = {Mathematical Programming},
  shortjournal = {Math. Program.},
  doi = {10.1007/s10107-017-1181-0},
}

@article{Holmberg1999,
  author = {Holmberg, Kaj and Tuy, Hoang},
  year = {1999},
  month = {05},
  pages = {157-179},
  title = {A production-transportation problem with stochastic demand and concave production costs},
  volume = {85},
  journal = {Mathematical Programming},
  shortjournal = {Math. Program.},
  doi = {10.1007/s101070050050},
}

@article{Horst1999,
  author = {R. Horst and N. V. Thoai},
  title = {{DC programming: overview}},
  journal = {Journal of Optimization Theory and Applications},
  shortjournal = {J. Optim. Theory Appl.},
  year = {1999},
  volume = {103},
  number = {1},
  pages = {1-43},
  month = {October},
  doi = {10.1023/A:1021765131316},
}

@book{Kanzow2002,
  title = {Theorie und Numerik restringierter Optimierungsaufgaben},
  author = {Geiger, C. and Kanzow, C.},
  series = {Springer-Lehrbuch Masterclass},
  year = {2002},
  publisher = {Springer Berlin Heidelberg},
}

@article{Kanzow2017,
  title = {An example comparing the standard and safeguarded augmented Lagrangian methods},
  journal = {Operations Research Letters},
  shortjournal = {Oper. Res. Lett.},
  volume = {45},
  number = {6},
  pages = {598-603},
  year = {2017},
  issn = {0167-6377},
  doi = {10.1016/j.orl.2017.09.005},
  author = {Christian Kanzow and Daniel Steck},
}

@article{Kiwiel1983,
  author = {Kiwiel, Krzysztof Czesław},
  year = {1983},
  title = {An aggregate subgradient method for nonsmooth convex minimization},
  journal = {Mathematical Programming},
  shortjournal = {Math. Program.},
  pages = {320 -- 341},
  volume = {27},
  doi = {10.1007/BF02591907},
}

@book{Lemarechal2001,
  title = {Fundamentals of Convex Analysis},
  author = {Hiriart-Urruty, Jean-Baptiste and Lemaréchal, Claude},
  year = {2001},
  publisher = {Springer Berlin},
  series = {Grund\-lehren Text Editions},
}

@book{Lemarechal2014,
  title = {Nonsmooth Optimization: Proceedings of a IIASA Workshop, March 28 - April 8, 1977},
  author = {Lemaréchal, C. and Mifflin, R.},
  year = {2014},
  publisher = {Elsevier Science},
}

@article{LeThi1996,
  title = {Numerical solution for optimization over the efficient set by d.c. optimization algorithms},
  shortjournal = {Oper. Res. Lett.},
  journal = {Operations Research Letters},
  volume = {19},
  number = {3},
  pages = {117-128},
  year = {1996},
  doi = {10.1016/0167-6377(96)00022-3},
  author = {Le Thi, Hoai An and Pham Dinh, Tao and Le Dung, Muu},
}

@article{LeThi2008,
  author = {Le Thi, Hoai An and Hoai Minh, Le and Nguyen, Van and Dinh, Tao},
  year = {2008},
  month = {02},
  pages = {259-278},
  title = {A DC programming approach for feature selection in support vector machines learning},
  volume = {2},
  journal = {Advances in Data Analysis and Classification},
  shortjournal = {Adv. Data Anal. Classif.},
  doi = {10.1007/s11634-008-0030-7},
}

@inproceedings{LeThi2014,
  author = {Le Thi, Hoai An and Huynh, Van Ngai and Pham Dinh, Tao},
  editor = {van Do, Tien and Le Thi, Hoai An and Nguyen, Ngoc Thanh},
  year = {2014},
  publisher = {Springer International Publishing},
  address = {Cham},
  pages = {15-35},
  title = {DC programming and DCA for general DC programs},
  booktitle = {Advanced Computational Methods for Knowledge Engineering},
  booksubtitle = {Advances in Intelligent Systems and Computing},
  volume = {282},
  doi = {10.1007/978-3-319-06569-4_2},
}

@article{LeThi2018,
  author = {Le Thi, Hoai An and Pham Dinh, Tao},
  shortjournal = {Math. Program.},
  journal = {Mathematical Programming},
  title = {{DC} programming and {DCA}: thirty years of developments},
  year = {2018},
  volume = {169},
  number = {1},
  pages = {5-68},
  doi = {10.1007/s10107-018-1235-y},
}

@article{Li1997,
  title = {Lagrangian multipliers, saddle points, and duality in vector optimization of set-valued maps},
  journal = {Journal of Mathematical Analysis and Applications},
  shortjournal = {J. Math. Anal. Appl.},
  volume = {215},
  number = {2},
  pages = {297-316},
  year = {1997},
  issn = {0022-247X},
  doi = {10.1006/jmaa.1997.5568},
  author = {Zhong-Fei Li and Guang-Ya Chen},
}

@article{Lu2022,
  author = {Lu, Zhaosong and Sun, Zhe and Zhou, Zirui},
  year = {2022},
  month = {01},
  pages = {2260-2285},
  title = {Penalty and augmented Lagrangian methods for constrained DC programming},
  doi = {10.1287/moor.2021.1207},
  journal = {Mathematics of Operations Research},
  shortjournal = {Math. Oper. Res.},
  volume = {47},
  issue = {3},
}

@book{Mordukhovich2013,
  author = {Mordukhovich, Boris S. and Nam, Nguyen Mau},
  title = {An Easy Path to Convex Analysis and Applications},
  year = {2013},
  isbn = {1627052372},
  publisher = {Morgan \& Claypool Publishers},
  edition = {1st},
}

@article{Pang2017,
  Author = {Pang, Jong-Shi and Razaviyayn, Meisam and Alvarado, Alberth},
  ISSN = {0364765X},
  Journal = {Mathematics of Operations Research},
  shortjournal = {Math. Oper. Res.},
  Number = {1},
  Pages = {95 - 118},
  Title = {Computing B-stationary points of nonsmooth DC programs.},
  Volume = {42},
  Year = {2017},
  doi = {10.1287/moor.2016.0795},
}

@article{Pang2018,
  author = {Pang, Jong-Shi and Tao, Min},
  title = {Decomposition methods for computing directional stationary solutions of a class of nonsmooth nonconvex optimization problems},
  journal = {SIAM Journal on Optimization},
  shortjournal = {SIAM J. Optim.},
  volume = {28},
  number = {2},
  pages = {1640-1669},
  year = {2018},
  doi = {10.1137/17M1110249},
  eprint = {https://doi.org/10.1137/17M1110249},
}

@book{Pang2021,
  author = {Cui, Ying and Pang, Jong-Shi},
  title = {Modern Nonconvex Nondifferentiable Optimization},
  publisher = {Society for Industrial and Applied Mathematics},
  year = {2021},
  doi = {10.1137/1.9781611976748},
  address = {Philadelphia, PA},
  edition = {},
  URL = {https://epubs.siam.org/doi/abs/10.1137/1.9781611976748},
  eprint = {https://epubs.siam.org/doi/pdf/10.1137/1.9781611976748},
}

@incollection{PhamDinh2014,
  author = {Pham Dinh, Tao and Le Thi, Hoai An},
  editor = {Nguyen, Ngoc-Thanh and Le Thi, Hoai An},
  title = {Recent advances in DC programming and DCA},
  bookTitle = {Transactions on Computational Intelligence XIII},
  year = {2014},
  publisher = {Springer Berlin Heidelberg},
  address = {Berlin, Heidelberg},
  pages = {1--37},
  doi = {10.1007/978-3-642-54455-2_1},
}

@book{Rockafellar1970,
  title = { Convex Analysis},
  author = { Rockafellar, Ralph Tyrrell },
  year = {1970},
  publisher = {Princeston University Press},
  series = {Princeton Mathematical Series},
  volume = {28},
}

@book{Schirotzek2007,
  title = {Nonsmooth Analysis},
  author = {Schirotzek, W.},
  isbn = {9783540713326},
  lccn = {2007922937},
  series = {Universitext},
  year = {2007},
  publisher = {Springer Berlin Heidelberg},
}

@article{Souza2015,
  title = {Global convergence of a proximal linearized algorithm for difference of convex functions},
  author = {Jo{\~a}o Carlos O. Souza and P. Roberto Oliveira and Antoine Soubeyran},
  journal = {Optimization Letters},
  shortjournal = {Optim. Lett.},
  year = {2015},
  volume = {10},
  pages = {1529 - 1539},
  doi = {10.1007/s11590-015-0969-1},
}

@article{Stein2004,
  Author = {Stein, O.},
  Journal = {Journal of Optimization Theory and Applications},
  shortjournal = {J. Optim. Theory Appl.},
  Pages = {647 - 671},
  Title = {On constraint qualifications in nonsmooth optimization},
  Volume = {121},
  Year = {2004},
  doi = {10.1023/B:JOTA.0000037607.48762.45},
}

@article{Sun2022,
  author = {Kaizhao Sun, Xu Andy Sun },
  journal = {INFORMS Journal on Optimization},
  shortjournal = {INFORMS J. Optim.},
  pages = {321-339},
  title = {Algorithms for difference-of-convex programs based on difference-of-Moreau-envelopes smoothing},
  volume = {5},
  year = {2022},
  doi = {10.1287/ijoo.2022.0087},
}

@book{Tuy2016,
  title = {Convex Analysis and Global Optimization},
  author = {Tuy, H.},
  isbn = {9783319314822},
  series = {Springer Optimization and Its Applications},
  url = {https://books.google.de/books?id=1qJQjwEACAAJ},
  year = {2016},
  publisher = {Springer International Publishing},
}

@article{Vardi2001,
  author = {Vardi, Yehuda and Zhang, Cun-Hui},
  year = {2001},
  month = {01},
  pages = {559-566},
  title = {A modified Weiszfeld algorithm for the Fermat-Weber location problem},
  volume = {90},
  journal = {Mathematical Programming, Series B},
  shortjournal = {Math. Program., Series B},
  doi = {10.1007/PL00011435},
}

@article{Watson1992,
  author = {Watson, G. A.},
  title = {Linear best approximation using a class of polyhedral norms},
  year = {1992},
  issue_date = {October 1992},
  publisher = {Springer-Verlag},
  address = {Berlin, Heidelberg},
  volume = {2},
  number = {3},
  issn = {1017-1398},
  doi = {10.1007/BF02139472},
  shortjournal = {Numer. Algorithms},
  journal = {Numerical Algorithms},
  month = oct,
  pages = {321–335},
  numpages = {15},
}

@article{Weiszfeld2009,
  author = {Weiszfeld, E. and Plastria, Frank},
  year = {2009},
  month = {03},
  pages = {7-41},
  title = {On the point for which the sum of the distances to $n$ given points is minimum},
  volume = {167},
  journal = {Annals of Operations Research},
  shortjournal = {Ann. Oper. Res.},
  doi = {10.1007/s10479-008-0352-z},
}

@article{Yin2015,
  author = {Yin, Penghang and Lou, Yifei and He, Qi and Xin, Jack},
  title = {Minimization of $\ell_{1-2}$ for compressed sensing},
  journal = {SIAM Journal on Scientific Computing},
  shortjournal = {SIAM J. Sci. Comput.},
  volume = {37},
  number = {1},
  pages = {A536-A563},
  year = {2015},
  doi = {10.1137/140952363},
  eprint = {https://doi.org/10.1137/140952363},
}

@article{Zeng2023,
  AUTHOR = {Zeng, Renying},
  TITLE = {Constraint qualifications for vector optimization problems in real topological spaces},
  JOURNAL = {Axioms},
  shortjournal = {Axioms},
  VOLUME = {12},
  YEAR = {2023},
  NUMBER = {8},
  ARTICLE-NUMBER = {783},
  ISSN = {2075-1680},
  DOI = {10.3390/axioms12080783},
}

@phdthesis{Zhang2014,
  title = {Enhanced Optimality Conditions and New Constraint Qualifications for Nonsmooth Optimization Problems},
  author = {Jin Zhang},
  school = {University of Victoria},
  year = {2014},
  note = {\url{https://api.semanticscholar.org/CorpusID:117842585}},
}

\end{document}